\documentclass[graybox,envcountsame,envcountchap]{LMsvmult}

\usepackage{subfig}
\usepackage{color}

\usepackage{mathptmx}       
\usepackage{helvet}         
\usepackage{courier}        
\usepackage{makeidx}         
\usepackage{graphicx}        
\usepackage{multicol}        
\usepackage[bottom]{footmisc}
\usepackage{url}
\usepackage{latexsym}
\usepackage{amssymb,amsmath}
\usepackage{amsfonts}

    \def\re{\textnormal {Re}}
    \def\im{\textnormal {Im}}
    \def\p{{\mathbb P}}
    \def\e{{\mathbb E}}
    \def\r{{\mathbb R}}
    \def\c{{\mathbb C}}
    	
    \def\d{{\textnormal d}}
    \def\i{{\textnormal i}}
     
\begin{document}

\begingroup

\title*{The theory of scale functions for spectrally\\ negative L\'evy processes}
\author{Alexey Kuznetsov,  Andreas E. Kyprianou and Victor Rivero}
\institute{
Alexey Kuznetsov \at Department of Mathematics and Statistics
York University
4700 Keele Street
Toronto, Ontario
Canada, M3J 1P3 \email{kuznetsov@mathstat.yorku.ca}; homepage: \url{http://www.math.yorku.ca/~akuznets/}
\and
Andreas E. Kyprianou \at Department of Mathematical Sciences, University of Bath, Claverton Down, Bath, BA1 2UU. 
\email{a.kyprianou@bath.ac.uk}; homepage: \url{http://www.maths.bath.ac.uk/~ak257/}
\and Victor Rivero \at Centro de Investigaci{\'o}n en Matem{\'a}ticas A.C.,
Calle Jalisco s/n,
C.P. 36240 Guanajuato, Gto. Mexico \email{rivero@cimat.mx}; homepage: \url{http://www.cimat.mx/~rivero/} 
}


\maketitle

\numberwithin{equation}{chapter} \smartqed
\numberwithin{section}{chapter}

\abstract{The purpose of this review article is to give an up to date account of the theory and application of scale functions for spectrally negative L\'evy processes. Our review also includes the first extensive overview of how to work numerically with scale functions.
Aside from being well acquainted with the general theory of probability, the reader is  assumed to have some elementary knowledge of L\'evy processes, in particular a reasonable understanding of the L\'evy-Khintchine formula and its relationship to the L\'evy-It\^o decomposition. We shall also touch on more general topics such as excursion theory and semi-martingale calculus. However, wherever possible, we shall try to focus on key ideas taking a selective stance on the technical details. For the reader who is less familiar with some of the mathematical theories and techniques which are used at various points in this review, we note that all the necessary technical background can be found in the following texts on L\'evy processes; Bertoin (1996), Sato (1999), Applebaum (2004), 
 Kyprianou (2006) and Doney (2007).
 \\
\noindent\textbf{AMS Subject Classification 2000:}\\
Primary: 60G40, 60G51, 60J45                      \\
Secondary:  91B70
\keywords{Scale functions, Spectrally negative L\'evy processes, excursion theory, fluctuation theory, first passage problem, applied probability, Laplace transform. 
} 
}

\chapter{Motivation}\label{motivation}

\section{Spectrally negative L\'evy processes}

Let us begin with a brief overview of what is to be understood by a spectrally negative L\'evy process.
Suppose that $(\Omega,\mathcal{F},\mathbb{F},P)$ is a filtered
probability space with filtration $\mathbb{F}=\{\mathcal{F}_{t}:t\geq
0\}$ which is naturally enlarged (cf.  Definition 1.3.38 of \cite{Bichteler}). A L\'{e}vy process on this space is  
a strong Markov,
$\mathbb{F}$-adapted process $X=\{X_{t}:t\geq 0\}$ with c\`adl\`ag paths having the properties that $P(X_0=0)=1$ and for each
$0\leq s\leq t$, the increment $X_{t}-X_{s}$ is independent of
$\mathcal{F}_{s}$ and has the same distribution as $X_{t-s}$. In this
sense, it is said that a L\'{e}vy process has stationary independent
increments.

On account of the fact that the process has stationary and independent
increments, it is not too difficult to show that there is a function $\Psi:\mathbb{R}\mapsto \mathbb{C},$ such that
\begin{equation*}
	E\bigl(\E^{\I\theta X_{t}}\bigr) =\E^{-t\Psi(\theta )}, \qquad t\geq 0, \theta\in\mathbb{R},
\end{equation*}
where $E$ denotes expectation with respect to $P$.
 The
L\'{e}vy--Khintchine formula gives the general form of the function
$\Psi$. That is,
\begin{equation}
	\Psi(\theta)=\I\mu \theta +\frac{\sigma^{2}}{2}\,\theta^{2}
	+\int_{(-\infty,\infty)}\bigl(1-\E^{\I\theta x}+\I\theta
	x\,\mathbf{1}_{|x|<1}\bigr)\,\Pi (\D x),
\label{kyprianou-palmowski:LH}
\end{equation}
for every $\theta \in \mathbb{R},$ where $\mu \in \mathbb{R}$, $\sigma
\in\mathbb{R}$ and $\Pi $ is a measure on $\mathbb{R}\backslash \{0\}$ such
that $\int (1\wedge x^{2})\,\Pi (\D x)<\infty $. Often we wish to specify the law of $X$ when issued from $x\in \mathbb{R}$. In that case we write $P_x$, still reserving the notation $P$ for the special case $P_0$. Note in particular that $X$ under $P_x$ has the same law as $x+X$ under $P$. The notation $E_x$ will be used for the obvious expectation operator.

We say that $X$ is {\it spectrally negative} if the measure $\Pi$ is carried by $(-\infty,0),$ that is $\Pi(0,\infty)=0$.  We exclude from the discussion
however the case of monotone paths. Under the present circumstances that means we are excluding
the case of a descending subordinator and the case of a pure positive linear drift. Included in
the discussion however are descending subordinators plus a positive
linear drift  and a Brownian motion with drift. Also included are processes
such as asymmetric $\alpha $-stable processes for $\alpha \in (1,2)$
which have unbounded variation and zero quadratic variation. 
These processes are  by no means representative of the true variety of processes that populate the class of spectrally negative L\'evy processes. 
Indeed in the forthcoming text we shall see many different explicit examples of spectrally negative L\'evy processes.
Moreover, by
adding independent copies of any of the aforementioned, and/or other, spectrally negative
processes together, the resulting process remains within the class of spectrally negative L\'{e}vy
process. Notationally we say  that a L\'evy process $X$ is {\it spectrally positive} when $-X$ is spectrally negative.

Thanks to the fact that there are no positive jumps, it is possible to talk of
the Laplace exponent $\psi(\lambda)$ for a spectrally negative L\'evy process, defined by
\begin{equation}
	E\bigl(\E^{\lambda X_{t}}\bigr)
	=\E^{\psi(\lambda) t},
\label{kyprianou-palmowski:Laplace-functional}
\end{equation}
for at least $\lambda\geq 0$.  In other words, taking into account the analytical extension of the characteristic exponent, we have  $\psi(\lambda) =-\Psi(-\I\lambda)$. 
In particular 
\begin{equation}
\psi(\lambda)  = -\mu \lambda +\frac{\sigma^{2}}{2}\,\lambda^{2}
	+\int_{(-\infty,0)}\bigl(\E^{\lambda x}-1-\lambda
	x\,\mathbf{1}_{|x|<1}\bigr)\,\Pi (\D x).
	\label{laplaceLK}
\end{equation}
 Further, using H\"older's inequality and the fact that $\int_{(-\infty,0)}(1\wedge x^2)\Pi(\D x)<\infty$, it is easy to check that
$\psi$ is strictly convex and tends to infinity as $\lambda$ tends to
infinity. This allows us to define for $q\geq 0$,
\begin{equation*}
	\Phi(q)=\sup \{\lambda \geq 0:\psi(\lambda) = q\},
\end{equation*}
the largest root of the equation $\psi(\lambda) =q$. Note that there exist  two roots when $q=0$ (there is  always a root at zero on account of the fact that $\psi(0)=0$) and
precisely one root when $q>0$. Further, from a straightforward differentiation of (\ref{kyprianou-palmowski:Laplace-functional}), we can identify $\psi'(0^+)
=E(X_{1}) \in [-\infty,\infty)$ which determines the long term behaviour of the process.
Indeed when $\pm\psi'(0+)>0$ we have $\lim_{t\uparrow\infty}X_t =\pm\infty$, that is to say that the process drifts towards $\pm\infty$,  and when $\psi'(0+) = 0$ then $\limsup_{t\uparrow\infty} X_t = - \liminf_{t\uparrow\infty} X_t = \infty$, in other words the process oscillates.

\section{Scale functions and applied probability}

Let us now turn our attention to the definition of  {\it scale functions} and motivate the need for a theory thereof.

\begin{definition}\label{scaledef}\rm 
For a given spectrally negative L\'evy process, $X,$ with Laplace exponent $\psi$, we define a family of functions indexed by $q\geq 0$,
 $W^{(q)}:\mathbb{R}\rightarrow [0,\infty),$ as follows. For each given $q\geq 0,$ we have
 $W^{(q)}(x)=0$ when $x<0,$ and otherwise on $[0,\infty)$, $W^{(q)}$ is the unique right continuous function whose Laplace transform is given by 
 \begin{equation}
 \int_0^\infty \E^{-\beta x}W^{(q)}(x)\D x = \frac{1}{\psi(\beta)  - q}
 \label{LTdefW}
 \end{equation}
 for $\beta>\Phi(q)$. For convenience we shall always write $W$ in place of $W^{(0)}$. Typically we shall refer the functions $W^{(q)}$ as {\it $q$-scale functions}, but we shall also refer to  $W$ as just the {\it scale function}.
\end{definition}
 
\noindent From the above definition alone, it is not clear why one would be interested in such functions. Later on we shall see that scale functions appear in the vast majority of known identities concerning boundary crossing problems and related path decompositions. This in turn has consequences for their use in a number of classical applied probability models which rely heavily on such identities. To give  but one immediate example of a boundary crossing identity which necessitates Definition \ref{scaledef}, consider the following result, the so-called {\it two-sided exit problem}, which has a long history; see for example \cite{zol64, tak67, rog90, ber96,  ber97}. 
\begin{theorem}\label{twosided-up}Define 
\[
\tau^+_a  = \inf\{t>0 : X_t >a\}\text{ and }\tau^-_0 = \inf\{t> 0 : X_t <0\}.
\]
For all $q\geq 0,$ and $x<a$, 
\begin{equation}
E_x(\E^{-q\tau^+_a}\mathbf{1}_{\{\tau^+_a<\tau^-_0\}})  = \frac{W^{(q)}(x)}{W^{(q)}(a)}.
\label{whatsinaname}
\end{equation}
\end{theorem}

In fact, it is through this identity that the `scale function' gets its name. Possibly the first reference to this terminology is found in Bertoin \cite{bert-scale}. Indeed (\ref{whatsinaname}) has some mathematical similarities with an analogous identity which holds for a large class of one-dimensional diffusions and which involves so-called scale functions for diffusions; see for example Proposition VII, 3.2 of Revuz and Yor \cite{RY}. In older Soviet-Ukranian literature $W^{(q)}$ is simply referred to as a {\it resolvent}, see for example \cite{Korolyuk}.

In the forthcoming chapters we shall explore in more detail the formal analytical properties of scale functions as well as look at explicit examples and methodology for working numerically with scale functions. However let us first conclude this motivational chapter by giving some definitive examples of how scale functions  make  a non-trivial contribution to a variety of applied probability models.

\bigskip

\noindent {\bf Optimal stopping:}  Suppose that $x,q>0$ and consider the optimal stopping problem 
\[
v(x) = \sup_{\tau} E(\E^{-q\tau + (\overline{X}_\tau \vee x)}),
\]
where the supremum is taken over all, almost surely finite stopping times with respect to $X$ and $\overline{X}_t = \sup_{s\leq t}X_s$. The solution to this problem entails finding an expression for the value function $v(x)$ and, if the supremum is attained by some stopping time $\tau^*$, describing this optimal stopping time quantitatively.

This particular optimal stopping problem was conceived in connection with the pricing and hedging of certain exotic financial derivatives known as {\it Russian options} by Shepp and Shiryaev \cite{sh1, sh2}. Originally formulated in the Black-Scholes setting, i.e. the case that $X$ is a linear Brownian motion, it is natural to look at the solution of this problem in the case that $X$ is replaced by  a L\'evy process. This is of particular pertinence on account of the fact that  modern theories of mathematical finance accommodate for, and even prefer, such a setting. As a first step in this direction, Avram et al. \cite{kyprianou-palmowski:1} considered the case that $X$ is a spectrally negative L\'evy process. These authors found that scale functions played a very natural role in describing the solution $(v,\tau^*)$.  Indeed, it was shown under very mild additional conditions that 
\[
v(x) = \E^x \left(1 + q \int_0^ {x^* - x} W^{(q)}(y) \D y\right)
\]
where, thanks to properties of  scale functions, it is known that the constant  $x^*\in[0,\infty)$ necessarily satisfies
\[
x^* = \inf\{ x\geq 0 : 1+ q \int_0^x W^{(q)}(y )\D y - qW^{(q)}(x)\leq 0\}.
\]
In particular an optimal stopping time can be taken as
\[
\tau^* = \inf\{t>0 : (x\vee\overline{X}_t)  - X_t > x^*\},
\]
which agrees with the original findings of Shepp and Shiryaev \cite{sh1, sh2} in the diffusive case.

\bigskip

\noindent {\bf Ruin theory:} One arm of actuarial mathematics concerns itself with modelling the event of ruin of an insurance firm. 
The surplus wealth of the insurance firm is often modelled as the difference of a linear growth, representing the deterministic collection of premiums, and a compound Poisson subordinator, representing a random sequence of i.i.d. claims. It is moreover quite often assumed that the claim sizes are exponentially distributed. Referred to as  Cram\'er-Lundberg processes, such  models for the surplus belong to the class of spectrally negative L\'evy processes. 

Classical ruin theory concerns the distribution of the, obviously named, {\it  time to ruin}, 
\begin{equation}
\tau^-_0 : = \inf\{t>0: X_t <0\},
\label{firstpassagetime}
\end{equation}
where $X$ is the Cram\'er-Lundberg process. Recent work has focused on more complex additional distributional features of ruin such as the overshoot and undershoot of the surplus at ruin. These quantities correspond to deficit at ruin and the wealth prior to ruin respectively. A large body of literature exists in this direction with many contributions citing as a point of reference the paper  Gerber and Shiu \cite{GS}. For this reason, the study of the joint distribution of the ruin time as well as the overshoot and undershoot at ruin is often referred to as Gerber-Shiu theory. 

It turns out that scale functions again provide a natural tool with which one may give a complete characterisation of the ruin problem in the spirit of Gerber-Shiu theory. Moreover, this can be done for a general spectrally negative L\'evy process rather than just the special case of a Cram\'er-Lundberg process. Having said that, scale functions already serve their purpose in the classical Cram\'er-Lundberg setting. A good case in point is the following result (cf. \cite{biffis-kyprianou}) which even goes a little further than what classical Gerber-Shiu theory demands in that it also incorporates distributional information on the infimum of the surplus prior to ruin. 

\begin{theorem}\label{triple-law}Suppose that $X$ is a spectrally negative L\'evy processes. For $t\geq 0$ define  $\underline{X}_t = \inf_{s\leq t}X_s$.  Let $q,x\geq 0$, $u, v, y\geq 0,$ then
\begin{equation*}
\begin{split}
 &{E_x(\E^{-q\tau^{-}_{0}} ;-X_{\tau^-_0}\in \D u ,\, X_{\tau^-_0- }\in \D v, \, \underline{X}_{\tau^-_0 -}\in \D y)}\\
&=1_{\{0<y<v\wedge x, u>0\}}\E^{-\Phi(q)(v-y)} \{W^{(q)\prime}(x-y) - \Phi(q)W^{(q)}(x-y)\}\Pi(-\D u-v)\D y \D v \\
&\qquad+\frac{\sigma^{2}}{2}\left(W^{(q)\prime}(x)-\Phi(q)W^{(q)}(x)\right)\delta_{(0,0,0)}(\D u,\D v,\D y),
\end{split}
\end{equation*}
where $\delta_{(0,0,0)}(\D u, \D v, \D y)$ is the Dirac delta measure.
\end{theorem}
Note that the second term on the right hand side corresponds to the event that ruin occurs by {\it creeping}. That is to say the process $X$ enters $(-\infty,0)$ continuously for the first time.

The use of scale functions has somewhat changed the landscape of ruin theory, offering access to a number of general results. See for example \cite{RZ, ARZ, KZ, Kuz_Morales}, as well as the work mentioned in the next example.

\bigskip

\noindent {\bf Optimal control:}  Suppose, as usual, that $X$ is a spectrally negative L\'evy process and consider the following optimal control problem
\begin{equation}
u(x) = \sup_{\pi}\mathbb{E}_x\left(\int_0^{\sigma^\pi} \E^{-qt }\D L^\pi_ t\right),
\label{defin}
\end{equation}
where $x,q>0$, $\sigma^\pi = \inf\{ t>0 :  X_t - L^\pi_t <0\}$ and  $L^\pi =\{L^\pi_t : t\geq 0\}$ is the control process associated with the strategy $\pi$. The supremum is taken over all strategies $\pi$ such  that $L^\pi$ is a non-decreasing, left-continuous adapted process which starts at zero and which does not allow the controlled process $X-L^\pi$ to enter $(-\infty,0)$ as a consequence of one of its jumps. 
The solution to this control problem entails finding an expression for $u(x)$ and, if the supremum is attained by some strategy $\pi^*$, describing this optimal strategy.

This
exemplary control problem originates  from the work of de Finetti  \cite{deF}, again in the context of actuarial science. In the actuarial setting,   $X$ plays the role of a surplus process as usual and $L^\pi$ is to be understood as a dividend payment. The objective then becomes to optimize the mean net present value of dividends paid until ruin  of the aggregate surplus process occurs.

A very elegant solution to this problem has been obtained by Loeffen \cite{Loeffen2008}, following the work of Avram et al. \cite{APP}, by imposing some additional assumptions on the underlying L\'evy process, $X$. His solution is intricately connected to certain analytical properties of the associated scale function $W^{(q)}$. By requiring that, for $x>0$,  $-\Pi(-\infty, -x)$ has a completely monotone density (recall that $\Pi$ is the L\'evy measure of $X$), it turns out that for each $q>0$, the associated $q$-scale function, $W^{(q)}$, has a first derivative
which is strictly convex on $(0,\infty)$. Moreover, the point $a^* : = \inf\{a \geq 0 : W^{(q)\prime} (a)\leq
W^{(q)\prime}(x) \text{ for all }x\geq 0\}$ provides a threshold from which an optimal strategy to (\ref{defin}) can be characterized; namely the
 strategy 
$L^*_t = (a^*\vee \overline{X}_t)- a^*, \, \, t\geq 0$.
The aggregate process,
\[
a^* +X_t  - (a^*\vee \overline{X}_t), \,\, t\geq 0,
\]
 thus has the dynamics of the underlying L\'evy process
reflected at the barrier $a^*$. The value function of this strategy is given by 
 \begin{equation}
 u(x) = \left\{
 \begin{array}{ll}
\frac{W^{(q)} (x)}{W^{(q)\prime}(a)} & \text{ for }0\leq x\leq a \\
x-a + \frac{W^{(q)} (a)}{W^{(q)\prime}(a)} & \text{ for }x> a.
 \end{array}
 \right.
 \end{equation}
The methodology that lead to the above solution has also been shown to have applicability in other related control problems. See for example \cite{Loeffen2009a, Loeffen2009b, KRS, LR, KLP}.

\bigskip

\noindent{\bf Queuing theory:} The classical M/G/1 queue is described as follows. Customers
arrive at a service desk according to a Poisson process with rate $\lambda>0$ and join a
queue. Customers have service times that are independent and
identically distributed with common distribution $F$ concentrated on $(0,\infty)$. Once served, they leave the queue.

The workload, $Y^{(x)}_{t},$ at each time $t\geq 0$, is defined to be
the time it will take a customer who joins the back of the queue
at that moment to reach the service desk when the workload at time zero is equal to  $x\geq 0$. That is to say, $Y^{(x)}_t$ is the
amount of processing time remaining in the queue at time $t$ when there was an initial workload $x$. If $\{X_t:t \geq 0\}$ is the difference of  a pure linear drift of unit rate and a compound Poisson process, with  rate $\lambda$ and jump distribution $F$, then it is not difficult to show that when the current workload at time $t =0$ is $x\geq 0$, then 
\[
Y^{(x)}_t = (x\vee \overline{X}_t) - X_t, \,\, t\geq 0.
\]
Said another way, the workload has the dynamics of a spectrally positive L\'evy process reflected at the origin.

Not surprisingly, many questions concerning the workload of the classical M/G/1 queue therefore boil down to issues regarding fluctuation theory of spectrally negative L\'evy processes. A simple example concerns the issue of buffer overflow adjustment as follows. Suppose that there is a limited workload capacity in the queue, say $B>0$. When the workload exceeds the buffer level $B$, the excess over level $B$ is transferred elsewhere so that the adjusted workload process, say $\{Y^{(x, B)}_t : t\geq 0\}$, never exceeds the amount $B$. We are interested in the busy period for this M/G/1 queue with buffer overflow adjustment. That is to say,
\[
\sigma_B^{(x)} : = \inf\{t>0 : Y^{(x, B)}_t =0\}.
\]
Once again, scale functions prove to be an instrumental tool. For the case at hand, it is known  that for all $q\geq 0$,
\[
E(\E^{-q\sigma^{(x)}_B})  = \frac{1+ q\int_0^{x-B} W^{(q)}(y)\D y}{1 + q\int_0^B W^{(q)}(y)\D y}.
\]
See Pistorius \cite{pistorius2003} and Dube et al. \cite{Dube} for more computations and applications in this vein.

\bigskip

\noindent{\bf Continuous state branching processes:} A $[0,\infty]$-valued strong Markov process $Y=\{Y_t : t\geq 0\}$ with probabilities $\{{P}_x : x\geq 0\}$ is called a continuous-state branching process if it has paths that are right continuous with left limits and its law observes the following property. For any $t\geq 0$ and $y_1, y_2\geq 0$,  $Y_t$
under $P_{y_1 + y_2}$ is equal in law to the independent sum
$Y^{(1)}_t + Y^{(2)}_t$ where the distribution of $Y^{(i)}_t$ is
equal to that of $Y_t$ under $P_{y_i}$ for $i=1,2$.

It turns out that there is a one-to-one mapping between continuous state branching processes whose paths are not monotone and spectrally positive L\'evy processes. Indeed, every continuous-state branching process $Y$ may be written in the form
\begin{equation}
Y_t = X_{\theta_t \wedge \tau^-_0},\,\, t\geq 0,
\label{lamp-transf}
\end{equation}
where $X$ is a spectrally positive L\'evy process,  $\tau^-_0 = \inf\{t>0 : X_t <0\}$ and
\[
\theta_t = \inf\{s>0 : \int_0^{s} \frac{{\D}u}{X_u}>t\}.
\]
Conversely, for any given spectrally positive L\'evy process $X$, the transformation on the right hand side of (\ref{lamp-transf}) defines a continuous-state branching process. (In fact the same bijection characterises all the continuous-state branching processes with monotone non-decreasing paths when $X$ is replaced by a subordinator).

A classic result due to Bingham \cite{bingham} gives (under very mild conditions) the law of the maximum of the continuous-state branching process with non-monotone paths as follows. For all $x\geq y>0$,
\[
P_y(\sup_{s\geq 0} Y_s \leq x) = \frac{W(x-y)}{W(x)},
\]
where $W$ is the scale function associated with the underlying spectrally negative L\'evy process in the representation (\ref{lamp-transf}). See Caballero et al. \cite{proofsof} and Kyprianou and Pardo \cite{KP} for further computations in this spirit and Lambert \cite{Lam2010} for further applications in population biology.

\bigskip

\noindent{\bf Fragmentation processes:}
A homogenous mass fragmentation process is a Markov process  $\mathbf{X}:=\{\mathbf{X}(t) : t\geq 0\}$, where $\mathbf{X}(t) =(X_1(t), X_2(t), \cdots)$, that  takes values in 
\[
\mathcal{S}: = \left\{\mathbf{s}=(s_1, s_2, \cdots) : s_1\geq  s_2 \geq \cdots \geq 0 , \, \sum_{i=1}^\infty s_i \leq 1\right\}.
\]
Moreover ${\bf X}$ possesses the fragmentation property as follows. Suppose that for each $\mathbf{s}\in\mathcal{S}$, $\mathbb{P}_{\bf s}$ denotes the law of $\mathbf{X}$ with $\mathbf{X}(0) = \mathbf{s}$. 
Given, for $t\geq 0$, that $\mathbf{X}(t)  = (s_1,s_2,\cdots)\in\mathcal{S}$, then  for all $u>0$, $\mathbf{X}(t+u)$ has the same law as the variable obtained by ranking in decreasing order the elements contained in the sequences $\mathbf{X}^{(1)}(u), \mathbf{X}^{(2)}(u),\cdots$, where the latter are independent, random mass partitions with values in $\mathcal{S}$ having the same distribution as $\mathbf{X}(u)$ under $\mathbb{P}_{(s_1,0,0,\cdots)}$, $\mathbb{P}_{(s_2,0,0,\cdots)}$, $\cdots$ respectively. The process $\mathbf{X}$ is homogenous in the sense that,  for every $r>0$, the law of $\{r\mathbf{X}(t): t\geq 0\}$ under $\mathbb{P}_{(1,0,0,\cdots)}$ is $\mathbb{P}_{(r,0,0,\cdots)}$.  

Similarly to branching processes, fragmentation processes have embedded genealogies in the sense that a block at time $t$ is a descendent from a block at time $s<t$ if it has fragmented from it. It is possible to formulate the notion of the evolution of such a genealogical line of descent chosen `uniformly at random'. To avoid a long, detailed exposition, we refrain  from providing the technical details here, and mention instead 
that if $\{X^*(t): t\geq 0\}$ is the sequence of embedded fragments along the aforesaid uniformly chosen genealogical line of descent, then it turns out that $\{-\log X^*(t): t\geq 0\}$ is a subordinator.  We suppose that, for $\theta\geq 0$, $\phi(\theta) = \mathbb{E}_{(1,0,0,\cdots)}(X^*(t)^\theta)$  is the Laplace exponent of this subordinator.

Knobloch and Kyprianou \cite{knobkyp} show that for appropriate values of $c,p\geq 0$,
\begin{equation}
M_t : = \sum_{i\in\mathcal{I}^{c}_t} \E^{- \psi(p)t}{X}_i(t)W^{(\psi(p))}(ct+ \log {X}_i(t)), \,\, t\geq 0, 
\label{robertmg}
\end{equation}
is a martingale, 
where for $\theta \geq 0$,
$\psi(\theta) = c \theta - \phi(\theta)$ is the Laplace exponent of a spectrally negative L\'evy process, $W^{(q)}$ is the $q$-scale function with respect to $\psi$ and $\mathcal{I}^{c}_t$ is the set of indices of fragments at time $t$ whose genealogical line of descent has the property that, for all $s\leq t$, its ancestral fragment at time $s$ is larger than  $\E^{-cs}$. One may think of the sum as being over a `thinned' version of the original fragmentation process. That is to say, an adjustment of the original process in which  blocks are removed if, at time $t\geq 0$, they become smaller than $\E^{-ct}$. Removed blocks may therefore no longer contribute fragmented mass to the on-going process of fragmentation.

Analysis of this martingale in \cite{knobkyp}, in particular making use of known properties of scale functions,  allows the authors to deduce that, under mild conditions, there exists a unique constant $p^*>0$ such that whenever $c>\phi'(p^*)$ 
\begin{equation}
\sup_{i\in\mathcal{I}^{c}_t}-\frac{\log X_i(t) }{t}\rightarrow \phi'(p^*),
\label{trimming}
\end{equation}
as $t\uparrow\infty$ on the event that the index set $\mathcal{I}^{c}_t$ remains non-empty for all $t\geq 0$ (i.e. the thinned process survives), which itself occurs with positive probability.
This result is of interest when one compares it against the growth of the largest fragment without restriction of the index set. In that case it is known that 
\[
-\frac{\log X_i(t) }{t}\rightarrow \phi'(p^*),
\]
as $t\uparrow\infty$. Intuitively speaking (\ref{trimming}) says that the thinned fragmentation process will either become eradicated (all blocks are removed in a finite time) or the original fragmentation process will survive the thinning procedure and the decay rate of the largest block is unaffected.

A more elaborate martingale, similar in nature to (\ref{robertmg}) and also built using a scale function, was used by Krell \cite{krell} for a more detailed analysis of different rates of fragmentation decay that can occur in the process. In particular, a fractal analysis of the different decay rates is possible.

\bigskip

\noindent {\bf L\'evy processes conditioned to stay positive:} What is important in many of the examples above is the behaviour of the underlying spectrally negative L\'evy process as it exists a half line. In contrast, one may consider the behaviour of such L\'evy processes conditioned never to exit a half line. This is of equal practical value from a modelling point of view, as well as having the additional curiosity that the conditioning event can occur with probability zero.

Take, as usual, $X$ to be a spectrally negative L\'evy process and recall the definition (\ref{firstpassagetime}) of $\tau^-_0$. Assume that $\psi'(0+)\geq 0$, then it is known that for all $t\geq 0$ and $x,y>0$,
\begin{equation}
 P^\uparrow_x  (X_{t}\in {\rm d}y) = \lim_{q\downarrow 0}P_x(X_t \in {\rm d}y, \, t<\mathbf{e}/q| \tau^-_0>\mathbf{e}/q),
\label{CSP}
\end{equation}
where $\mathbf{e}$ is an independent and exponentially distributed random variable with unit mean, defines the semigroup of a conservative strong Markov process which can be meaningfully called a {\it spectrally negative L\'evy process conditioned to stay positive}.  Moreover, it turns out to be the case that the laws $\{P^\uparrow_x: x>0\}$  can be described through the laws $\{P_x: x>0\}$ via a classical Doob $h$-transform. Indeed, for all $A\in\mathcal{F}_t$ and $x>0$, 
\[
P_x^\uparrow(A) = E_x\left(\frac{W(X_t)}{W(x)}\mathbf{1}_{\{A, \, \tau^-_0 >t \} } \right).
\]
Hence a significant amount of probabilistic information concerning this conditioning is captured by the scale function. See Chaumont and Doney \cite{CD} for a complete overview. In a similar spirit Chaumont \cite{ch1, ch2} also shows that scale functions can be used to describe the law of a L\'evy process conditioned to hit the origin continuously in a finite time. Later on in this review we shall see another example of conditioning spectrally negative L\'evy processes due to Lambert \cite{lambert} which again involves the scale function in a similar spirit.

\chapter{The general theory of scale functions}

\section{Some additional facts about spectrally negative L\'evy processes}

Let us return briefly to the family of spectrally negative L\'evy processes and remind ourselves of some additional deeper properties thereof which, as we shall see, are closely intertwined with the properties of scale functions.

\bigskip

\noindent {\bf Path variation.} Over and above the assumption of spectral negativity in (\ref{kyprianou-palmowski:LH}), the  conditions $$\int_{(-1, 0)} |x| \Pi(\D x)<\infty\, \text{ and }\sigma=0,$$ are necessary and sufficient for $X$ to  have paths of bounded variation.  In that case it may necessarily be decomposed uniquely in the form 
\begin{equation}
X_t  = \delta t -S_t, \, t\geq 0,
\label{BVdecomp}
\end{equation}
where $\delta>0$ and $\{S_t : t\geq 0\}$ is a pure jump subordinator.

\bigskip

\noindent {\bf Regularity of the half line:}   Recall that  {\it irregularity of $(-\infty, 0)$ for $0$} for $X$ means that $P(\tau^-_0 >0) =1$ where 
\[
\tau^-_0 = \inf\{t>0 : X_t<0\}.
\]
Thanks to Blumenthal's zero-one law, the only alternative to irregularity of $(-\infty,0)$ for $0$ is {\it regularity of $(-\infty,0)$ for $0$} meaning that $P(\tau^-_0 >0) =0$. Roughly speaking one may say in words that a spectrally negative L\'evy process enters the open lower half line a.s. immediately from $0$ if and only if it has paths of unbounded variation, otherwise it takes an a.s. positive amount of time before visiting the open lower half line. In contrast, for all spectrally negative L\'evy processes we have $\mathbb{P}(\tau^+_0 =0)=1$, where 
\[
\tau^+_0 = \inf\{t>0 : X_t >0\}.
\]
That is to say, there is always regularity of $(0,\infty)$ for $0$.

\bigskip

\noindent{\bf Creeping.} When a spectrally negative L\'evy process issued from $x>0$ enters $(-\infty,0)$ for the first time it may do so either by a jump or continuously. That is to say, on the event $\{\tau^-_0<\infty\}$, either $X_{\tau^-_0} =0$ or $X_{\tau^-_0 }<0$. The former behaviour is called {\it creeping} downwards (with positive probability). We are deliberately vague here about the initial value $x>0$ under which creeping occurs as it turns out that if $P_y(X_{\tau^-_0} =0)>0$ for some $y>0$, then $P_x(X_{\tau^-_0} = 0)>0$ for all $x> 0$. It is known that the only spectrally negative L\'evy processes which can creep downwards are those processes which have a Gaussian component. Note that, thanks to the fact that there are no positive jumps, a spectrally negative L\'evy process will always a.s. creep upwards over any level above its initial position, providing the path reaches that level. That is to say, for all $a\geq x$,  $P_x(X_{\tau^+_a} = a; \tau^+_a <\infty ) =P_x(\tau^+_a <\infty ) $ where
\[
\tau_a^+=\inf \{t\geq0:X_{t} > a\}.
\]

\bigskip

\noindent{\bf Exponential change of measure.} The Laplace exponent also provides a natural instrument with which one may construct an `exponential change of measure'  on $X$ in the sprit of the classical Cameron-Martin-Girsanov transformation, a technique which plays a central role throughout this text. 
The equality
(\ref{kyprianou-palmowski:Laplace-functional}) allows for a
Girsanov-type change of measure to be defined, namely via
\begin{equation}
	\frac{\D P_{x}^{c}}{\D P_{x}}\biggr|_{\mathcal{F}_{t}}
	=
	\E^{c (X_t-x) - \psi(c)t},\,\, t\geq 0,
	\label{COM}
\end{equation}
for any $c$ such that $|\psi(c)|<\infty$. 
Note that that it is a straightforward exercise to show that the right hand side above is a martingale thanks to the fact that
$X\;$has stationary independent increments together with
(\ref{kyprianou-palmowski:Laplace-functional}). Moreover, the absolute continuity implies
that under this change of measure, $X$ remains within the class of
spectrally negative processes and the Laplace exponent of $X$ under
$P_{x}^{c}$ is given by
\begin{equation}
	\psi_{c}(\theta) =\psi(\theta+c) -\psi(c),
	\label{exp-shift}
\end{equation}
for $\theta \geq -c$. If $\Pi_{c}$ denotes the L\'evy measure of $(X, P^c)$, then it is straightforward to deduce from (\ref{exp-shift}) that 
\[
\Pi_c({\rm d}x) = {\rm e}^{cx}\Pi({\rm d}x),
\]
for $x<0$.

\bigskip

\noindent{\bf The Wiener-Hopf factorization}.
Suppose that $\underline{X}_t : = \inf_{s\leq t} X_s$ and $\mathbf{e}_q$ is an independent and exponentially distributed random variable with rate $q>0$, then the variables $X_{\mathbf{e}_q} - \underline{X}_{\mathbf{e}_q}$ and $-\underline{X}_{\mathbf{e}_q}$ are independent. 
The fundamental but simple concept of duality for L\'evy processes, which follows as a direct consequence of stationary and independent increments and c\`adl\`ag paths states that $\{X_t - X_{(t-s)-}: s\leq t \}$ is equal in law to $\{X_s : s\leq t\}$. This in turn implies that $X_t-\underline{X}_t$ is equal in distribution to $\overline{X}_t : = \sup_{s\leq t}X_s$. It follows that  for all $\theta\in\mathbb{R}$,
\begin{equation}
E(\E^{\I\theta X_{\mathbf{e}_q}}) = E(\E^{\I\theta\overline{X}_{\mathbf{e}_q}})
E(\E^{\I\theta\underline{X}_{\mathbf{e}_q}}).
\label{WHF}
\end{equation}
This identity is known as the Wiener-Hopf factorization.
Note that the reasoning thus far applies to the case that $X$ is a general L\'evy process. However, when $X$ is a spectrally negative L\'evy process, it is possible to express the right hand side above in a more convenient analytical form.
Let $a>0,$ since $t\wedge\tau_a^+ $ is a
bounded stopping time and $X_{t\wedge\tau_a^+}\leq a$, it follows from
Doob's Optional Stopping Theorem applied to the exponential martingale in (\ref{COM})  that
\begin{equation*}
	E\Bigl(\E^{\Phi(q) X_{t\wedge\tau_a^+}
	-q(t\wedge \tau_a^+)}\Bigr) =1.
\end{equation*}
By dominated convergence and the fact that $X_{\tau_a^+}=a$ on
$\tau_a^+<\infty$ we have,
\begin{eqnarray}
	P(\overline{X}_{\mathbf{e}_q} >a) &=& 
	P(\tau^+_a < \mathbf{e}_q)\notag\\
&=&	E\bigl(\E^{-q\tau_a^+}\mathbf{1}_{(\tau_a^+<\infty )}\bigr)\notag\\
	&=&\E^{-\Phi(q) a}.
	\label{firstpassagesub}
\end{eqnarray}
The conclusion of the above series of equalities is that the quantity $\overline{X}_{\mathbf{e}_q}$ is exponentially distributed with parameter $\Phi(q)$ and therefore for all $\theta\in\mathbb{R},$
\begin{equation}
E(\E^{\I\theta \overline{X}_{\mathbf{e}_q}}) = \frac{\Phi(q)}{\Phi(q) - \I \theta }.
\label{WHup}
\end{equation}
It follows from the factorization (\ref{WHF}) that  for all $\theta\in\mathbb{R}$
\begin{equation}
E(\E^{\I\theta\underline{X}_{\mathbf{e}_q}}) = \frac{q}{\Phi(q)}\frac{\Phi(q) -\I \theta}{q+\psi(\theta)}.
\label{WHdown}
\end{equation}
As we shall see later on, a number of features described above concerning the Wiener-Hopf factorization, including the previous identity, play an instrumental role in the theory and application of scale functions.


\section{Existence  of scale functions}\label{section:existence}

We are now ready to show the existence of scale functions. That is to say, we will show that for Definition \ref{LTdefW} to make sense, the function $(\psi(\beta) -q)^{-1}$ is genuinely a Laplace transform in $\beta$. This will largely be done through the Wiener-Hopf factorization in combination with an exponential change of measure. 

\begin{theorem}\label{th:existence}
For all spectrally negative L\'evy processes, $q$-scale functions exist for all $q\geq 0$.
\end{theorem}

\begin{proof} First assume that $\psi'(0+)>0$.
With a pre-emptive choice of notation, define the function
\begin{equation}
\label{kyprianou-palmowski:forces-left-cts}
	W(x)=\frac{1}{\psi'(0^+)}\,P_{x}(\underline{X}_\infty \geq  0).
\end{equation}
Clearly $W(x)=0$ for $x< 0$, it is right-continuous since it is also
proportional to the distribution function
$P(-\underline{X}_\infty \leq x).$
Integration by parts shows that, on the one hand, 
\begin{eqnarray}
 \int_0^{\infty}\E^{-\beta x}
	W(x)\D x
	&=&
\frac{1}{\psi'(0+)}\int_0^{\infty}\E^{-\beta x}
	P(-\underline{X}_{\infty}\leq x)\,\D x  \notag\\
	&=&\frac{1}{\psi'(0+)\beta}	 \int_{[0,\infty)}
	\E^{-\beta x}\,P(-\underline{X}_{\infty}\in \D x)\notag \\
	&=&\frac{1}{\psi'(0+)\beta}  E\bigl(\E^{\beta \underline{X}_{\infty}}\bigr). 
	\label{LTW}
\end{eqnarray}
On the other hand, taking limits as $q\downarrow 0$ in (\ref{WHdown}) gives us the identity
\[
E(\E^{\I\theta\underline{X}_{\infty} }) = -\psi'(0+)\frac{\I \theta}{\psi(\theta)},
\]
where, thanks to the assumption that $\psi'(0+)>0$, we have used the fact that $\Phi(0)=0$  and hence
\[
\lim_{q\downarrow 0} \frac{q}{\Phi(q)} = \lim_{q\downarrow 0} \frac{\psi(\Phi(q))}{\Phi(q)} = \lim_{\theta\downarrow 0}\frac{\psi(\theta)}{\theta} = \psi'(0+).
\]
A straightforward argument using analytical extension now gives us the identity 
\begin{equation}
E(\E^{\beta\underline{X}_{\infty} }) = \frac{\psi'(0+)\beta}{\psi(\beta)}.
\label{LTinf}
\end{equation}
Combining  (\ref{LTW}) with (\ref{LTinf}) gives us (\ref{LTdefW}) as required for the case $q=0$ and $\psi'(0+)>0$.

Next we deal with the case where $q>0$ or $q=0$ and $\psi'(0+)<0$. To this end, again making use of a pre-emptive choice of notation, let us define 
\begin{equation}
W^{(q)}(x) = \E^{\Phi(q)x}W_{\Phi(q)}(x),
\label{W_P}
\end{equation}
where $W_{\Phi(q)}$ plays the role of $W$ but for the process $(X, P^{\Phi(q)})$. Note in particular that by (\ref{exp-shift}) the latter process has Laplace exponent 
\begin{equation}
\psi_{\Phi(q)}(\theta) = \psi(\theta+\Phi(q)) -q,
\label{tilted-exponent}
\end{equation}
for $\theta\geq 0,$ and hence $\psi'_{\Phi(q)}(0+) = \psi'(\Phi(q))>0,$ which ensures that $W_{\Phi(q)}$ is well defined by the previous part of the proof. Taking Laplace transforms we have for $\beta>\Phi(q)$,
\begin{eqnarray*}
\int_0^\infty \E^{-\beta x}W^{(q)}(x)\D x &=&\int_0^\infty \E^{-(\beta-\Phi(q)) x}W_{\Phi(q)}(x)\D x\\
&=&\frac{1}{\psi_{\Phi(q)}(\beta-\Phi(q))}\\
&=&\frac{1}{\psi(\beta) -q},
\end{eqnarray*}
thus completing the proof for the case $q>0$ or $q=0$ and $\psi'(0+)<0$. 

Finally, the case that $q=0$ and $\psi'(0+)= 0$ can be dealt with as follows. Since
$W_{\Phi(q)}(x)$ is an increasing function, we may also treat it as a distribution function of a measure which we also, as an abuse of notation, call $W_{\Phi(q)}$. Integrating by parts thus gives us for $\beta>0,$
\begin{equation}\label{kyprianou-palmowski:W}
	\int_{[0,\infty)}\E^{-\beta x}\,W_{\Phi(q)}(\D x)
	=\frac{\beta}{\psi_{\Phi(q)}(\beta)}\,.
\end{equation}
Note that the assumption $\psi'(0+)=0,$ implies that $\Phi(0)=0,$ and hence for 
$\theta\geq 0$, 
\[
\lim_{q\downarrow 0}\psi_{\Phi(q)}(\theta) = \lim_{q\downarrow 0}
[\psi(\theta + \Phi(q))-q ]= \psi(\theta).
\]
One may appeal to the Extended Continuity Theorem for Laplace
Transforms, see Feller (1971) Theorem XIII.1.2a, and
(\ref{kyprianou-palmowski:W}) to deduce that since
\begin{equation*}
	\lim_{q\downarrow 0}\int_{[0,\infty)}
	\E^{-\beta x}\,W_{\Phi(q)}(\D x)
	=\frac{\beta}{\psi(\beta)},
\end{equation*}
then there exists a measure $W^{\ast}$ such that 
$W^{\ast}(x) := W^*[0,x] = \lim_{q\downarrow 0} W_{\Phi(q)}(x)$ and
\begin{equation*}
	\int_{[0,\infty)}\E^{-\beta x}\,W^{\ast}(\D x)
	=\frac{\beta}{\psi(\beta)}.
\end{equation*}
Integration by parts shows that the right-continuous distribution,
\[
	W(x):=W^*[0,x]
\]
satisfies
\begin{equation*}\label{kyprianou-palmowski:W(x)Dx}
	\int_0^{\infty}\E^{-\beta x}W(x)\,\D x=\frac{1}{\psi(\beta)},
\end{equation*}
for $\beta>0$ as required. 
\hfill$\square$\end{proof}

Let us return to Theorem \ref{twosided-up} and show that, now that we have a slightly clearer understanding of what a scale function is, a straightforward proof of the aforementioned identity can be given. 

\begin{proof}[of Theorem \ref{twosided-up}]
First we deal with the case that $q=0$ and $\psi'(0+)>0$ as in the previous proof.  Since we have identified $W(x) = P_x(\underline{X}_\infty\geq 0)/\psi'(0+)$, a simple argument using the law of
total probability and the Strong Markov Property now yields for $x\in
[0,a]$
\begin{eqnarray}
	\hbox to 1em{$P_{x}(\underline{X}_{\infty} \geq  0)$\hss}
	&&\nonumber\\
	&=&E_{x}\left(P_{x}(\underline{X}_{\infty} >0\,|\,\mathcal{F}_{\tau_a^+})\right)
	\nonumber\\
	&=&E_{x}\Bigl(\mathbf{1}_{(\tau_a^+<\tau_0^-)}
	P_a(\underline{X}_{\infty} \geq  0)\Bigr)
	+E_{x}\Bigl(\mathbf{1}_{(\tau_a^+>\tau_0^-)}
	P_{X_{\tau_0^-}}(\underline{X}_{\infty}\geq  0)\Bigr).
\label{kyprianou-palmowski:importance of left cts}
\end{eqnarray}
The first term on the right hand side above is equal to 
\[
P_a(\underline{X}_{\infty} \geq 0)\,P_{x}(\tau_a^+<\tau_0^-).
\]
The second term on the right hand side of (\ref{kyprianou-palmowski:importance of left cts}) is more complicated to handle but none the less is always equal to zero. To see why, first suppose that $X$ has no Gaussian component. In that case it cannot creep downwards and hence $X_{\tau^-_0}<0$ and the claim follows by virtue of the fact that 
$P_x (\underline{X}_{\infty} \geq  0)=0$ for
$x< 0$. If, on the other hand, $X$ has a Gaussian component, the previous argument still applies on the event $\{X_{\tau^-_0}<0\}$, however we must also consider the event $X_{\tau^-_0}=0$. In that case we note that regularity of $(-\infty, 0)$ for $0$ for $X$ implies that $P (\underline{X}_{\infty} \geq  0)=0$. Returning to (\ref{kyprianou-palmowski:importance of left cts}) we may now deduce that 
\begin{equation}
	P_{x}(\tau_a^+<\tau_0^-)=\frac{W(x)}{W(a)},
\label{kyprianou-palmowski:q=0}
\end{equation}
and clearly the same equality holds even when $x< 0$ as both left and right hand side are identically equal to zero.

Next we deal with the case $q>0$. Making use of (\ref{kyprianou-palmowski:Laplace-functional}) in a, by now, familiar way, and recalling that $X_{\tau^+_a} = a$, we have that
\begin{eqnarray*}
	E_{x}\Bigl(\E^{-q\tau_a^+}\mathbf{1}_{(\tau_a^+<\tau_0^-)}\Bigr)
	&=&E_{x}\biggl(\E^{\Phi(q)(X_{\tau^+_a} - x) - q\tau^+_a}\mathbf{1}_{(\tau_a^+<\tau_0^-)}\biggr)
	\E^{-\Phi(q)(a-x)}
\\
	&=&\E^{-\Phi(q)(a-x)}P_{x}^{\Phi(q)}
	\bigl(\tau_a^+<\tau_0^-\bigr)\\
	&=&\E^{-\Phi(q)(a-x)} \frac{W_{\Phi(q)} (x)}{W_{\Phi(q)}(a)}\\
	&=& \frac{W^{(q)}(x)}{W^{(q)}(a)}.
\end{eqnarray*}

Finally, to deal with the case that $q=0$ and $\psi'(0+)\leq 0$, one needs only to take limits as $q\downarrow 0$ in the above identity, making use of monotone convergence on the left hand side and continuity in $q$ on the right hand side thanks to the Continuity Theorem for Laplace transforms.
\hfill$\square$\end{proof}

Before moving on to looking at the intimate relation between scale functions and the so-called excursion measure associated with $X$,  let us make some further remarks about the scale function and its relation to the Wiener-Hopf factorization.  It is clear from the proof of Theorem \ref{th:existence} that the very definition of a scale function is embedded in the Wiener-Hopf factorization. The following corollary to the aforementioned theorem reinforces this idea.

\begin{corollary}\label{infexp}
For $x\geq
0$,
\begin{equation}
P(-\underline{X}_{\mathbf{e}_{q}}\in {\D}x)=\frac{q}{\Phi
\left( q\right) }W^{(q)}({\D}x)-qW^{(q)}(x){\D}x.  \label{law of inf}
\end{equation}
\end{corollary}
\begin{proof}
A straightforward argument using analytical extension allows us to re-write the Wiener-Hopf factor (\ref{WHdown}) as a Laplace transform,
\[
E(\E^{\beta\underline{X}_{\mathbf{e}_q}}) = \frac{q}{\Phi(q)}\frac{\beta -\Phi(q)}{\psi(\beta)-q},
\]
for $\beta \geq 0$. Note in particular, when $\beta = \Phi(q)$ the right hand side is understood in the limiting sense using L\'H\^opital's rule. The identity (\ref{law of inf})
now follows directly from the Laplace transform (\ref{LTdefW}) and an integration by parts. \hfill$\square$
\end{proof}

\section{Scale functions and the excursion measure}\label{scale-excursion}

Section \ref{section:existence} shows that there is  an intimate relationship between scale functions and the Wiener-Hopf factorization. The Wiener-Hopf factorization  can itself be seen as a distributional consequence of a decomposition of the path of any L\'evy process into excursions from its running maximum, so-called {\it excursion theory}; see for example the presentation in Greenwood and Pitman \cite{GP}, Bertoin \cite{bert96} and Kyprianou \cite{kyp06}.
Let us briefly spend some time reviewing the theory of excursions within the context of spectrally negative L\'evy processes and thereafter we will show the connection between scale functions and a key object which plays a central role in the latter  known as the {\it excursion measure}.

The basic idea of excursion theory for a L\'evy process  is to provide a structured mathematical description of the successive sections of its trajectory which make up  sojourns, or excursions, from its previous maximum. In order to do this, we need to introduce a new time-scale  which will help us locate the original times at which $X$ creates new maxima. Specifically we are interested in a  (random) function of original time, or clock, say $L=\{L_t: t\geq 0\}$, such that $L$ increases precisely at times when $X$ creates a new maximum; namely the times $\{t: X_t = \overline{X}_t\}$. Moreover, $L$ should have a regenerative property in that if $T$ is any $\mathbb{F}$-stopping time such that $\overline{X}_T = X_T$ on $\{T<\infty\}$ then $\{(X_{T+t} - X_T, L_{T+t} - L_T): t \geq 0\}$ has the same law as $\{(X_t, L_t): t\geq 0\}$ under $P$. The process $L$ is referred to as {\it local time at the maximum}.

It turns out that for spectrally negative L\'evy processes, there is a natural choice of $L$ which is simply $L = \overline{X}$. Indeed it is straightforward to verity that $\overline{X}$ has the required regenerative property.  We shall henceforth proceed with this choice of $L$. Now define 
\[
L^{-1}_t = \inf\{s>0: L_s >t\}.
\]
Note that in the above definition, $t$ is a local time and $L^{-1}_t$ is the real time which is required to accumulate $t$ units of local time. On account of the fact that $L = \overline{X}$, it is also the case that $L^{-1}_t$ is the first passage time of the process $X$ above level $t$. In the notation of the previous section, $L^{-1}_t = \tau^+_t$.

Standard theory tells us that, for all $t>0$, both $L^{-1}_t$ and $L^{-1}_{t-}$ are stopping times. Moreover, it is quite clear that whenever $\Delta L^{-1}_t : = L^{-1}_{t-}- L^{-1}_t >0$, the process $X$ experiences an excursion from from its previous maximum $\overline{X}$. For such local times $t>0$
the associated excursion shall be written as 
\[
\epsilon_t = \{\epsilon_t(s) := X_{L^{-1}_{t}} - X_{L^{-1}_{t-}+s}: 0< s\leq \Delta L^{-1}_{t}  \}.
\]
For local times $t>0$ such that $\Delta L_{t}^{-1}=0$ we define $\epsilon_t = 
\partial$ where $\partial$ is some isolated state. For each $t>0,$ $\epsilon_{t}$ takes values in the space of excursions, $\mathcal{E},$ the space of real valued right continuous left limited paths killed at the first hitting time of $(-\infty,0]$. We suppose that $\mathcal{E}$ is endowed with the sigma-algebra generated by the coordinate maps, say $\mathcal{G}$. 

The key feature of excursion theory in the current setting is that $\{(t, \epsilon_t): t\geq 0 \text{ and }\epsilon_t \neq \partial\}$ is a Poisson point process with values in $(0,\infty)\times\mathcal{E},$ with intensity measure ${\rm d}t\times {\rm d}n,$ where $n$ is a measure on the space $(\mathcal{E}, \mathcal{G})$. 

  Let us now return to the relationship between excursion theory and the scale function. First write $\zeta$ for the lifetime  and  $\overline{\epsilon }$ for the height of the generic excursion $\epsilon\in\mathcal{E}$. That is to say, for $\epsilon\in\mathcal{E}$,
\[
\zeta = \inf\{s>0 : \epsilon(s)\leq 0\} \text{ and }\overline{\epsilon} = \sup\{\epsilon(s): s\leq \zeta\}.
\]
 Choose $a>x\geq 0\ $. Recall that 
$L=\overline{X}$ we have that  $L_{\tau _{a-x}^{+}}=\overline{X}_{\tau
_{a-x}^{+}}=a-x.$ Using this it is not difficult to see that
\[
\{\underline{X}_{\tau^+_{a-x}} \geq  - x\} = \{\forall t\leq a-x
\text{ and }\epsilon_t\neq \partial, \overline{\epsilon}_t \leq t+x\}.
\]
Let $N$ be the Poisson random measure associated with the Poisson point 
process of excursions; that is to say for all ${\rm d}t\times n({\rm d}\epsilon)$-measurable sets $B\in[0,\infty)\times\mathcal{E}$, $N(B) = \mbox{\rm card}\{(t, \epsilon_t)\in B, t>0\}$.  With the help of (\ref{kyprianou-palmowski:q=0}) we have that
\begin{eqnarray*}
\frac{W(x)}{W\left( a\right) } &=&P_{x}\left(
\underline{X}_{\tau
_{a}^{+}}\geq 0\right)  \notag \\
&=&P\left( \underline{X}_{\tau _{a-x}^{+}}\geq -x\right)  \notag \\
&=&P(\forall t\leq a-x \text{ and }\epsilon_t\neq \partial, \overline{\epsilon}_t \leq t+x)\\
&=&P(N(A)=0),
\end{eqnarray*}
where   $A=\{(t,\epsilon_t): t\leq a-x \text{
and } \overline{\epsilon}_t > t+x\}$. Since $N(A)$ is Poisson
distributed with parameter $\int \mathbf{1}_A \ {\D}t \
n({\D}\epsilon) = \int_{0}^{a-x}n\left( \overline{\epsilon }>
t+x\right) {\D}t = \int_{x}^{a}n\left( \overline{\epsilon }>
t\right) {\D}t$ we have that
\begin{equation}
\frac{W(x)}{W\left( a\right) } =\exp \left\{ -\int_{x}^{a}n\left(
\overline{\epsilon }> t\right) {\D}t\right\}  \label{C and D
follow}.
\end{equation}
This is a fundamental representation of the scale function $W$ which we shall return to in later sections. However, for now, it leads us immediately to the following analytical conclusion.

\begin{lemma}\label{a.e.diff}
For any $q\geq 0$, the scale function $W^{(q)}$ is continuous, almost everywhere differentiable and strictly increasing.
\end{lemma}
\begin{proof}Assume that $q=0$.
Since $a$ may be chosen arbitrarily large, continuity and strict
monotonicity follow from (\ref{C and D follow}). 
Moreover, we also see from (\ref{C and D follow}) that the left and right
first derivatives exist and  are given by
\begin{equation}
W_-^{\prime }(x)=n(\overline{\epsilon }\geq x)W(x) \text{ and }
W_+^{\prime }(x)=n(\overline{\epsilon }> x)W(x).
\label{landrder}
\end{equation}
Since there can be at most a countable number of points for which the monotone function $n(\overline\epsilon>x)$ has discontinuities, it follows that $W$ is almost everywhere differentiable. 

From the relation (\ref{W_P}) we know that
\begin{equation}
W^{(q)}(x)=\mathrm{e}^{\Phi (q)x}W_{\Phi (q)}\left( x\right)
\label{W interchange},
\end{equation}
and hence the properties of continuity, almost everywhere differentiability and strict monotonicity
carry over to the case $q>0.$ \hfill$\square$\end{proof}

Let us write as an abuse of notation  $W^{(q)}\in C^1(0,\infty)$ to mean that the restriction of $W^{(q)}$ to $(0,\infty)$ belongs to $C^{1}(0,\infty)$.
Then the proof of the previous lemma indicates that $W\in C^1(0,\infty)$ as soon as the measure $n(\overline{\epsilon}\in \D x)$ has no atoms. It is possible to combine this observation together with other analytical and probabilistic facts to recover necessary and sufficient conditions for first degree smoothness of $W$.

\begin{lemma}\label{C1}For each $q\geq 0$,
the scale function $W^{(q)}$ belongs to $C^1(0,\infty)$ if and only if at least one of the following three criteria holds, 
\begin{enumerate}
\item[(i)] $\sigma\neq 0$
\item[(ii)] $\int_{(-1,0)}|x|\Pi(\D x) = \infty$
\item[(iii)] $\overline{\Pi}(x): = \Pi(-\infty, -x)$ is continuous.
\end{enumerate}
\end{lemma}

\begin{proof}
Firstly note that it suffices to consider the case that $q=0$ and $\psi'(0+)\geq 0$, i.e. the case that the L\'evy process oscillates or drifts to $+\infty$. Indeed, in the case that $\psi'(0+)<0$ or $q>0,$ one recalls that $\Phi(0)>0,$ respectively $\Phi(q)>0,$  and we appeal to  (\ref{W interchange}), (\ref{W_P})
where $W_{\Phi(0)},$ respectively $W_{\Phi(q)},$ is the scale function of a spectrally negative L\'evy process which drifts to $+\infty$.

If it is the case that there exists an $x>0,$ such that $n(\overline{\epsilon} = x)>0,$ then one may consider the probability law $n(\cdot | \overline{\epsilon} = x) = n(\cdot , \overline\epsilon =x)/n(\overline\epsilon = x)$. Assume now that $X$ has paths of unbounded variation, that is to say (i) or (ii) holds, and recall that in this case $0$ is regular for $(-\infty, 0)$. If we apply the strong Markov property under $n(\cdot | \overline{\epsilon} = x)$ to the excursion at the time $T_x : = \inf\{t> 0 : \epsilon (t) = x\}$, the aforementioned regularity of $X$ implies that  $(x,\infty)$ is regular for $\{x\}$ for the process $\varepsilon$ and therefore that $\overline{\epsilon}>x,$ under $n(\cdot|\overline\epsilon = x)$. However this constitutes a contradiction. It thus follows that $n(\overline\epsilon = x) =0$ for all $x>0$.

Now assume that $X$ has paths of bounded variation. It is known from
\cite{Pistorius} and Proposition 5 in \cite{Winkel} that when $X$ is of
bounded variation the excursion measure of $X$ reflected at its
supremum begins with a jump and can be described by the formula
\begin{equation*}
n\left(F\left(\epsilon(s), 0\leq s\leq \zeta\right)
\right)=\frac{1}{\delta}\int^{0}_{-\infty}\Pi(\D x)\widehat{E}_{-x}
\left(F(X_{s}, 0\leq s\leq \tau^-_{0})\right),
\end{equation*}
where  $F$ is any nonnegative measurable functional on the space of
cadlag paths,  $\widehat{E}_{x}$ denotes the law of the dual L\'evy
process $\widehat{X}=-X$ and $\tau^-_{x}=\inf\{s>0: X_{s}< x\},$
$x\in\mathbb{R}.$ Recall that  $\tau^+_{z}=\inf\{s>0: X_{s}>z\},$
$z\in\mathbb{R}.$ Hence, it follows that
\begin{eqnarray*}
n\left(\overline\epsilon>z\right)&=&
\frac{1}{\delta}\int_{(-\infty,0)}\Pi(\D x)\widehat{P}_{-x}\left(
\sup_{0\leq s\leq \tau^-_{0}}X_{s}>z\right)\\
&=&\frac{1}{\delta}\Pi(-\infty, -z)+\frac{1}{\delta}\int_{[-z,0)}
\Pi(\D x)\widehat{P}_{-x}\left(\tau^{+}_{z}<\tau^-_{0}\right)\\
&=&\frac{1}{\delta}\Pi(-\infty, -z)+\frac{1}{\delta}\int_{[-z,0)}
\Pi(\D x)\widehat{P}\left(\tau^{+}_{z+x}<\tau^-_{x}\right)\\
&=&\frac{1}{\delta}\Pi(-\infty, -z)+\frac{1}{\delta}\int_{[-z,0)}
\Pi(\D x)P\left(\tau^{-}_{-x-z}<\tau^+_{-x}\right)\\
&=&\frac{1}{\delta}\Pi(-\infty, -z)+\frac{1}{\delta}\int_{[-z,0)}
\Pi(\D x)\left(1-\frac{W(z+x)}{W(z)}\right),
\end{eqnarray*}
where in the last equality we have used Theorem \ref{twosided-up} with $q=0$.
From this it follows that 
\[
n\left(\overline\epsilon=z\right) = \frac{1}{\delta}\Pi(\{-z\})\frac{W(0)}{W(z)},
\]
and hence the required property for  that $n\left(\overline\epsilon=z\right) =0 $ (which in turn leads to the $C^1(0,\infty)$ property of the scale function) follows as soon as  $\Pi$  has no atoms. That is to say $\overline\Pi$ is continuous.
\hfill$\square$\end{proof}
\bigskip

The proof of the above lemma in the bounded variation case gives us a little more information about non-differentiability in the case that $\Pi$ has atoms.

\begin{corollary}\label{corr-for-alexey}
When $X$ has paths of bounded variation the scale function $W^{(q)}$ does not possess a derivative at $x>0$ (for all $q\geq 0$) if and only if $\Pi$ has an atom at $-x$. In particular, if $\Pi$ has a finite number of atoms supported by the set $\{-x_1, \cdots, -x_n\}$ then, for all $q\geq 0$,  $W^{(q)}\in C^1((0,\infty)\backslash\{x_1, \cdots, x_n\})$.
\end{corollary}



We conclude by noting  that the above results give us our first analytical impressions on the `shape' and smoothness of $W^{(q)}$. In Section \ref{analytic-properties} we shall explore these properties in greater detail.

\section{Scale functions and the descending ladder height process}\label{scaledescending}

In a similar spirit to the previous section, it is also possible to construct an excursion theory for spectrally negative L\'evy processes away from their infimum. Indeed 
it is well known, and easy to prove, that the process $X$ reflected in its past infimum $X-\underline{X}:=\{X_t-\underline{X}_t: t\geq 0\}$, 
where $\underline{X}_t := \inf_{s\leq t}X_s$, is a strong Markov process with state space $[0,\infty)$.  Following standard theory of Markov local times (cf. Chapter IV of \cite{bert96}), it is possible to construct a  local time at zero 
for $X-\underline{X}$ which we henceforth refer to as $\widehat{L}=\{\widehat{L}_t : t\geq 0\}$. Its right-continuous inverse process, $\widehat{L}^{-1}:=\{\widehat{L}^{-1}_t : t\geq 0\}$ where $\widehat{L}^{-1}_t =\inf\{s>0 : \widehat{L}_s >t\}$, 
is a (possibly killed) subordinator.
Sampling $X$ at $\widehat{L}^{-1}$ we recover the points of local minima of $X$. If we define $\widehat{H}_t =X_{\widehat{L}^{-1}_t}$ 
when $\widehat{L}^{-1}_t<\infty,$ with $\widehat{H}_t  =\infty$ otherwise, then it is known that the 
process $\widehat{H}=\{\widehat{H}_t : t\geq 0\}$ is a (possibly killed) subordinator. 
The latter is known as 
the {\em descending ladder height process}. 

The starting point for the relationship between the descending ladder height process and scale functions 
is given by the following  factorization of the Laplace exponent of $X$,
\begin{equation}
\psi(\theta) = (\theta -\Phi(0))\phi(\theta),\qquad \theta\geq 0,
\label{WHfact}
\end{equation}
where $\phi,$ 
is the Laplace exponent of a (possibly killed) subordinator. See e.g. Section 6.5.2 in \cite{kyp06}. This can be seen  by recalling that $\Phi(0)$ is a root of the equation $\psi(\theta) = 0$ and then factoring out  $(\theta-\Phi(0))$ from $\psi$ by making some integration by parts to deduce that the term $\phi(\theta)$ must necessarily take the form
\begin{equation}\label{C}
\phi(\theta) = \kappa + \xi\theta + \int_{(0,\infty)} (1-\E^{-\theta x})\Upsilon(\D x),
\end{equation}
where $\kappa, \xi\geq 0$ and $\Upsilon$ is a measure concentrated on $(0,\infty)$ which satisfies $\int_{(0,\infty} (1\wedge x)\Upsilon(\D x)<\infty$.
Indeed, it turns out that 
\begin{equation}
\Upsilon(x,\infty) = \E^{\Phi(0)x}\int_x^\infty \E^{-\Phi(0)u}\Pi(-\infty,-u)\D u\quad \text{for}\ x>0,
\label{tailmeasure}
\end{equation}
 $\xi = \sigma^2/2$ and $\kappa =E(X_1)\vee 0 = \psi'(0+)\vee 0$.

 Ultimately the factorization (\ref{WHfact}) can also be derived by a procedure of analytical extension and taking limits as $q$ tends to zero in (\ref{WHF}), simultaneously making use of the identities (\ref{WHup}) and (\ref{WHdown}). A deeper look at this  derivation yields the additional information that
$\phi$ is the Laplace exponent of $\widehat{H}$; namely $\phi(\theta) = -\log E(\E^{-\theta \widehat{H}_1})$.

In the special case that $\Phi(0)=0$, that is to say, the process $X$ does not drift to $-\infty$ or equivalently that $\psi'(0+)\geq 0$, it can be shown that the scale function $W$ describes the potential measure of $\widehat{H}$. Indeed, recall that the potential measure of $\widehat{H}$ is defined by
\begin{equation}
\int_0^\infty {\rm d} t \cdot P(\widehat{H}_t \in \D x), \qquad \text{ for }x\geq 0.
\label{resolvent2}
\end{equation}
Calculating its Laplace transform we get the identity
\begin{equation}
\int_0^\infty \int_0^\infty {\rm d} t \cdot P(\widehat{H}_t \in \D x) \E^{-\theta x}= 
\frac{1}{\phi(\theta)}
\qquad \text{ for }\theta>0.
\label{equiv}
\end{equation}
Inverting the Laplace transform on the left hand side 
 we get the identity 
\begin{equation}\label{a}
W(x)=\int_0^\infty \D t \cdot P(\widehat{H}_t \in [0,x]), \qquad x\geq 0.
\end{equation}
It can be shown similarly that in general, when $\Phi(0)\geq 0,$ the scale function is related to the potential measure of $\widehat{H}$ by the formula
\begin{equation}\label{b}
W(x)=\E^{\Phi(0)x}\int^x_{0}\E^{-\Phi(0)y}\int^{\infty}_{0}\D t\cdot P(\widehat{H}_t \in \D y),\qquad x\geq 0.\end{equation}

The connection between $W$ and the potential measure of $\widehat{H}$ turns out to be of importance at several points later on in this exposition. 


\section{A suite of fluctuation identities}\label{section:fluctuation identities}

Let us now expose the robustness of scale functions as a natural family of functions to describe a whole suite of fluctuation identities concerning first and last passage problems. We do not offer full proofs, concentrating instead on the basic  ideas that drive the results. These  are largely concentrated around the Strong Markov Property and earlier established results on the law of $\overline{X}_{\mathbf{e}_q}$ and $-\underline{X}_{\mathbf{e}_q}$. 

Many of the results in this section are taken from the recent work \cite{ber96, Pistorious-potential-theoretic, kyprianou-palmowski:1, Pistorius}. However some of the identities can also be found in the  compound Poisson setting from earlier Soviet-Ukranian work; see for example \cite{Korolyuk, korolyuk1975a,korolyuk1975b, bratiychuk1991,korolyuk1976,korolyuk1981, suprun} to name but a few.

\subsection{First passage problems}

Recall that for all $a\in\mathbb{R}$, $\tau^+_a = \inf\{t>0 : X_t >a\}$ and $\tau^-_0 = \inf\{t>0 : X_t < 0\}$. We also introduce for each $q\geq 0$,
\begin{equation}
 Z^{(q)}(x) = 1 + q\int_0^x W^{(q)}(y)\D y.
 \label{zedque}
\end{equation}

\begin{theorem}[One- and two-sided exit formulae]\label{oneandtwo}
\begin{description}
\item[(i)]  For any $x\in \mathbb{R}$ and $q\geq 0,$%
\begin{equation}
 E_{x}\left( \mathrm{e}^{-q\tau
_{0}^{-}}\mathbf{1}_{\left( \tau
_{0}^{-}<\infty \right) }\right) =Z^{(q)}(x)-\frac{q}{\Phi \left( q\right) }%
W^{(q)}(x)\;,  \label{one-sided-down}
\end{equation}
where we understand $q/\Phi \left( q\right) $ in the limiting sense for $%
q=0, $ so that

\begin{equation}
 P_x(\tau _{0}^{-}<\infty
)=\left\{
\begin{array}{ll}
1-\psi'(0+)W(x) & \mathrm{if}\ \psi'(0+)>0 \\[2pt]
1 & \mathrm{if} \ \psi'(0+)\leq 0
\end{array}
\right. . \label{one sided down q=0} \vspace*{3pt}\end{equation}

\item[(ii)] For all $x\geq 0$ and $q\geq 0$ 
\begin{equation}
 E_x (\E^{-q\tau^-_0}\mathbf{1}_{\{ X_{\tau^-_0} = 0\}}) =  \frac{\sigma^2}{2}\left\{
W^{(q)\prime}(x)  -\Phi(q)W^{(q)}(x)
\right\},
\label{creep-q}
\end{equation}
where the right hand side is understood to be identically zero if $\sigma =0$ and otherwise the derivative is well defined thanks to Lemma \ref{C1}.

\item[(iii)]  For any $x\leq a$ and $q\geq 0,$%
\begin{equation}  E_{x}\left( \E^{-q\tau
_{0}^{-}}\mathbf{1}_{\left( \tau
_{0}^{-}<\tau _{a}^{+}\right) }\right) =Z^{(q)}(x)-Z^{(q)}(a)\frac{W^{(q)}(x)%
}{W^{(q)}(a)}.  \label{two-sided-down} \vspace*{3pt}\end{equation}
\end{description}
\end{theorem}

\bigskip

\noindent{\it Sketch proof.} 
(i) Making use of Corollary \ref{infexp} we have  for $q>0$ and $x\geq
0$,
\begin{eqnarray}
E_{x}\left( \mathrm{e}^{-q\tau _{0}^{-}}\mathbf{1}_{(\tau
_{0}^{-}<\infty )}\right) &=&
P_{x}(\mathbf{e}%
_{q}>\tau _{0}^{-})  \notag \\
&=&P_{x}(\underline{X}_{\mathbf{e}_{q}} <0) \notag \\
&=&P(-\underline{X}%
_{\mathbf{e}_{q}}>x)  \notag \\
&=&1-P(-\underline{X}_{\mathbf{e}_{q}}\leq x) \notag \\
&=&1+q\int_{0}^{x}W^{(q)}(y){\D}y-\frac{q}{\Phi \left( q\right) }W^{(q)}(x)  \notag\\
&=&Z^{(q)}(x)-\frac{q}{\Phi \left( q\right) }W^{(q)}(x).
\label{use later}
\end{eqnarray}
The proof for the case $q=0$ follows by taking limits as $q\downarrow 0$ on both sides of the final equality in (\ref{use later}).

(iii) Fix $q>0.$ We have for $x\geq 0$,
\[
E_x(\mathrm{e}^{-q\tau^-_0}\mathbf{1}_{(\tau^-_0 <
\tau^+_a)}) = E_x(\mathrm{e}^{-q
\tau^-_0}\mathbf{1}_{(\tau^-_0<\infty)}) -
E_x(\mathrm{e}^{-q\tau^-_0} \mathbf{1}_{(\tau^+_a
<\tau^-_0)}).
\]
Applying the Strong Markov Property at $\tau^+_a$ and using the
fact that $X$ creeps upwards, we also have that
\[
E_x(\mathrm{e}^{-q\tau^-_0} \mathbf{1}_{(\tau^+_a
<\tau^-_0)}) =E_x(\mathrm{e}^{-q\tau^+_a}
\mathbf{1}_{(\tau^+_a <\tau^-_0)})
 E_a(\mathrm{e}^{-q\tau^-_0} \mathbf{1}_{(\tau^-_0
 <\infty)}).
\]
Piecing the previous two equalities together and appealing to (\ref{one-sided-down}) and (\ref{whatsinaname}) yields the desired conclusion. The case that $q=0$ is again handled by taking limits as $q\downarrow 0$ on both sides of (\ref{two-sided-down}). Here we have used the discussion in Section \ref{scaledescending}.

(ii)  Suppose that $q=0$ and $\psi'(0+)\geq 0$ (equivalently the descending ladder height process is not killed) then the claimed identity reads 
\begin{equation}
P_x ( X_{\tau^-_0} = 0) = \frac{\sigma^2}{2} W'(x),\qquad x\geq 0.
\label{creep-no-q}
\end{equation}
Note that the probability on the left hand side is also equal to the probability that the descending ladder height subordinator creeps over $x$, $P(\widehat{H}_{\widehat{T}_x} =x)$, where $\widehat{T}_x =\inf\{t\geq 0 : \widehat{H}_t>x\}$.
The identity (\ref{creep-no-q}) now follows directly from a classic result of Kesten \cite{kesten} which shows that the probability that a subordinator creeps is non-zero if and only its drift coefficient is strictly positive, in which case it is equal to the drift coefficient multiplied by its potential density, which necessarily exists. It follows together, with (\ref{tailmeasure}) and (\ref{b}), that $P(\widehat{H}_{\widehat{T}_x} =x)=\sigma^{2} W^{\prime}(x)/2$.

When $q>0$ or $\psi^{\prime}(0+)<0,$ the formula is proved using the change of measure (\ref{COM})  with $c=\Phi(q),$ $q\geq 0,$ and the above result together with (\ref{W_P}). 
\hfill$\square$

\bigskip

\noindent Note that part (ii) above tallies with the earlier mentioned fact that spectrally negative L\'evy processes do not creep downwards unless they have a Gaussian component. Note also that when $\sigma\neq 0$ Lemma \ref{C1} tells us that the derivative appearing on the right hand side of the density is everywhere defined for $x\geq 0$.

\medskip

We also give expressions for the expected occupation measure of $X$ in a given Borel set over its entire lifetime as well as when time is restricted up to the first passage times $\tau^+_a, \tau^-_0$ and 
\[
\tau: = \tau^+_a\wedge \tau^-_0. 
\]
Such expected occupation measures are generally referred to as resolvents and play an important role in establishing, for example, deeper identities concerning first passage problems such as the one presented in Theorem \ref{triple-law}.

\begin{theorem}[Resolvents] $\mbox{ }$\label{resolve}
 \begin{description}
  \item[(i)] For all $a\geq x\geq 0$, $q\geq 0$ and Borel set $A\subseteq[0,a]$, 
\[
E_x\left[\int_{0}^\infty \E^{-qt} \mathbf{1}_{\{X_t \in A, \, t<\tau\}}\D t\right]
 =\int_A \left\{\frac{W^{(q)}(x)W^{(q)}(a-y)}{W^{(q)}(a)} -
W^{(q)}(x-y)\right\}\D y.
\]
\item[(ii)] For all $a\geq x$ and Borel set $A\subseteq(-\infty, a]$,
\[
E_x\left[\int_{0}^\infty \E^{-qt} \mathbf{1}_{\{X_t \in A, \, t<\tau^+_a\}}\D t\right]
 =\int_A \left\{\E^{-\Phi(q)(a-x)} W^{(q)}(a-y) - W^{(q)}(x-y)\right\}\D y.
\]
\item[(iii)] For all $x\geq 0$ and Borel set $A\subseteq[0,\infty)$,
\[
E_x\left[\int_{0}^\infty \E^{-qt} \mathbf{1}_{\{X_t \in A, \, t<\tau^-_0\}}\D t\right]
 =\int_A \left\{\E^{-\Phi(q)y}W^{(q)}(x) - W^{(q)}(x-y)\right\}\D y.
\]
\item[(iv)] For all $x\in\mathbb{R}$ and Borel set $A\subseteq \mathbb{R}$
\[
 E\left[\int_{0}^\infty \E^{-qt} \mathbf{1}_{\{X_t \in A\}}\D t\right]
 =\int_A \left\{\Phi'(q) \E^{-\Phi(q)y } - W^{(q)}(-y)\right\}\D y.
\]
\end{description}
\end{theorem}

\bigskip

\noindent {\it Sketch proof.} We give an outline of the proof of (iii) from which 
 the proof of (i) easily follows. The remaining two identities can be obtained by taking limits of the barriers in (i) relative to the initial position. As usual we shall perform the relevant analysis in the case that $q>0$. 
  The case that $q=0$ follows by taking limits as $q\downarrow 0$.
  
 We start by noting that for all $x,y \geq 0$ and $q>0$,
\[
R^{(q)}(x,{\D}y):=\int_0^\infty \mathrm{e}^{-qt} \
P_x(X_t \in {\D}y, \tau^-_0 > t){\D}t =
\frac{1}{q}P_x(X_{\mathbf{e}_q} \in {\D}y,
\underline{X}_{\mathbf{e}_q} \geq 0),
\]
where $\mathbf{e}_q$ is an independent, exponentially distributed
random variable with parameter $q>0$. 

Appealing to the Wiener-Hopf factorization, specifically that
$X_{\mathbf{e}_q}  - \underline{X}_{\mathbf{e}_q} $ is independent
of $\underline{X}_{\mathbf{e}_q}$, and that $X_{\mathbf{e}_q}  - \underline{X}_{\mathbf{e}_q} $ is equal in distribution to $\overline{X}_{\mathbf{e}_q}$, we have that
\begin{eqnarray*}
R^{(q)}(x,{\D}y) &=& \frac{1}{q}P((X_{\mathbf{e}_q} - \underline{X}_{\mathbf{e}_q} ) + \underline{X}_{\mathbf{e}_q} \in {\D}y -x,  - \underline{X}_{\mathbf{e}_q} \leq x) \\
&=&
\frac{1}{q}\int_{[0,x]}P(-\underline{X}_{\mathbf{e}_q }
\in {\D}z) P(
\overline{X}_{\mathbf{e}_q} \in {\D}y - x +z)1_{\{y\geq x-z\}}.
\end{eqnarray*}
Recall however, that $\overline{X}_{\mathbf{e}_q}$  is exponentially
distributed with parameter $\Phi(q)$. In addition, the law of
$-\underline{X}_{\mathbf{e}_q}$ has been identified in Corollary \ref{infexp}. Putting the pieces together and making some elementary manipulations the identity in (iii) follows.

Now suppose we denote the left hand side of the identity in (i) by $U^{(q)}(x, A)$. With the
help of the Strong Markov Property we have that
\begin{eqnarray*}
qU^{(q)}(x,{\D}y) &=& P_x(X_{\mathbf{e}_q} \in {\D}y, \underline{X}_{\mathbf{e}_q} \geq 0, \overline{X}_{\mathbf{e}_q} \leq a) \\
&=&  P_x(X_{\mathbf{e}_q} \in {\D}y, \underline{X}_{\mathbf{e}_q} \geq 0) \\
&&- P_x(X_{\mathbf{e}_q} \in {\D}y, \underline{X}_{\mathbf{e}_q} \geq 0, \overline{X}_{\mathbf{e}_q} >a) \\
&=&  P_x(X_{\mathbf{e}_q} \in {\D}y, \underline{X}_{\mathbf{e}_q} \geq 0) \\
&&- P_x(X_{\tau}=a, \tau < \mathbf{e}_q)P_a(
X_{\mathbf{e}_q}\in {\D}y, \underline{X}_{\mathbf{e}_q} \geq 0).
\end{eqnarray*}
The first and third of the three probabilities on the right-hand
side above have been computed in the previous paragraph, the
second probability may be written
\[
P_x(\mathrm{e}^{- q\tau^+_a}; \tau^+_a<\tau^-_0) =
\frac{W^{(q)}(x)}{W^{(q)}(a)}.
\]
The result now follows by assembling the relevant pieces.
 \hfill$\square$

 \bigskip
 
 On a final note, the above resolvents easily lead to further identities of the type given in Theorem \ref{triple-law}. 
Indeed, suppose that  $N$ is the Poisson random measure associated with the jumps
of $X$. That is to say, $N$ is a Poisson random measure on $[0,\infty)\times(-\infty,0)$ with intensity ${\D}t\times \Pi(\D x)$. Recall that $\tau: = \tau^+_a\wedge \tau^-_0$. Then 
with the help of the Compensation Formula, we have that  for $x\in [0,a]$, $A$ any Borel set in
$[0,a)$ and $B$ any Borel set in $(-\infty,0)$ and $q\geq 0$,
\begin{eqnarray}
&&P_x(\E^{-q\tau};X_{\tau} \in B, X_{\tau-}\in A)\notag \\
 &&=E_x\left(
\int_{[0,\infty)}\int_{(-\infty,0)} \E^{-qt}\mathbf{1}_{(\overline{X}_{t-}\leq
a, \underline{X}_{t-} \geq 0, X_{t-}\in A)}\mathbf{1}_{( y\in B -
X_{t-} )} N({\D}t\times {\D}y)
\right) \notag\\
&&= E_x\left( \int_0^\infty \E^{-qt}\mathbf{1}_{(t<\tau)} \Pi(B -
X_{t})\mathbf{1}_{(X_t\in A)}
{\D}t\right) \notag\\
&&= \int_{A} \Pi(B -y) U^{(q)}(x, {\D}y), \label{return to,}
\end{eqnarray}
where, as noted earlier,
\[
U^{(q)}(x, {\D}y)  = \int_0^\infty \E^{-qt} P_x(X_t \in {\D}y, \tau
>t) {\D}t.
\] Observe that the fact that $B$ is a Borel subset of $(-\infty,0),$ allow us not to consider the event where the process leaves the interval from below by creeping. Nevertheless, the probability of this event has been calculated in (\ref{creep-q}).
 

\subsection{First passage problems for reflected processes}
The list of fluctuation identities continues on when one considers the reflected processes $\overline{Y}:=\{\overline{X}_t - X_t : t\geq 0\}$ and $\underline{Y}:=\{X_t - \underline{X}_t: t\geq 0\}$. Note that it is easy to prove that both of these processes are non-negative strong Markov processes. We shall henceforth denote their probabilities by $\{\overline{P}_x:x\in[0,\infty)\}$ and $\{\underline{P}_x:x\in[0,\infty)\}$. Note that $(\overline{Y}, \overline{P}_x)$ is equal in law to $\{(x\vee\overline{X}_t -X_t): t\geq 0\}$ under $P$ and $(\underline{Y}, \underline{P}_x)$ is equal in law to $\{X_t - (\underline{X}_t \wedge -x):t\geq 0\}$ under $P$.
The following theorem is a compilation of results taken from  \cite{kyprianou-palmowski:1} and \cite{Pistorius}. We do not offer proofs but instead we shall settle for remarking that the proofs use similar techniques to those of the previous section.

For convenience, let us define for $a>0$,
\[
\underline\sigma_a = \inf\{t\geq 0 : \underline{Y}_t >a\} \text{ and } \overline{\sigma}_a
= \inf\{t\geq 0 : \overline{Y}_t>a\},
\]
where $\underline{Y}_t = X_t - \underline{X}_t$ and $\overline{Y}_t = \overline{X}_t - X_t$.

\begin{theorem}
Suppose that $a>0$, $x\in[0,a]$ and $q\geq 0$. Then 
\begin{enumerate}
\item[(i)] $\underline{E}_x(\E^{-q\underline{\sigma}_a}) = Z^{(q)}(x)/Z^{(q)}(a),$
\item[(ii)]    taking $W_+^{(q)\prime}(a)$ is the right derivative of $W^{(q)}$ at $a$, $$\overline{E}_x(\E^{-q\overline{\sigma}_a}) = Z^{(q)}(a-x) -qW^{(q)}(a-x)W^{(q)}(a)/W^{(q)\prime}_+(a),$$ 
\item[(iii)]  for any Borel set $A\in[0,a)$,
\[
\underline{E}_x\left[\int_0^\infty \E^{-qt}\mathbf{1}_{\{\underline{Y}_t\in A, \, t<\underline{\sigma}_a\}} \D t\right] = \int_A\left\{ \frac{Z^{(q)}(x)}{Z^{(q)}(a)}W^{(q)}(a-y)  - W^{(q)}(x-y) \right\}\D y,
\]
and
\item[(iv)]   for any Borel set $A\in[0,a)$,
\begin{eqnarray*}
\overline{E}_x\left[\int_0^\infty \E^{-qt} \mathbf{1}_{\{\overline{Y}_t\in A, \, t<\overline{\sigma}_a\}} \D t\right] &=& \int_A \left\{
W^{(q)}(x-a)\frac{W^{(q)\prime}_+(y)}{W^{(q)\prime}_+(a)} - W^{(q)}(y-x)
\right\}\D y\\
&&+\int_A \left\{
W^{(q)}(x-a)\frac{W^{(q)}(0)}{W^{(q)\prime}_+(a)} 
\right\}\delta_0(\D y).
\end{eqnarray*}

\end{enumerate}
\end{theorem}



\chapter{Further analytical properties of scale functions}\label{analytic-properties}

\section{Behaviour at 0 and $+\infty$}

Ultimately we are interested in describing the `shape' of scale functions. We start by looking at their behaviour at the origin and $+\infty$. 
In order to state the results more precisely, we recall from (\ref{BVdecomp}) that when $X$ has paths of bounded variation, we may write it in the form $X_t = \delta t - S_t$, $t\geq 0$, where $\delta>0$ and $S$ is a pure jump subordinator.

\begin{lemma}\label{W(0)}
For all $q\geq 0$, $W^{(q)}(0)=0$ if and only if $X$ has unbounded
variation. Otherwise, when $X$ has bounded variation,  $W^{(q)}(0) =1/\delta$.
\end{lemma}

\begin{proof}
Note that for all $q>0$,
\begin{eqnarray}
W^{(q)}(0) &=&\lim_{\beta   \uparrow \infty }\int_{0}^{\infty
}\beta \ \mathrm{e}^{-\beta x}W^{(q)}(x){\D}x \notag\\
&=&\lim_{\beta \uparrow \infty }\frac{\beta }{\psi \left( \beta
\right) -q}\notag\\
&=& \lim_{\beta\uparrow\infty}\frac{\beta}{\psi(\beta) }.
\label{atom?}
\end{eqnarray}
Recall the spatial Wiener-Hopf factorization of $\psi$ in (\ref{WHfact}), $$\psi(\theta)=(\theta-\Phi(0))\phi(\theta), \qquad \theta\geq 0,$$ and that $\phi$ denotes the Laplace exponent of the downward ladder height subordinator $\widehat{H}$. It follows that $$W^{(q)}(0)= \lim_{\beta\uparrow\infty}\frac{\beta}{\psi(\beta) }=\lim_{\beta\to\infty}\frac{1}{\phi(\beta)}.$$ Now, observe that $\lim_{\theta\to \infty}\phi(\beta)<\infty$, if and only if the L\'evy measure of $\widehat{H}$ is a finite measure and its drift is $0.$ We know this happens if and only if $X$ has paths of bounded variation, as we already mention that this is the only case in which $0$ is irregular for $(-\infty,0)$ and hence starting from $0$ it takes a strictly positive amount of time to enter the open lower half line. The first claim of the Theorem follows. Now, assume that $X$ has paths of bounded variation. In this case, one may use (\ref{BVdecomp}) to write more simply
\[
\psi(\beta)= \delta\beta - \int_{(0,\infty)}(1- \E^{-\beta x})\Pi(\D x).
\]    
An integration by parts tells us that
\begin{equation}\label{dividebybeta}
\frac{\psi(\beta)}{\beta} = \delta - \int_0^\infty \E^{-\beta x}\Pi(-\infty,-x)\D x,
\end{equation}
 and hence, noting that bounded variation paths (equivalently $\int_{(-1,0)}|x|\Pi(\D x)<\infty$) necessarily implies that $\int_{(0,1)} \Pi(-\infty,-x) \D x<\infty$, it follows that $\psi(\beta)/\beta\rightarrow \delta$ as $\beta\uparrow\infty$. From (\ref{atom?}) we now see that, for all $q\geq 0$,  $W^{(q)}(0) = 1/\delta$ as claimed. 
 
 Alternatively, one can prove the result making an integration by parts for $\psi$ instead of using the spatial Wiener-Hopf factorization.  
 \hfill$\square$\end{proof}


Returning to (\ref{kyprianou-palmowski:q=0}) we see that the conclusion of
the previous lemma indicates that, precisely when $X$ has bounded
variation,
\begin{equation}\label{only when bv}
P_0(\tau^+_a < \tau_0^-) = \frac{W(0)}{W(a)}>0.
\end{equation}
 Note that the stopping time $\tau^-_0$ is defined with strict entry into $(-\infty,0)$.
Hence when $X$ has the property that $0$ is irregular for
$(-\infty,0)$, it takes an almost surely positive amount of time
to exit the half line $[0,\infty)$. Since the aforementioned
irregularity is equivalent to bounded variation for this class of
L\'evy processes, we see that (\ref{only when bv}) makes sense.

Next we turn to the behaviour of the gradient of $W^{(q)}$ at $0$. 

\begin{lemma}\label{Wprime}
For all $q\geq 0$ we have
\begin{gather*}
W^{(q) \prime}(0+)=
\begin{cases} 2/\sigma^{2}, & \text{when  $\sigma\neq 0$ or $\Pi(-\infty,0) = \infty$}\\
(\Pi(-\infty, 0) + q)/\delta^2 & \text{when  $\sigma =0$ and $\Pi(-\infty,0)<\infty$,}
\end{cases}
\end{gather*}
where we understand the first case to be $+\infty$ when $\sigma = 0$.
\end{lemma}

\begin{proof}
Integrating (\ref{LTdefW}) by parts we find that for $\theta>\Phi(q)$,
\begin{equation}
W^{(q)}(0) + \int_0^\infty \E^{-\theta x} W^{(q)\prime}(x)\D x = \frac{\theta}{\psi(\theta) - q}.
\label{integratedbyparts}
\end{equation}

Now assume that $X$ has paths of unbounded variation so that $W^{(q)}(0)=0$. Next noting from Lemma \ref{a.e.diff} that a right derivative of $W^{(q)}$ at
zero always exists, we have for each $q\geq 0,$
\[
W^{(q)\prime}(0+) =
\lim_{\theta\uparrow\infty}\int_0^\infty \theta \E^{-\theta x}
W^{(q)\prime}(x)\D x=\lim_{\theta\uparrow\infty} \frac{\theta^2}{\psi(\theta)-q}.
\]
Then dividing (\ref{WHfact}) by $\theta^2$ it is easy to proved using (\ref{tailmeasure}) that  
the limit above is equal to $2/\sigma^2$ as required.
The expression for $W^{(q)\prime}(0+)$ in the first case now follows.

When $X$ has bounded
variation, a little more care is needed. Recall that  $W^{(q)}(0+)=1/\delta$, we have,
\begin{equation*}
\begin{split}
W^{(q)\prime}(0+)&\\
=&
\lim_{\beta\uparrow\infty} \frac{\beta^2}{\delta\beta - \beta\int_0^\infty \E^{-\beta x}\Pi(-\infty, -x)\D x - q} - \beta W^{(q)}(0+)\\
=&\lim_{\beta\uparrow\infty} \frac{\beta^2 \left(1 - W^{(q)}(0+) \delta + W^{(q)}(0+) \int_0^\infty \E^{-\beta x}\Pi(-\infty, -x)\D x  )\right) + q\beta  W^{(q)}(0+)}{\delta\beta - \int_0^\infty \beta \E^{-\beta x} \Pi(-\infty, -x)\D x +q}\\
=& \lim_{\beta\uparrow\infty} \frac{1}\delta\frac{  \int_0^\infty \beta \E^{-\beta x}\Pi(-\infty, -x)\D x   + q }{\delta - \int_0^\infty  \E^{-\beta x} \Pi(-\infty, -x)\D x}\\
=& \frac{\Pi(-\infty,0) + q}{\delta^2}.
\end{split}
\end{equation*}
In particular, if $\Pi(-\infty,0)=\infty,$ then the right hand side
above is equal to $\infty$ and otherwise, if $\Pi(-\infty,0)<\infty,$
then $W^{(q)\prime}(0+)$ is finite and equal to $(\Pi(-\infty,0) +
q)/\delta^2$. 
\hfill$\square$\end{proof}

Next we look at the asymptotic behaviour of the scale function at $+\infty$.

\begin{lemma}\label{lem:ZoverW}
For $q\ge0$ we have, $$\lim_{x\to\infty} \E^{-\Phi(q)x}W^{(q)}(x) = 1/\psi'(\Phi(q)),$$ and $$\lim_{x\to\infty}Z^{(q)}(x)/W^{(q)}(x)=q/\Phi(q),$$
where the right hand side above is understood in the limiting sense $\lim_{q\downarrow0}q/\Phi(q) = 0\vee(1/\psi'(0))$ when $q=0$.
\end{lemma}

\begin{proof}
For the first part, recall the identity (\ref{W interchange}) which is valid for all $q\geq 0$. It follows from (\ref{kyprianou-palmowski:forces-left-cts}) that 
\begin{equation}
W^{(q)}(x) = \E^{\Phi(q)x}\frac{1}{\psi'_{\Phi(q)}(0+)}P_x^{\Phi(q)}(\underline{X}_\infty \geq 0).
\label{interesting-representation}
\end{equation}
Appealing to (\ref{tilted-exponent}) we note that $\psi_{\Phi(q)}'(0+)= \psi'(\Phi(q))>0$ (which in particular implies that $X$ under $\mathbb{P}^{\Phi(q)}$ drifts to $+\infty$) and hence 
\[
\lim_{x\uparrow\infty}\E^{-\Phi(q)x}W^{(q)}(x) = \frac{1}{\psi'(\Phi(q))}\lim_{x\uparrow\infty}P_x^{\Phi(q)}(\underline{X}_\infty \geq 0)
= \frac{1}{\psi'(\Phi(q))}.
\]
Note that from this proof we also see that $W_{\Phi(q)}(+\infty) = 1/\psi'(\Phi(q))$.

For the second part,  it suffices to compare the identity (\ref{two-sided-down}) as $a\uparrow\infty$ against  (\ref{one-sided-down}).
 \hfill$\square$\end{proof}

\section{Concave-convex properties}\label{con-con}

Lemma \ref{lem:ZoverW} implies that when $q>0$, $W^{(q)}$ grows in a way that is asymptotically exponential and this opens the question as to whether there are any convexity properties associated with such scale functions for large values. Numerous applications have shown the need to specify more detail about the shape, and ultimately, the smoothness of scale functions. See for example the discussion in Chapter \ref{motivation}.
In this respect, there has been a recent string of articles, each one improving of the last, which has investigated concavity-convexity properties of scale functions and which are based on the following fundamental observation of Renming Song. Suppose that $\psi'(0+)\geq 0$ and that 
 $$\overline{\Pi}(x):=\Pi(-\infty,-x),$$ the density of the L\'evy measure of the descending ladder height process, see (\ref{tailmeasure}), is completely monotone.
Amongst other things, this implies that the descending ladder height process $\widehat{H}$ belongs to the class of so-called complete subordinators. 
A convenience of this class of subordinators is that  the potential measure  associated to $\widehat{H}$, which in this case is $W$, has a derivative which is is completely monotone; see for example Song and Vondra\v{c}ek \cite{SV2007}. In particular this implies that $W$ is concave (as well as being infinitely smooth). 
Loeffen \cite{Loeffen2008} pushes this idea further to the case of $W^{(q)}$ for $q>0$ as follows. (Similar ideas can also be developed from the paper of Rogers \cite{rogers83}).
\begin{theorem}\label{ronnie-result}
Suppose that $-\overline\Pi$ has a density which is completely monotone then $W^{(q)}$ has a density on $(0,\infty)$ which is strictly convex. In particular, this implies the existence of a constant  $a^*$ such that $W^{(q)}$ is strictly concave on $(0,a^*)$ and strictly convex on $(a^*,\infty)$. 
\end{theorem}

\begin{proof}
Recall from (\ref{W_P}) that we may always write $W^{(q)}(x) = \E^{\Phi(q)x}W_{\Phi(q)}(x),$
where $W_{\Phi(q)}$ plays the role of $W$ under the measure $P^{\Phi(q)}$. Recall also from the discussion under (\ref{tilted-exponent}) and (\ref{exp-shift}) that $(X,P^{\Phi(q)})$ drifts to $+\infty$, and that the L\'evy measure of $X$ under $P^{\Phi(q)}$ is given by $\Pi_{\Phi(q)}(dx)=\E^{\Phi(q)x}\Pi(dx),$ $x\in \mathbb{R}$. It follows that under the assumptions of the Theorem that the function $\Pi_{\Phi(q)}(-\infty,-x)$, $x>0$, is completely monotone. Hence the discussion preceding the statement of Theorem \ref{ronnie-result} tells us that $W'_\Phi(q)(x)$ exists and is completely monotone. According to the definition by Song and Vondra\v{c}ek \cite{SV2007}, this makes of $W_{\Phi(q)}$ a {\it Bernstein function}.

The general theory of Bernstein functions dictates that there necessarily exists a triple $(\texttt{a}, \texttt{b}, \xi)$, where $\texttt{a}, \texttt{b}\geq 0$ and $\xi$ is a measure concentrated on $(0,\infty)$ satisfying $\int_{(0,\infty)}(1\wedge t)\xi(\D t)<\infty$, such that
\[
W_{\Phi(q)}(x) = \texttt{a} + \texttt{b}x + \int_{(0,\infty)}(1- \E^{-xt})\xi(\D t).
\]
It is now a straightforward exercise to check with the help of the above identity and the Dominated Convergence Theorem that for $x>0$,
\begin{eqnarray*}
W^{(q)\prime\prime\prime}(x)& =& f'''(x) + \int_{(0,\Phi(q)]}(\Phi(q)^3\E^{\Phi(q)x} - (\Phi(q) -t)^3\E^{(\Phi(q) -t)x})\xi(\D t)\\
&+& \int_{(\Phi(q),\infty)}(\Phi(q)^3 \E^{\Phi(q)x} +(t-\Phi(q))^3 \E^{-x(t-\Phi(q))})\xi(\D t), 
\end{eqnarray*}
where $f(x) = (\texttt{a}+\texttt{b}x)\E^{\Phi(q)x}$. Hence $W^{(q)\prime\prime\prime}(x)>0$ for all $x>0$, showing that $W^{(q)\prime}$ is strictly convex on $(0,\infty)$ as required.
\hfill$\square$ 
\end{proof}

Following additional contributions in \cite{KRS}, the final word on concavity-convexity currently stands with the following theorem, taken from \cite{LR}, which overlaps all of the aforementioned results.

\begin{theorem}
Suppose that $\overline\Pi$ is log-convex. Then for all $q\geq 0$, $W^{(q)}$ has a log-convex first derivative. 
\end{theorem}

Note that the existence of a log-convex density of $-\overline{\Pi}$ implies that $\overline{\Pi}$ is log-convex and hence the latter is a weaker condition than the former. This is not an obvious statement, but it can be proved using elementary analytical arguments.

\section{Analyticity in $q$}

Let us now look at the behaviour of $W^{(q)}$ as a function in $q$. 

\begin{lemma}
\label{analytically extend}For each $x\geq 0$, the function
$q\mapsto W^{(q)}(x)$ may be analytically extended to $q\in
\mathbb{C}$.
\end{lemma}

\begin{proof}
For a fixed choice of $q>0$,
\begin{eqnarray}
\int_{0}^{\infty }\mathrm{e}^{-\beta x}W^{(q)}(x){\D}x
&=&\frac{1}{\psi \left( \beta
\right) -q}\text{ }  \notag \\
&=&\frac{1}{\psi \left( \beta \right) }\frac{1}{1-q/\psi \left(
\beta
\right) }  \notag \\
&=&\frac{1}{\psi \left( \beta \right) }\sum_{k\geq
0}q^{k}\frac{1}{\psi \left( \beta \right) ^{k}} , \label{apply
fubini}
\end{eqnarray}
for $\beta >\Phi (q).$ The latter inequality implies that
$0<q/\psi(\beta)<1$. Next note that
\begin{equation*}
\sum_{k\geq 0}q^{k}W^{\ast (k+1)}\left( x\right)
\end{equation*}
converges for each $x\geq 0$ where $W^{\ast k}$ is the $k$ th
convolution of $W$ with itself. This is easily deduced once one
has the estimates
\begin{equation}
W^{\ast (k+1)}\left( x\right) \leq \frac{x^{k}}{k!}W\left(
x\right)^{k+1} \label{induction}, \qquad x\geq0;
\end{equation}
which itself can easily be proved by induction. Indeed note that
if (\ref {induction}) holds for $k\geq 0,$ then by monotonicity of
$W$,
\begin{eqnarray*}
W^{\ast (k+1)}\left( x\right)  &\leq &\int_{0}^{x}\frac{y^{k-1}}{(k-1)!}%
W\left( y\right)^{k}W\left( x-y\right) {\D}y \\
&\leq &\frac{1}{(k-1)!}W\left( x\right)^{k+1}\int_{0}^{x}y^{k-1}\ {\D}y \\
&=&\frac{x^{k}}{k!}W\left( x\right)^{k+1}.
\end{eqnarray*}
Returning to (\ref{apply fubini}) we may now apply Fubini's
Theorem (justified by the assumption that $\beta>\Phi(q)$) and
deduce that
\begin{eqnarray*}
\int_{0}^{\infty }\mathrm{e}^{-\beta x}W^{(q)}(x){\D}x &=&\sum_{k\geq 0}q^{k}\frac{1}{%
\psi \left( \beta \right) ^{k+1}} \\
&=&\sum_{k\geq 0}q^{k}\int_{0}^{\infty }\mathrm{e}^{-\beta
x}W^{\ast (k+1)}\left(
x\right) {\D}x \\
&=&\int_{0}^{\infty }\mathrm{e}^{-\beta x}\sum_{k\geq
0}q^{k}W^{\ast (k+1)}\left( x\right) {\D}x.
\end{eqnarray*}
Thanks to continuity of $W$ and $W^{(q)}$ we have that
\begin{equation}
W^{(q)}(x)=\sum_{k\geq 0}q^{k}W^{\ast (k+1)}\left( x\right) , \,\, x\in \mathbb{R}.
\label{analytic}
\end{equation}
Now noting that $\sum_{k\geq 0}q^{k}W^{\ast (k+1)}\left( x\right)
$ converges for all $q\in \mathbb{C}$ we may extend the definition
of $W^{(q)}$ for each fixed $x\geq 0,$ by the equality given in
(\ref{analytic}). \hfill$\square$\end{proof}

For each $c\geq 0$, denote by  $W_{c}^{(q)}$ the $q$-scale function for $(X, P^c)$. The
previous Lemma allows us to establish the following relationship for $%
W_{c}^{(q)}$ with different values of $q$ and $c.$

\begin{lemma}\label{LEMMA-general-W-rel}
For any $q\in \mathbb{C}$ and $c\in \mathbb{R}$ such that $\psi
(c)<\infty $ we have
\begin{equation}
W^{(q)}(x)=\mathrm{e}^{cx}W_{c}^{(q-\psi (c))}(x),
\label{general-W-rel}
\end{equation}
for all $x\geq 0.$
\end{lemma}

\begin{proof}
For a given $c\in \mathbb{R},$ such that $\psi (c)<\infty,$ the
identity (\ref {general-W-rel}) holds for $q-\psi(c)\geq 0,$ on
account of both left and right-hand side being continuous
functions with the same Laplace transform. By Lemma
\ref{analytically extend} both left- and right-hand side of (\ref
{general-W-rel}) are analytic in $q$ for each fixed $x\geq 0.$ The
Identity
Theorem for analytic functions thus implies that they are equal for all $%
q\in \mathbb{C}$. \hfill$\square$\end{proof}

 Unfortunately a convenient relation such as (\ref{general-W-rel}) cannot be given  for $Z^{(q)}.$ Nonetheless we do have the following obvious corollary.

\begin{corollary}
For each $x>0$ the function $q\mapsto Z^{(q)}(x)$ may be
analytically extended to $q\in \mathbb{C}$.
\end{corollary}

The above results allow one to push some of the identities in Section \ref
{section:fluctuation identities} further by applying an exponential change of measure. In principle this allows one to gain distribution information about the position of the L\'evy process at first passage.  We give one example here but the reader can easily explore other possibilities. 

Consider the first passage identity in (\ref{one-sided-down}).
Suppose that $v\geq 0,$ then $|\psi(v)|<\infty,$ and assume it satisfies $u>\psi(v)\vee 0$. Then with the help of the aforementioned identity together with the change of measure (\ref{COM}) we have for  $x\in \mathbb{R}$,
\begin{eqnarray*}
E_x\left(\E^{-u\tau_0^- + vX_{\tau^-_0}}\mathbf{1}_{(\tau_0^- <\infty)}\right)
&=&E^v_x\left(\E^{-(u-\psi(v))\tau_0^-}\mathbf{1}_{(\tau_0^- <\infty)}\right)\\
&=&\E^{vx}\left(Z_v^{(p)}(x) -\frac{p}{\Phi_v(p)}W_v^{(p)}(x)\right),
\end{eqnarray*}
where $p = u-\psi(v)>0$. We can develop the right hand side further using   the relationship  between scale functions given in Lemma \ref{LEMMA-general-W-rel} as well as by noting   that 
\begin{eqnarray*}
\Phi_v(p) &=& \sup\{\lambda \geq 0: \psi_v(\lambda) = p\}\\
&=&  \sup\{\lambda \geq 0: \psi(\lambda + v)- \psi(v) = u - \psi(v)\}\\
&=&\sup\{\theta \geq 0: \psi(\theta) = u \} -v\\
&=& \Phi(u)-v.
\end{eqnarray*}
Hence we have that 
\begin{eqnarray}
\lefteqn{E_x\left(\E^{-u\tau_0^- + vX_{\tau^-_0}}\mathbf{1}_{(\tau_0^- <\infty)}\right)}&&\notag\\
&&=\E^{vx}\left(1 +(u - \psi(v))\int_0^x \E^{-v y}W^{(u)}(y){\rm d}y -\frac{u- \psi(v)}{\Phi(u)-v}\E^{-vx}W^{(u)}(x)\right).
\label{extendme}
\end{eqnarray}
Clearly the restriction that $u>\psi(v)$ is an unnecessary constraint for both the left and right hand side of the above equality to be finite and it would suffice that $u,v\geq 0$. In particular for the right hand side, when $u = \psi(v)$ it follows that $\Phi(u) = v$ and hence the ratio $(u-\psi(v))/(\Phi(u)-v)$ should be understood in the limiting sense. That is to say,
\[
\lim_{p\rightarrow0}\frac{p}{\Phi_v(p)} = \lim_{p\rightarrow0}\frac{\psi_v(\Phi_v(p))}{\Phi_v(p)} = \psi'_v(0+) = \psi'(v),
\]
where we have used the fact that $\Phi_v(0) = 0$ as $\psi_v'(0+) = \psi'(v)>0$. 

Note that  the left hand side of (\ref{extendme}) is  analytic for complex  $u$  with a strictly positive real part. 
Moreover, thanks to Lemma \ref{analytically extend} and the fact that $\Phi(u)$ is a Laplace exponent  (cf. the last equality of (\ref{firstpassagesub})) it is also clear that the right hand side can be analytically extended to allow for complex-valued $u$ with strictly positive real part. Hence once again the Identity Theorem allows us to extend the equality (\ref{extendme}) to allow for the case that $u>0$. The case that $u=0$ can be established by taking limits as $u\downarrow 0$ on both sides. 

Careful inspection of the above argument shows that one may even relax the constraint on $v\geq 0$  to simply any $v$ such that $|\psi(v)|<\infty$ in the case that the exponent $\psi(v)$ is finite for negative values of $v$.


\section{Spectral gap}

Bertoin \cite{ber97} showed an important consequence of the analytic nature of scale function $W^{(q)}$ in its argument $q$.

For $a>0$, let  $\tau: = \tau^+_a\wedge\tau^-_0$, then \cite{ber97} investigates ergodicity properties of the spectrally negative L\'evy process $X$ on exiting $[0,a]$. The main object of concern is the killed transition kernel
\[
P(x, t, A) = P_x(X_t \in A; t< \tau).
\]

\begin{theorem}
Define
\[
\rho = \inf\{q \geq 0 : W^{(-q)}(a) = 0\}.
\]
Then $\rho$ is finite and positive, and for any $q<\rho$ and $x\in(0,a)$, $W^{(-q)}(x) > 0$. Furthermore,
the following assertions hold,
\begin{enumerate}
\item[(i)] $\rho$ is a simple root of the entire function $q \mapsto W^{(-q)}(a)$,
\item[(ii)] The function $W^{(-\rho)}$ is positive on $(0,a)$ and is $\rho$-invariant for $P(x, t, \cdot)$ in the sense that 
\[
\int_{[0,a]}P(x,t,{\rm d}y)W^{(-\rho)}(y)=\E^{-\rho t}W^{(-\rho)}(x),\text{ for any }x\in(0,a),
\]
\item[(iii)] the measure $W^{(-\rho)}(a-x){\rm d}x $ on $[0,a]$ is $\rho$-invariant in the sense that
\[
\int_{[0,a]} {\rm d}y \cdot W^{(-\rho)}(a-y)P(y,t,{\rm d}x) = \E^{-\rho t}W^{(-\rho)}(a-x){\rm d}x,
\]
\item[(iv)] There is a constant $c>0$ such that, for any $x\in(0,a)$,
\[
\lim_{t\uparrow\infty}
\E^{\rho t}P(x,t, {\rm d}y) = cW^{(-\rho)}(x) W^{(-\rho)}(a-y){\rm d}y,
\]
in the sense of weak convergence.
\end{enumerate}
\end{theorem}

A particular consequence of the part (iv) of the previous theorem is that there exists a constant $c'>0$ such that 
\[
P_x(\tau>t) \sim c'W^{(-\rho)}(x)\E^{-\rho t}.
\]
as $t\uparrow\infty$. The constant $\rho$ thus describes the rate of decay of the exit probability. By analogy with the theory of diffusions confined to compact domains, $-\rho$ also plays the role of the leading eigen-value, or spectral gap, of the infinitesimal generator of $X$ constrained to the interval $(0,a)$ with Dirichlet boundary conditions. From part (i) of the above theorem, we see that the associated eigen-function is $W^{(-\rho)}$.

Lambert \cite{lambert} strengthens this analogy with diffusions and showed further that  for each $x\in (0,a)$,
\[
\E^{\rho t} \frac{W^{(-\rho)}(X_t)}{W^{(-\rho)}(x)}\mathbf{1}_{\{\tau>t\}},\, t\geq 0
\]
is a $P_x$ martingale which, when used as a Radon-Nikodim density to change measure, induces a new probability measure, say, $P^\updownarrow_x$. He shows moreover that this measure corresponds to the law of $X$ conditioned to remain in $(0,a)$ in the sense that for all $t\geq 0$, 
\[
P_x^\updownarrow(A) = \lim_{s\uparrow\infty}P_x(A|\tau>t+s), \,\,\,\, A\in\mathcal{F}_t.
\]

The role of the scale function is no less important in other types of related conditioning. We have already alluded to its relevance to conditioning spectrally negative L\'evy processes to stay positive in the first chapter. Pistorius \cite{Pistorius, pistorius2003} also shows that a similar agenda to the above can be carried out with regard to reflected spectrally negative L\'evy processes. 


\section{General smoothness and Doney's Conjecture}

Let $a>0,$ and recall $\tau=\tau^{+}_{a}\wedge \tau^{-}_{0}.$ It is not difficult to show from the identity (\ref{whatsinaname}) that for all $q\geq 0$, 
\[
\E^{-q(t\wedge \tau)}W^{(q)}(X_{t\wedge \tau}) ,\, t\geq 0
\]
is a martingale. Indeed,  we have from the
Strong Markov Property that for all $x\in \mathbb{R}$ and $t\geq
0$
\begin{eqnarray*}
E_x(\mathrm{e}^{-q\tau^+_a}\mathbf{1}_{(\tau^-_0 >
\tau^+_a)} | \mathcal{F}_{t\wedge\tau^-_0 \wedge \tau^+_a}) &=&
\E^{-q t}\mathbf{1}_{\{t<\tau\}}E_{X_t}(\mathrm{e}^{-q\tau^+_a}\mathbf{1}_{(\tau^-_0 >
\tau^+_a)} ) + \E^{-q\tau^+_a}\mathbf{1}_{\{\tau^+_a<t\wedge \tau^-_0\}}\\
&=&
\mathrm{e}^{-q(t\wedge\tau^-_0 \wedge \tau^+_a)} \frac{ W^{(q)}(
X_{t\wedge\tau^-_0 \wedge \tau^+_a} ) }{ W^{(q)}(a) },
\end{eqnarray*}
where it should be noted that for the second equality we have used the fact that  $W^{(q)}(X_{\tau^+_a})/W^{(q)}(a)=1$
and $W^{(q)}(X_{\tau^-_0})/W^{(q)}(a)=0$. Note in particular, for the last equality, $X$ creeps downwards if and only if it has a Gaussian component in which case $W^{(q)}(X_{\tau^-_0}) = W^{(q)}(0+) = 0,$ on the event of creeping, and otherwise on the event that  $X_{\tau^-_0}<0,$
it trivially holds that $W^{(q)}(X_{\tau^-_0}) =0$.

Naively speaking this reiterates an idea seen in the previous section that $W^{(q)}$ is an eigen-function with respect to the infinitesimal generator of $X$ with eigen-value $q$. That is to say, $W^{(q)}$ solves the integro-differential equation 
\[
(\Gamma -q)W^{(q)}(x)=0 \text{ on }(0,a),
\]
where $\Gamma$ is the infinitesimal generator of $X,$ which is known to be given by
\[
\Gamma f(x) = \mu f'(x) + \frac{1}{2}\sigma^2 f''(x) + \int_{(-\infty,0)} \{f(x+y) - f(x) - f'(x)y\mathbf{1}_{\{y>-1\}}\}\Pi(\D y),
\] for all $f$ in its domain.

The problem with the above heuristic observation is that $W^{(q)}$ may not be in the domain of $\Gamma$. In particular, for the last equality to have a classical meaning we need at least that $f\in C^2(0,a)$ and $\int_{(-\infty,0)}f(x+y)\Pi(\D y)<\infty$.
This  motivates the question of how smooth scale functions are.  Given the dependency of concavity-convexity properties on the L\'evy measure discussed in Section \ref{con-con} as well as the statement of Lemma \ref{C1}, it would seem sensible to believe that a relationship exists between the smoothness of the L\'evy measure and the smoothness of the scale function $W^{(q)}$. In addition, one would also expect the inclusion of a Gaussian coefficient to have some effect on the smoothness of the scale function. Below we give a string of recent results which have attempted to address this matter. For notational convenience we shall write $W^{(q)}\in C^{k}(0,\infty)$ to mean that the restriction of $W^{(q)}$ to $(0,\infty)$ belongs to the $C^k(0,\infty)$ class.

\begin{theorem}\label{C2}
Suppose that $\sigma^2>0$, then for all $q\geq 0$, $W^{(q)}\in C^2(0,\infty)$.
\end{theorem}
\begin{proof} As before, it is enough to consider the case where $q=0$ and $\psi^{\prime}(0+)\geq 0.$ In other case we use the latter with the help of (\ref{W_P}). Note that it suffices to prove the result for $q=0$ thanks to (\ref{W_P}) and (\ref{tilted-exponent}). We know from Lemma \ref{C1} that when $\sigma^2>0$, $W^{(q)}\in C^1(0,\infty)$.
Recall from (\ref{landrder}) that $W^{\prime}(x)=n(\overline\epsilon \geq
x )W(x)$ and hence if the limits exist, then
\begin{eqnarray}
W^{\prime\prime}(x+) &=& \lim_{\varepsilon\downarrow 0}\frac{%
W^{\prime}(x+\varepsilon) - W^{\prime}(x)}{\varepsilon}  \notag \\
&=&-\lim_{\varepsilon\downarrow 0} \frac{ n(\overline\epsilon \in
[x,x+\varepsilon) ) W(x )- n(\overline\epsilon \geq x + \varepsilon) (W(x +
\varepsilon)- W(x))}{\varepsilon}  \notag \\
&=&-\lim_{\varepsilon\downarrow 0}\frac{n(\overline\epsilon \in
[x,x+\varepsilon) )}{\varepsilon} W(x) + n(\overline\epsilon \geq
x)W^{\prime}(x+ ).  \label{first step}
\end{eqnarray}

To show that the limit on the right hand side of (\ref{first step}) exists,
define $\sigma_x = \inf\{t>0 : \epsilon_t \geq x\}$ and $\mathcal{G}_t =
\sigma(\epsilon_s : s\leq t)$. With the help of the Strong Markov Property
for the excursion process we may write
\begin{eqnarray}
n(\overline\epsilon \in [x,x+\varepsilon) )&=&n(\sigma_x<\infty,
\epsilon_{\sigma_x}< x+\varepsilon, \overline\epsilon < x+\varepsilon) 
\notag \\
&=&n( \mathbf{1}_{\{\sigma_x<\infty, \epsilon_{\sigma_x}<
x+\varepsilon\}}n(\overline\epsilon < x+\varepsilon| \mathcal{G}_{\sigma_x}))
\notag \\
&=&n( \mathbf{1}_{\{\sigma_x<\infty, \epsilon_{\sigma_x}<
x+\varepsilon\}}P_{-\epsilon_{\sigma_x}}(\tau^+_0 <
\tau^-_{-(x+\varepsilon)}))  \notag \\
&=& n(\sigma_x<\infty, \epsilon_{\sigma_x}= x)\frac{W(\epsilon)}{%
W(x+\varepsilon)}  \notag \\
&& + n\left( \mathbf{1}_{\{\sigma_x<\infty,x< \epsilon_{\sigma_x}<
x+\varepsilon\}}\frac{W(x+\varepsilon - \epsilon_{\sigma_x})}{W(x+\varepsilon)%
}\right).  \label{see creep}
\end{eqnarray}

We know that spectrally negative L\'evy processes which have a Gaussian component can creep downwards. Hence for the case at hand it follows that the event $%
\{\epsilon_{\sigma_x}=x\}$ has non-zero $n$-measure.

Using the facts that $W'(0+) = 2/\sigma^2$ and $W(0+) =0$ (cf. Lemmas \ref{W(0)} and \ref{Wprime}) together with the monotonicity of $W$, we have that
\begin{eqnarray}
\lefteqn{\limsup_{\varepsilon\downarrow 0}\frac{1}{\varepsilon}n\left( \mathbf{1}%
_{(\overline\epsilon \geq x, \epsilon_{\sigma_x}\in(x, x+\varepsilon))}\frac{%
W(x+\varepsilon - \epsilon_{\sigma_x})}{W(x+\varepsilon)}\right)}  \notag \\
&\leq& \frac{1}{W(x)} \limsup_{\varepsilon\downarrow 0}n(\overline\epsilon
\geq x, \epsilon_{\sigma_x}\in(x,x+\varepsilon))\frac{W(\varepsilon)}{%
\varepsilon}  \label{non-creep} \\
& =&0.  \notag
\end{eqnarray}

In conclusion, $W^{\prime\prime}(x+)$ exists and 
\begin{equation*}
W^{\prime\prime}(x+)=-W^{\prime}(0+)n(\overline\epsilon \geq x,
\epsilon_{\sigma_x}= x)+n(\overline\epsilon \geq x)W^{\prime}(x),
\end{equation*}
that is to say,
\begin{equation}
n(\overline\epsilon \geq x,
\epsilon_{\sigma_x}= x)=\frac{\sigma ^{2}}{2}\left\{\frac{W^{\prime
 }(x)^{2}}{W(x)}- W^{\prime \prime }(x+)\right\}.
 \label{W-doubledash}
\end{equation}

 We shall now show that there exists a left second derivative $W^{\prime
 \prime }(x-)$ which may replace the role of $W''(x+)$ on the right hand side of (\ref{W-doubledash}) thus
 completing the proof. For this we need to  recall from Doney's description of the excursion measure as limit (cf. \cite{Doney2004}) 
 that for $A\in \mathcal{F}_{t}$
\begin{equation*}
n(A,t<\zeta )=\lim_{y\downarrow 0}\frac{\widehat{ P}_{y}(A,t<\tau
_{0}^{-})}{y},
\end{equation*}
and moreover this identity may be extended to stopping times.
From this we may write 
 \begin{equation}
 n(\overline{\epsilon }\geq x,\epsilon _{\sigma _{x}}=x)=\lim_{y\downarrow 0}%
  \frac{\widehat{ P}_{y}(X_{\tau _{x}^{+}}=x;\tau _{x}^{+}<\tau _{0}^{-})}{y},%
   \label{radlimit}
 \end{equation}%
 where $\widehat{ P}_y$ is the law of $-X$ when issued from $y$.
 From part (ii) of Theorem \ref{oneandtwo}  we also know that 
 \begin{equation*}
 \widehat{ P}_{y}(X_{\tau _{x}^{+}}=x;\tau _{x}^{+}<\infty )=\frac{\sigma
 ^{2}}{2} W^{\prime }(x-y).
 \end{equation*}%
  Using the Strong Markov Property, the above formula and the fact that $X$ creeps upwards it is straightforward to deduce that 
\[
 \widehat{ P}_{y}(X_{\tau _{x}^{+}}=x;\tau _{0}^{-}< \tau _{x}^{+})
=\frac{W(x-y)}{W(x)} \times \frac{\sigma ^{2}}{2}W^{\prime }(x)
\]
and hence
 \[
 \widehat{ P}_{y}(X_{\tau _{x}^{+}}=x;\tau _{x}^{+}<\tau _{0}^{-}) 
 =\frac{\sigma ^{2}}{2}\left\{ W^{\prime }(x-y)-\frac{W(x-y)}{W(x)}%
 W^{\prime }(x)\right\} .
 \]%
 Returning to (\ref{radlimit}) we compute, 
 \begin{eqnarray*}
 n(\overline{\epsilon }\geq x,\epsilon _{\sigma _{x}}=x) &=&\lim_{y\downarrow
 0}\frac{\sigma ^{2}}{2}\left\{\frac{W^{\prime }(x)}{W(x)}\frac{W(x)-W(x-y)}{y}-\frac{W^{\prime }(x)-W^{\prime }(x-y)}{y}\right\}  \\
 &=&\frac{\sigma ^{2}}{2}\left\{\frac{W^{\prime
 }(x)^{2}}{W(x)}- W^{\prime \prime }(x-)\right\} .
 \end{eqnarray*}%
 Note that the existence of $W^{\prime \prime }(x-)$ is guaranteed in light
 of (\ref{radlimit}). We see
 that the final equality above is identical to (\ref{W-doubledash}) but with $%
 W^{\prime \prime }(x+)$ replaced by $W^{\prime \prime }(x-)$.

 Thus far we have shown that a second derivative exists everywhere. To complete the proof, we need to show that this second derivative is continuous. To do this, it suffices to show that $n(\overline{\epsilon }\geq x,\epsilon _{\sigma _{x}}=x)$ is continuous. To this end note that a straightforward computation, similar to (\ref{radlimit}) but making use of Part (i) of Theorem \ref{resolve} and L'H\^opital's rule to compute the relevant limit, shows that
\[
n(\overline{\epsilon }\geq x,\epsilon _{\sigma _{x}}>x) = 
 \int_0^x \left\{W'(x-y) - \frac{W'(x)}{W(x)}W(x-y)\right\}\overline{\Pi}(y){\rm d}y,
\]
where we recall that $\overline\Pi(y)  = \Pi(-\infty,-y)$.
Hence it is now easy to see that 
\begin{eqnarray*}
 n(\overline{\epsilon }\geq x,\epsilon _{\sigma _{x}}=x)  &=&n(\overline{\epsilon }\geq x)-n(\overline{\epsilon }\geq x,\epsilon _{\sigma _{x}}>x)
\end{eqnarray*}
is continuous, thus completing the proof.\hfill$\square$
\end{proof}

Chan et al. \cite{CKS} take a more analytical approach to the smoothness of scale functions by exploring its connection with the classical renewal equation. They note that, on account of (\ref{W_P}), it suffices to consider smoothness in the case that $\psi'(0+)\geq 0$ (i.e. $\Phi(0) = 0$). Within this regime, the basic idea is to start with the obvious statement that under the measure $P_x$ the L\'evy process will either cross downwards below zero by either creeping across it or jumping clear below it. Taking account of (\ref{one sided down q=0}), (\ref{creep-q}) and (\ref{return to,}) in its limiting form as $a\uparrow\infty$, we thus have
\begin{eqnarray}
1-\frac{1}{\psi'(0+)}W(x) &=&P_x(\tau^-_0<\infty) \notag\\
&=& P_x (X_{\tau^-_0} = 0) 
+ P_x(X_{\tau^-_0}<0)\notag\\
&=& \frac{\sigma^2}{2}W'(x) + \int_0^\infty \{W(x) - W(x-y)\} \overline{\Pi}(y)\D y\notag\\
&=& \frac{\sigma^2}{2}W'(x) + \int_{0}^x W'(z) \overline{\overline{\Pi}}(x-z)\D z\notag\\
&=& \frac{\sigma^2}{2}W'(x) + \int_{0}^x W'(x-z) \overline{\overline{\Pi}}(z)\D z,
\label{pre-renewal}
\end{eqnarray}  
where $\overline{\overline{\Pi}}(z) = \int_z^\infty \overline{\Pi}(y)\D y$. Hence, 
\begin{equation}
1 = \frac{\sigma^2}{2}W'(x) + \int^x_0 W'(x-z) \{\overline{\Pi}(z) +\frac{1}{\psi'(0+)}\}\D z,
\label{post-pre-renewal}
\end{equation}
and after a little manipulation, we arrive at the classical  renewal equation 
\[
f= 1 + f\star g,
\]
where $f(x) = \sigma^2W'(x)/2,$ and $g(x) = -2\sigma^{-2}\int_0^x\overline{\overline{\Pi}}(y)\D y - 2\sigma^{-2}\psi'(0+)x,$ and momentarily we have assumed that $\sigma^2>0$. By engaging with the well known convolution series solution to the renewal equation they establish the following result beyond the scope of Theorem \ref{C2}.

\begin{theorem}\label{II}
 Suppose that $X$ has a Gaussian component and its Blumenthal-Getoor index belongs to  $[0,2)$, that is to say
 \[
 \inf\{\beta\geq0: \int_{|x|<1} |x|^\beta \Pi({\rm d}x)<\infty\}\in[0,2).
 \]
 Then for each $q\geq 0$ and $n=0,1,2,...$, $%
W^{(q)}\in C^{n+3}(0,\infty)$ if and only if $\overline{\Pi}\in C^{n}(0,\infty)$.
\end{theorem}

\noindent In fact the method used to establish the above theorem can also be used to prove Theorem \ref{C2}. It is also apparent from the analysis of \cite{CKS} that their method cannot be applied when considering renewal equations of the form $1 = f\star g$, which would be the form of the resulting renewal equation in (\ref{post-pre-renewal}) when $\sigma = 0$ with an appropriate choice of $f$ and $g$. 

However when   $X$ does have paths of bounded variation, the aforementioned is possible following an integration by parts and a little algebra in the convolution on the right hand side of (\ref{pre-renewal}). In that case we take $f = \delta W(x)$ and $g(x) = \delta^{-1}\int_0^x \overline{\Pi}(y)\D y$ where we recall that $W(0+) = \delta^{-1} $ and $\delta$ is the coefficient of the linear drift when $X$ is written according to the decomposition (\ref{BVdecomp}). Note that this integration by parts is not possible if $X$ has paths of unbounded variation with no Gaussian component as  the aforementioned choice of $g$ is not finite. 

The same analysis for the Gaussian case but now for the bounded variation setting in \cite{CKS} yields the following result.

 \begin{theorem}\label{III}
 Suppose that $X$ has paths of bounded variation and $-\overline{\Pi}$ has a density $\pi(x)$, such that   $\pi(x)\leq C|x|^{-1-\alpha}$ in the neighbourhood of the origin, for some $\alpha<1$ and $C>0$. Then for each $q\geq 0$ and $n=1,2,...$, $%
W^{(q)}\in C^{n+1}(0,\infty)$ if and only if $\overline{\Pi}\in C^{n}(0,\infty)$.\end{theorem}

Essentially the results of \cite{CKS} are primarily results about the smoothness of the solutions to renewal (or indeed Volterra) equations which are then applied  where possible to the particular setting of scale functions. Unfortunately this means that nothing has been said about the case that $X$ has paths of unbounded variation and $\sigma=0$ to date. Moreover, the results in the last two theorems above are subject to conditions  which do not appear to be probabilistically natural other in so much as providing the technical basis with which to push through their respective analytical proofs. None the less they go part way to addressing {\it Doney's conjecture for scale functions}\footnote{Curious about the results in \cite{CKS}, R.A.Doney produced a number of specific examples  of L\'evy processes whose scale functions exhibited analytical behaviour that lead to his conjecture (personal communication with A.E.Kyprianou).} as follows.

 \begin{conjecture}\label{DC}
 For $k=0,1,2,\cdots$
 \begin{enumerate}
 \item if $\sigma^2 > 0$  then 
\[
W\in C^{k+3}(0,\infty)\Leftrightarrow \overline\Pi\in C^{k}(0,\infty),
\]

\item if $\sigma=0$ and $\int_{(-1,0)}|x|\Pi(\D x)=\infty$ then 
\[
W\in C^{k+2}(0,\infty)\Leftrightarrow \overline\Pi\in C^{k}(0,\infty),
\]

\item if $\sigma=0$ and $\int_{(-1,0)}|x|\Pi(\D x)<\infty$ then 
\[
W\in C^{k+1}(0,\infty)\Leftrightarrow \overline\Pi\in C^{k}(0,\infty).
\]
\end{enumerate}
 \end{conjecture}

On a final note we mention that D\"oring and Savov \cite{Doring_Savov} have obtained further results concerning smoothness for potential measures of subordinators which has implications for the smoothness of scale functions. A particular result of interest in their paper is understanding where smoothness breaks down when atoms are introduced into the L\'evy measure.

\chapter{Engineering scale functions}

\section{Construction through the Wiener-Hopf factorization}

Recall that the spatial Wiener-Hopf factorization can be expressed through the identity
\begin{equation}
\psi(\lambda) = (\lambda-\Phi(0))\phi(\lambda),
\label{allstemsfromthis}
\end{equation}
for $\lambda\geq 0$
where $\phi(\lambda)$ is the Laplace exponent of the descending ladder height process. The killing rate, drift coefficient and L\'evy measure associated with $\phi$ are given by $ \psi'(0+)\vee 0$, $\sigma^2/2$ and 
\[
\Upsilon(x,\infty) = \E^{\Phi(0)x}\int_x^\infty \E^{-\Phi(0)u}\Pi(-\infty,-u){\rm d} u,\quad \text{for}\ x>0,
\]
respectively.
Moreover, from this decomposition one may derive that 
\begin{equation}
W(x)=\E^{\Phi(0)x}\int^x_{0}\E^{-\Phi(0)y}\int^{\infty}_{0}d t\cdot P(\widehat{H}_t \in {\rm d} y),\qquad x\geq 0.
\label{w-oo}
\end{equation}

This relationship between scale functions and potential measures of subordinators lies at the heart of the approach we shall describe in this section. Key to the method is the fact that  one can find in the literature several subordinators for which the potential measure is known explicitly. Should these subordinators turn out to be the descending ladder height process of a spectrally negative L\'evy process which does not drift to $-\infty$, i.e. $\Phi(0) =0$, then this would give an exact expression for its scale function.
Said another way, we can build scale functions using the following approach.

\bigskip

\noindent\textit{Step 1.} Choose a subordinator, say $\widehat{H},$ with Laplace exponent $\phi,$ for which one knows its potential measure or equivalently, in light of (\ref{equiv}), one can explicitly invert the Laplace transform $1/\phi(\theta)$.

\bigskip

\noindent \textit{Step 2.} Verify whether the relation $$\psi(\lambda):=\lambda\phi(\lambda),\qquad \lambda\geq 0,$$ defines the Laplace exponent of a spectrally negative L\'evy process. 

\bigskip

Of course, for this method to be useful we should first provide necessary and sufficient conditions for a subordinator to be the downward ladder height process of some spectrally negative L\'evy process or equivalently a verification method for the Step 2.

The following theorem, taken from Hubalek and Kyprianou \cite{HK2007} (see also Vigon \cite{Vig2002}), shows how one may identify  a spectrally negative L\'evy process $X$ (called the {\it parent process}) for a given descending ladder height process $\widehat{H}$.
The proof follows by a straightforward manipulation of the Wiener-Hopf factorization (\ref{allstemsfromthis}).

\begin{theorem}\label{th:HK08}
Suppose that $\widehat H$ is a subordinator, killed at rate $\kappa\geq 0$, with drift $\delta$ and
L\'evy measure $\Upsilon$ which is absolutely continuous with non-increasing density. Suppose further that $\varphi\geq 0$ is 
given such that $\varphi\kappa=0$. Then
there exists a spectrally negative L\'evy process  $X$, henceforth referred to as the `parent process', 
such that for all $x\geq 0$, $\mathbb{P}(\tau^+_x <\infty)=\E^{-\varphi x}$ and whose descending ladder 
height process is precisely the process $\widehat H$. The L\'evy triple $(a,\sigma, \Pi)$ of the parent process 
is uniquely identified as follows.
The Gaussian coefficient is given by $\sigma = \sqrt{2\delta}$. The L\'evy measure is given by 
\begin{equation}
\Pi(-\infty, -x) = \varphi\Upsilon(x,\infty)+ \frac{ {\rm d} \Upsilon}{\D x}(x).
\label{density}
\end{equation}
Finally 
\begin{equation}
a =\int_{(-\infty, -1)} x\Pi(\D x)  - \kappa
\end{equation}
if $\varphi=0$ and otherwise when $\varphi>0$ 
\begin{equation}
a = \frac{1}{2}\sigma^2\varphi + \frac{1}{\varphi}\int_{(-\infty, 0)}(\E^{\varphi x} -1- x\varphi\mathbf{1}_{\{x>-1\}})\Pi(\D x).
\end{equation}
In all cases, the Laplace exponent of the parent process is also given by
\begin{equation}
\psi(\theta) = (\theta-\varphi)\phi(\theta),
\label{WHwithupkilling}
\end{equation}
for $\theta \geq 0$ where $\phi(\theta) = - \log \mathbb{E}(\E^{-\theta \widehat{H}_1})$.

Conversely, the killing rate, drift and L\'evy measure of the descending ladder height process associated to a given spectrally negative L\'evy process $X$ satisfying $\varphi = \Phi(0)$ are also given by the above formulae.
\end{theorem}
   
Note that when describing parent processes later on in this text, for practical reasons we shall prefer to specify the triple $(\sigma, \Pi, \psi)$ instead of $(a,\sigma, \Pi)$. However both triples provide an equivalent amount of information. It is also worth making an observation for later reference concerning the path variation of the 
process $X$ for a given a descending ladder height process $H$.

\begin{corollary}\label{variation?}
 Given a killed subordinator $H$ satisfying the conditions of the previous Theorem, 
\begin{description}
 \item[(i)]  the parent process has paths of unbounded variation if and only if $\Upsilon(0,\infty)=\infty$ or 
 $\delta>0$,
\item[(ii)] if $\Upsilon(0,\infty)=\lambda <\infty$ then the parent process necessarily decomposes in the form 
\begin{equation}\label{driftterm}
X_t=(\kappa+\lambda-\delta\varphi)t+\sqrt{2\delta}B_t-S_t,
\end{equation}
where $B=\{B_t: t\geq 0\}$ is a Brownian motion, $S=\{S_t : t\geq 0\}$ is an independent driftless 
subordinator with L\'evy measure $\nu$ satisfying $\nu(x,\infty) = \Pi(-\infty, -x)$.
\end{description}
\end{corollary}

\begin{proof} The path variation of $X$ follows directly from (\ref{density}) and the fact that $\sigma=\sqrt{2\delta}$. Also using (\ref{density}), the Laplace exponent of the decomposition (\ref{driftterm}) can be computed as follows with the help of an integration by parts;
\begin{eqnarray*}
\lefteqn{(\kappa + \lambda -\delta\varphi)\theta + \delta\theta^2 - \varphi\theta\int_0^\infty \E^{-\theta x}\Upsilon(x,\infty)\D x - \theta\int_0^\infty \E^{-\theta x}\frac{{\rm d}\Upsilon}{\D x}(x)\D x}&&\\
&=& (\kappa + \Upsilon(0,\infty) -\delta\varphi)\theta + \delta\theta^2 - \varphi\int_0^\infty (1-\E^{-\theta x})\frac{{\rm d}\Upsilon}{\D x}(x)\D x - \theta\int_0^\infty \E^{-\theta x}\frac{{\rm d}\Upsilon}{\D x}(x)\D x\\
&=&(\theta-\varphi)\left(\kappa + \delta\theta + \int_0^\infty (1-\E^{-\theta x})\frac{{\rm d}\Upsilon}{\D x}(x)\D x\right).
\end{eqnarray*}
This agrees with the Laplace exponent $\psi(\theta)=(\theta-\varphi)\phi(\theta)$ of the parent process constructed in Theorem \ref{th:HK08}. 
\hfill$\square$\end{proof}

\begin{example} 
 Consider a spectrally negative L\'evy process which is the parent process of a (killed) tempered stable process. That is to say a subordinator with Laplace exponent given by
\begin{equation*}\label{kTS}
\phi(\theta)=\kappa+c\Gamma(-\alpha)
\left((\gamma+\theta)^\alpha-\gamma^\alpha\right),
\end{equation*}
where $\alpha\in(-1,1)\setminus\{0\}$, $\gamma\geq 0$ and $c>0$. For $\alpha, \beta>0,$ we will denote by
\[
\mathcal{E}_{\alpha,\beta}(x) = \sum_{n\geq 0}\frac{x^n}{\Gamma(n\alpha + \beta)},
\] 
the two parameter Mittag-Leffler function, see (\ref{MLfct}) for a precise definition of it and a survey of some of its properties. The following well known transforms 
\begin{equation}
\int_0^\infty \E^{-\theta x}x^{\alpha-1}\mathcal{E}_{\alpha,\alpha}(\lambda x^\alpha)\D x=\frac1{\theta^\alpha-\lambda},
\end{equation}
and
\begin{equation}
\int_0^\infty \E^{-\theta x}\lambda^{-1}x^{-\alpha-1}\mathcal{E}_{-\alpha,-\alpha}(\lambda^{-1} x^{-\alpha})\D x=
\frac\lambda{\lambda-\theta^\alpha}-1,
\end{equation}
valid for $\alpha>0,$ resp.\ $\alpha<0,$ together with 
the well-known rules for Laplace transforms concerning primitives and tilting allow us to quickly deduce the following expressions for the scale functions associated to the parent process with Laplace exponent given by (\ref{WHwithupkilling}) such that $\kappa\varphi =0$.

If $0<\alpha<1$ 
then
\begin{equation*}
W(x)=-\frac{\E^{\varphi x}}{c\Gamma(-\alpha)}\int_0^x\E^{-(\gamma+\varphi)y}y^{\alpha-1}
\mathcal{E}_{\alpha,\alpha}\left(\frac{\kappa+c\Gamma(-\alpha)\gamma^\alpha y^\alpha}{c\Gamma(-\alpha)}\right)\D y.
\end{equation*}
If $-1<\alpha<0$, then
\begin{eqnarray*}
\hspace{-1cm}W(x)&=&
\frac{\E^{\varphi x}}{\kappa+c\Gamma(-\alpha)\gamma^\alpha}\notag\\
&&+\frac{c\Gamma(-\alpha)\E^{\varphi x}}{(\kappa+c\Gamma(-\alpha)\gamma^\alpha)^2}\int_0^x\E^{-(\gamma+\varphi)y}y^{-\alpha-1}
\mathcal{E}_{-\alpha,-\alpha}\left(\frac{c\Gamma(-\alpha)y^{-\alpha}}{\kappa +c\Gamma(-\alpha)\gamma^\alpha}\right)\D y.
\end{eqnarray*}
\end{example}

\begin{example}
Let $c>0,$ $\nu\geq 0$ and $\theta\in(0,1)$ and $\phi$ be defined by
$$\phi(\lambda)=\frac{c\lambda\Gamma(\nu+\lambda)}{\Gamma(\nu+\lambda+\theta)},\qquad \lambda\geq 0.$$
 Elementary but  tedious calculations using the Beta integral allow to prove that $\phi$ is the Laplace exponent of some subordinator, $\widehat{H}.$ Its characteristics are $\kappa=0, \delta =0,$ $$\overline{\Upsilon}(x):=\Upsilon(x,\infty)=\frac{c}{\Gamma(\theta)}\E^{-x(\nu+\theta-1)}\left(\E^{x}-1\right)^{\theta-1},\qquad x>0.$$ Moreover, we have that $\overline{\Pi}_{\widehat{H}}$ is non-increasing and log-convex, so $\Pi_{\widehat{H}}$ has a non-increasing density. It follows from Theorem \ref{th:HK08} that there exists an oscillating spectrally negative L\'evy process, say $X,$ whose Laplace exponent is $\psi(\lambda)=\lambda\phi(\lambda),$ $\lambda\geq 0,$ with $\sigma=0,$ and L\'evy density given by $-\D^2\overline{\Upsilon}/\D x^2.$ Using again the Beta integral we can obtain the potential measure of $\widehat{H},$ and as a consequence the scale function associated to $X$, which turns out to be given by
$$W(x)=\frac{\Gamma(\nu+\theta)}{c\Gamma(\nu)}+\frac{\theta}{c\Gamma(1-\theta)}\int^x_{0}\left\{\int^\infty_{y}\frac{\E^{z(1-\nu)}}{(\E^{z}-1)^{1+\theta}}\D z\right\}\D y,\qquad x\geq 0.$$

The integral that defines this scale function can be calculated using the hypergeometric series. The particular case where $\nu=1,$ $c=\Gamma(1+\theta)$  appears in Chaumont, Kyprianou and Pardo \cite{CKP2007}

An interesting feature of this example is that it comes together with another example. Indeed, observe that the first derivative of $W$ is given by,
$$W^{\prime}(x)=\int^\infty_{x}\frac{\E^{z(1-\nu)}}{(\E^{z}-1)^{1+\theta}}\D z,\qquad x\geq 0,$$ which is non-increasing, convex and such that the  second derivative satisfies the integrability condition $\int^{\infty}_{0}(1\wedge x)|W^{\prime\prime}(x)|\D x<\infty.$ So, $|W^{\prime\prime}(x)|$ defines the L\'evy density of some subordinator. More precisely, the function defined by 
\begin{equation*}
\begin{split}
\phi^*(\lambda)&:=\frac{\lambda}{\phi(\lambda)}\\
&=\lambda\int^\infty_{0}\E^{-\lambda x}dW(x)\\
&=\frac{\Gamma(\nu+\theta)}{c\Gamma(\nu)}+\frac{\theta}{c\Gamma(1-\theta)}\int^\infty_{0}(1-\E^{-\lambda x})\frac{\E^{x(1-\nu)}}{(\E^{x}-1)^{1+\theta}}\D x,\ \lambda\geq 0,
\end{split}
\end{equation*}
 is the Laplace exponent of some subordinator $H^*,$ which in turn has a L\'evy measure with a non-increasing density. Hence $$\psi^*(\lambda)=\lambda\phi^*(\lambda)=\frac{\lambda^2}{\phi(\lambda)},\qquad \lambda\geq 0,$$ defines the Laplace exponent of a spectrally negative L\'evy process that drifts to $\infty.$ It can be easily verified by an integration by parts that the associated scale function is given by $$W^*(x)=\frac{c}{\Gamma(\theta)}\int^x_{0}\E^{-z(\nu+\theta-1)}(\E^{z}-1)^{\theta-1}\D z=\int^x_{0}\overline{\Pi}_{\widehat{H}}(z)\D z,\quad x\geq 0.$$
\end{example}

\bigskip 

The method described in the previous example for generating two examples of scale functions simultaneously can be formalized into a general theory that applies to a large family of subordinators, namely that of special subordinators.

\section{Special and conjugate scale functions}

In this section we introduce the notion of a special Bernstein functions and special subordinators and use the latter to justify the existence of pairs of so called  {\it conjugate scale functions} which have a particular analytical structure. We refer the reader to the lecture notes of Song and Vondra\v{c}ek \cite{SV2007}, the recent book by Schilling, Song and Vondra\v{c}ek \cite{SSV} and the books of  Berg and Gunnar \cite{BCFG1975}  and Jacobs \cite{J1} for a more complete account of the theory of Bernstein functions and their application in potential analysis.

Recall that the class of Bernstein functions coincides precisely with the class of Laplace exponents of possibly killed subordinators. That is to say, a general Bernstein function takes the form 
\begin{equation}
\phi(\theta) = \kappa + \delta\theta + \int_{(0,\infty)} (1 - \E^{-\theta x})\Upsilon (\D x), \text{ for }\theta\geq 0,
\label{Bernstein}
\end{equation}
where $\kappa\geq 0 $, $\delta\geq 0$ and $\Upsilon$ is a measure concentrated on $(0,\infty)$ such that $\int_{(0,\infty)}(1\wedge x)\Upsilon(\D x)<\infty$.

\begin{definition}
Suppose that $\phi(\theta)$ is a Bernstein function, then it is called a {\it special Bernstein function} if 
\begin{equation}
 \phi(\theta) = \frac{\theta}{\phi^*(\theta)},\qquad \theta\geq 0,
\label{SB}
\end{equation}
where $\phi^*(\theta)$ is another Bernstein function. In this case we will say that the Bernstein function $\phi^*$ is conjugate to $\phi$. Accordingly a possibly killed subordinator is called a special subordinator if its Laplace exponent is a special Bernstein function.
\end{definition}
It is apparent from its definition that $\phi^*$ is a special Bernstein function and $\phi$ is conjugate to $\phi^*$.   In \cite{Haw1975} and \cite{SV2007a} it is shown that a sufficient condition for $\phi$ to be a special subordinator is that $\Upsilon(x,\infty)$ is log-convex on $(0,\infty)$.

For conjugate pairs of special Bernstein functions $\phi$ and $\phi^*$ we shall write in addition to (\ref{SB})
\begin{equation}
 \phi^*(\theta) = \kappa^* + \delta^*\theta + \int_{(0,\infty)}(1 - \E^{-\theta x})\Upsilon^*(\D x),\qquad \theta\geq 0,
 \label{Bernstein*}
\end{equation}
where necessarily $\Upsilon^*$ is a measure concentrated on $(0,\infty)$ satisfying $\int_{(0,\infty)}(1\wedge x)\Upsilon^*(\D x)<\infty$. One may express the triple $(\kappa^*, \delta^*, \Upsilon^*)$ in terms of related quantities coming from the conjugate $\phi$. Indeed it is known that 
\[
 \kappa^* = \left\{
\begin{array}{ll} 
0, & \kappa>0, \\
\left(\delta + \int_{(0,\infty)} x\Upsilon(\D x)\right)^{-1}, & \kappa = 0;
\end{array}
\right.
\]
and
\begin{equation}
 \delta^* = \left\{\begin{array}{ll}
0, & \delta>0 \text{ or } \Upsilon(0,\infty)=\infty, \\
\left( \kappa + \Upsilon(0,\infty)\right)^{-1}, & \delta =0 \text{ and }\Upsilon(0,\infty)<\infty.
\end{array}\right.
\label{d*}
\end{equation}
 Which implies in particular that $\kappa^*\kappa=0=\delta^*\delta.$
In order to describe the measure $\Upsilon^*$ let us denote by 
 $W(\D x)$ 
the potential measure of $\phi$.  (This choice of notation will of course prove to be no coincidence). Then we have that $W$ necessarily satisfies
\[
 W(\D x) = \delta^*\delta_0(\D x) + \{\kappa^* + \Upsilon^*(x,\infty)\}\D x,\quad \text{ for }x\geq 0,
\]
where $\delta_0({\rm d}x)$ is the Dirac measure at zero.
Naturally, if $W^*$ is the potential measure of $\phi^*$ then we may equally describe it in terms of $(\kappa, \delta, \Upsilon)$.
In fact it can be easily shown that a necessary and sufficient condition for a Bernstein function to be special is that its potential measure has a density on $(0,\infty)$
which is non-increasing and integrable in the neighborhood of the origin.

\bigskip

We are interested in constructing a parent process whose descending ladder height process is a special subordinator. The following theorem and corollary are now evident given the discussion in the current and previous sections.

\begin{theorem}\label{specialscale}
 For conjugate special Bernstein functions $\phi$ and $\phi^*$ satisfying (\ref{Bernstein}) and (\ref{Bernstein*}) respectively, where  $\Upsilon$ is absolutely continuous with non-increasing density, there exists a spectrally negative L\'evy process that does not drift to $-\infty$, 
whose Laplace exponent is described by 
\begin{equation}
 \psi(\theta) = \frac{\theta^2}{\phi^*(\theta)} = \theta\phi(\theta),\text{ for }\theta\geq 0,
 \label{parent}
\end{equation}
and
whose scale function is a concave function and is given by 
\begin{equation}
 W(x) = \delta^* + \kappa^*x + \int_0^x \Upsilon^*(y,\infty)\D y.
 \label{specialW}
\end{equation}
\end{theorem}

The assumptions of the previous theorem require only that  the L\'evy and potential measures associated to $\phi$ have a non-increasing density in $(0,\infty)$, respectively; this condition on the potential measure is equivalent to the existence of $\phi^*.$ If in addition it is  assumed that the potential density be a convex function, in light of the representation (\ref{specialW}), we can interchange the roles of $\phi$ and $\phi^*,$ respectively, in the previous theorem. The key issue to this additional assumption is that it ensures the absolute continuity of $\Upsilon^*$ with a non-increasing  density and hence that we can apply the Theorem \ref{th:HK08}. We thus  have the following Corollary.    

\begin{corollary}\label{corrspecialscale}
If conjugate special Bernstein functions $\phi$ and $\phi^*$ exist satisfying (\ref{Bernstein}) and (\ref{Bernstein*}) such that both $\Upsilon$ and $\Upsilon^*$ are absolutely continuous with non-increasing densities, then there exist a pair of scale functions $W$ and $W^*$, such that $W$ is concave, its first derivative is a convex function,  (\ref{specialW}) is satisfied, and 
\begin{equation}
 W^*(x) = \delta + \kappa x + \int_0^x \Upsilon(y,\infty)\D y.
 \label{specialW*}
\end{equation}
Moreover, the respective parent processes are given by (\ref{parent}) and
\begin{equation}\label{parent*}
\psi^*(\theta) = \frac{\theta^2}{\phi(\theta)} = \theta\phi^*(\theta).
\end{equation}
\end{corollary}

It is important to mention that the converse of the latter Theorem and Corollary hold true but we omit a statement and proof for sake of brevity and refer the reader to the article \cite{KR2008} for further details.


For obvious reasons we shall henceforth refer to the scale functions identified in (\ref{Bernstein}) and (\ref{Bernstein*}) as {\it special scale functions}.
Similarly, when $W$ and $W^*$ exist then we refer to them as {\it conjugate (special) scale functions} and their respective parent processes are called {\it conjugate parent processes}. This conjugation can be seen by noting that thanks to (\ref{SB}).
\[
W*W^*(\D x) = \D x .
\]

\section{Tilting and parent processes drifting to $-\infty$}\label{genconst}
In this section we present two methods for which, given a scale function and associated parent process, it is possible to construct further examples of scale functions by appealing to two procedures. 

The first method relies on the following known facts concerning translating the argument of a given Bernstein function. 
Let $\phi$ be a special Bernstein function with representation given by (\ref{Bernstein}). Then for any $\beta\geq 0$ the function $\phi_{\beta}(\theta)=\phi(\theta+\beta),$ $\theta\geq 0,$ is also a special Bernstein function with killing term $\kappa_{\beta}=\phi(\beta),$ drift term ${\rm d}_\beta = {\rm d}$ and L\'evy measure $\Upsilon_\beta(\D x)=\E^{-\beta x}\Upsilon(\D x),$ $x>0,$ see e.g. \cite{kyprianou-palmowski:9} Section 33. Its associated potential measure, $W_\beta$, has a decreasing density  in $(0,\infty)$ such that  $W_{\beta}(\D x)=\E^{-\beta x}W'(x)\D x,$ $x>0,$ where $W'$ denotes the density of the potential  measure associated to $\phi.$ Moreover, let $\phi^*$ and $\phi^{*}_{\beta},$ denote the conjugate Bernstein functions of $\phi$ and $\phi_{\beta},$ respectively. Then the following identity
\begin{equation}\label{bernsteintilting}
\phi^*_{\beta}(\theta)=\phi^*(\theta+\beta)-\phi^*(\beta)+\beta\int^\infty_{0}\left(1-\E^{-\theta x}\right)\E^{-\beta x}W'(x)\D x,\qquad \theta\geq 0,
\end{equation}
holds.
Note in particular that if $\Upsilon$ has a non-increasing density then so does $\Upsilon_\beta$. Moreover, if $W'$  is convex (equivalently $\Upsilon^*$ has a non-increasing density) then $W'_\beta$ is  convex (equivalently $\Upsilon^*_\beta$ has a non-increasing density). These facts lead us to the following Lemma.
\begin{lemma}\label{exptilting}
If conjugate special Bernstein functions $\phi$ and $\phi^*$ exist satisfying (\ref{Bernstein}) and (\ref{Bernstein*}) such that both $\Upsilon$ and $\Upsilon^*$ are absolutely continuous with non-increasing densities, then there exist conjugate parent processes with Laplace exponents 
\begin{equation*}
 \psi_{\beta}(\theta)=\theta\phi_{\beta}(\theta)\text{ and }  \psi_{\beta}^*(\theta)=\theta\phi^*_{\beta}(\theta),\qquad \theta\geq 0,
\end{equation*}
whose respective scale functions are given by 
\begin{equation}\label{eq:trans}
\begin{split}
W_{\beta}(x)&=\delta^*+\int^x_{0}\E^{-\beta y}\Upsilon^*(y,\infty){\rm d}y\\
&=\E^{-\beta x}W(x)+\beta\int^x_{0}\E^{-\beta z}W(z)\D z,\qquad x\geq 0,
\end{split}
\end{equation}
and
\begin{equation}\label{conjscaletilted}
W^*_{\beta}(x)=\delta+\phi(\beta)x+\int^x_{0}\left(\int^\infty_{y}\E^{-\beta z}\Upsilon({\rm d}z)\right){\rm d}y,
\end{equation} 
using obvious notation. 
\end{lemma}
All the statements in this Lemma, but the equality (\ref{eq:trans}), follow from the previous discussion. The equality (\ref{eq:trans}) is a simple consequence of the expression of $W$ in (\ref{specialW}) and an integration by parts. 

\bigskip

The second method builds on the first to construct examples of scale functions whose parent processes may be seen as an auxiliary parent process conditioned to drift to $-\infty$.

Suppose that  $\phi$ is a Bernstein function such that $\phi(0)=0,$ its associated L\'evy measure has a decreasing density and let $\beta>0$. Theorem \ref{th:HK08}, as stated in its more general form in \cite{HK2007}, says that there exists a parent process, say $X,$ that drifts to $-\infty$ such that its Laplace exponent $\psi$ can be factorized as $$\psi(\theta)=(\theta-\beta)\phi(\theta),\qquad \theta\geq 0.$$ It follows that $\psi$ is a convex function and $\psi(0)=0=\psi(\beta),$ so that $\beta$ is the largest positive solution to the equation $\psi(\theta)=0.$ Now, let $W_{\beta}$ be the $0$-scale function of the spectrally negative L\'evy process, say  $X_{\beta},$ with Laplace exponent 
$\psi_{\beta}(\theta):=\psi(\theta+\beta),$ for $\theta\geq 0.$ It is known that the L\'evy process $X_{\beta}$ is obtained by an exponential change of measure and can be seen as the L\'evy process $X$ conditioned to drift to $\infty,$ see chapter VII in \cite{bert96}. Thus the Laplace exponent $\psi_{\beta}$ can be factorized as $\psi_{\beta}(\theta)=\theta\phi_{\beta}(\theta),$ for $\theta\geq 0,$ where, as before, $\phi_{\beta}(\cdot):=\phi(\beta+\cdot).$ 
It follows from Lemma 8.4 in \cite{kyp06}, that the $0$-scale function of the process with Laplace exponent $\psi$ is related to $W_{\beta}$ by $$W(x)=\E^{\beta x}W_{\beta}(x),\qquad x\geq 0.$$ 
The above considerations thus lead to the following result which allows for the construction of a second parent process and associated scale function over and above the pair described in Theorem \ref{specialscale}.

\begin{lemma}\label{beta-shift} Suppose that $\phi$ is a special Bernstein function satisfying (\ref{Bernstein}) such that $\Upsilon$ is absolutely continuous with non-increasing density and $\kappa=0$. Fix $\beta>0$. Then there exists a parent process with Laplace exponent 
\[
\psi(\theta) = (\theta-\beta)\phi(\theta), \quad \theta\geq 0,
\]
whose associated scale function is given by 
$$W(x)=\delta^*\E^{\beta x}+\E^{\beta x}\int^x_{0}\E^{-\beta y}\Upsilon^*(y,\infty)\D y,\qquad x\geq 0,$$ where we have used our usual notation. 
\end{lemma}

In \cite{KR2008} the interested reader may find a discussion about the conjugated pairs arising in this construction.  Note that the conclusion of this Lemma can also be derived from (\ref{b}) and (\ref{specialW}).

\section{Complete scale functions}\label{completeclass}
We have seen several methods that allow us to construct scale functions and pairs of conjugate scale functions which in principle generate large families of scale functions. In particular the method of constructing pairs of conjugate scale functions and its tilted versions needs the hypothesis of decreasing densities for the L\'evy measures of the underlying conjugate subordinators. This may be a serious issue because in order to verify that hypothesis one needs to determine explicitly both densities, which can be a very hard and technical task. Luckily, there is a large class of Bernstein functions, the so-called {\it complete Bernstein functions}, for which this condition is satisfied automatically. Hence, our purpose in this section is to recall some of the keys facts related to this class and its consequences.

We begin by introducing the notion of a complete Bernstein function with a view to constructing scale functions whose parent processes are derived from descending ladder height processes with Laplace exponents which belong to the class of complete Bernstein functions.

\begin{definition}\rm
A function $\phi$ is called {\it complete Bernstein function} if there exists an auxiliary Bernstein function $\eta$ such that 
\begin{equation}
\phi(\theta) = \theta^2 \int_{(0,\infty)} \E^{-\theta x}\eta(x)\D x.
\label{cB}
\end{equation}
\end{definition}
It is well known that a complete Bernstein function is necessarily a special Bernstein function (cf. \cite{J1}) and in addition, its conjugate is also a complete Bernstein function. Moreover, from the same reference one finds that a necessary and sufficient condition for $\phi$ to be complete Bernstein is that $\Upsilon$ satisfies for $x>0$
\[
\Upsilon(\D x) = \left\{\int_{(0,\infty)} \E^{- xy}\gamma(\D y)\right\} \D x,
\]
where 
$
\int_{(0,1)}\frac{1}{y}\gamma(\D y) + \int_{(1,\infty)}\frac{1}{y^2}\gamma(\D y) <\infty.
$
Equivalently $\Upsilon$ has a completely monotone density.
Another necessary and sufficient condition is that the potential measure associated to $\phi$ has a density on $(0,\infty)$ which is completely monotone, this is a result due to Kingman \cite{Kin1967} and Hawkes \cite{Haw1976}. The class of infinitely divisible laws and subordinators related to this type of Bernstein functions has been extensively studied by several authors, see e.g. \cite{bondesson}, \cite{thorin}, \cite{rosinski}, \cite{donatiyor}, \cite{lancelotroynetteyor} and the references therein. 

Since necessarily $\Upsilon$ is absolutely continuous with  a completely monotone density, it follows that any subordinator whose Laplace exponent is a complete Bernstein function may be used in conjunction with Corollary \ref{corrspecialscale}. The following result is now a straightforward application of the latter and the fact that from  (\ref{cB}), any Bernstein function $\eta$ has a Laplace transform $\phi(\theta)/\theta^2$ where $\phi$ is complete Bernstein.

\begin{corollary}\label{completescale} Let $\eta$ be any Bernstein function and suppose that $\phi$ is the complete Bernstein function associated with the latter via the relation (\ref{cB}). Write $\phi^*$ for the conjugate of $\phi$ and  $\eta^*$ for the Bernstein function associated with $\phi^*$ via the related (\ref{cB}). Then 
\[
W(x) = \eta^*(x) \text{ and } W^*(x) = \eta(x),\qquad x\geq 0,
\]
are conjugate  scale functions with conjugate parent processes whose Laplace exponents are given by 
\[
\psi(\theta) = \frac{\theta^2}{\phi^*(\theta)}  = \theta\phi(\theta) \text{ and }\psi^*(\theta) = \frac{\theta^2}{\phi(\theta)} = \theta\phi^*(\theta),\qquad \theta\geq 0.
\]
\end{corollary}
An interesting and useful consequence of this results is that any given Bernstein function $\eta$ is a scale function whose parent process is the spectrally negative L\'evy process whose Laplace exponent is given by $\psi^*(\theta) = \theta^2/\phi(\theta)$ where $\phi$ is given by (\ref{cB}).

\begin{example}\label{sumstable}
Let $0<\alpha<\beta\leq1$, $ a ,b>0$ and $\phi$ be the Bernstein function defined by 
$$
\phi(\theta)= a \theta^{\beta-\alpha}+b\theta^{\beta},\qquad \theta\geq 0.
$$ 
That is, in the case where $\alpha<\beta<1,$ $\phi$ is the Laplace exponent of a subordinator which is obtained as the sum of two independent stable subordinators one of parameter $\beta-\alpha$ and the other of parameter $\beta,$ respectively, so that the killing and drift term of $\phi$ are both equal to $0,$ and its L\'evy measure is given by $$\Upsilon(\D x)=\left(\frac{ a (\beta-\alpha)}{\Gamma(1-\beta+\alpha)}x^{-(1+\beta-\alpha)}+\frac{b\beta}{\Gamma(1-\beta)}x^{-(1+\beta)}\right)\D x,\quad x>0.$$ 
In the case that $\beta =1$, $\phi$ is the Laplace exponent of a stable subordinator with parameter $1-\alpha$ and a linear drift.
In all cases, the underlying L\'evy measure has a density which is completely monotone, and thus its potential density, or equivalently the density of the associated scale function $W$, is completely monotone. 

We recall that the two parameter Mittag-Leffler function is defined by 
\begin{equation}\label{MLfct}
\mathcal{E}_{\alpha,\beta}(x) = \sum_{n\geq0}\frac{x^n}{\Gamma(n\alpha+\beta)},\qquad x\in\mathbb{R},
\end{equation}
where $\alpha,\beta>0$. The latter function can be identified via a pseudo-Laplace transform. Namely, for $\lambda\in\mathbb{R}$ and $\Re(\theta)>\lambda^{1/\alpha}-\gamma$,
\[
\int_0^\infty \E^{-\theta x}\E^{-\gamma x}x^{\beta-1}\mathcal{E}_{\alpha,\beta}(\lambda x^\alpha)\D x=
\frac{(\theta+\gamma)^{\alpha-\beta}}{(\theta+\gamma)^\alpha-\lambda}.
\] 

The  associated scale function to $\phi$ can now be identified via 
\begin{equation} 
W'(x) = \frac{1}{b}x^{\beta-1}\mathcal{E}_{\alpha,\beta}\left(- a  x^{\alpha}/b\right),  \qquad x> 0,\label{MLreq}\end{equation} which is a completely monotone function. So, the function $$\psi(\theta)=\theta\phi(\theta)=a \theta^{\beta-\alpha+1} + b\theta^{\beta+1},  \qquad\theta\geq 0,$$ is the Laplace exponent of a spectrally negative L\'evy process, the parent process. It oscillates and is obtained by adding two independent spectrally negative stable processes with stability index $\beta+1$ and $1+\beta-\alpha,$ respectively. The scale function associated to it is given by 
$$
W(x)=\frac{1}{b}\int^x_{0}t^{\beta-1}\mathcal{E}_{\alpha,\beta}(- a  t^{\alpha}/b)d t,\qquad x\geq 0.
$$ 
The associated conjugates are given by $$\phi^*(\theta)=\frac{\theta}{a\theta^{\beta-\alpha}+b\theta^{\beta}},\quad \psi^*(\theta)=\frac{\theta^2}{ a \theta^{\beta-\alpha}+b\theta^{\beta}},\qquad \theta\geq 0,$$ and 
\begin{equation}
W^*(x)=\frac{ a }{\Gamma(2-\beta+\alpha)}x^{1-\beta+\alpha}+\frac{b}{\Gamma(2-\beta)}x^{1-\beta},\qquad x\geq 0.
\label{producedanotherway}
\end{equation}
The subordinator with Laplace exponent $\phi^*$ has zero killing and drift terms and its L\'evy measure is obtained by taking the derivative of the expression in (\ref{MLreq}). By Theorem \ref{th:HK08} the spectrally negative L\'evy process with Laplace exponent $\psi^*,$ oscillates, has unbounded variation, has zero  Gaussian terms, and its L\'evy measure is obtained by derivating twice the expression in (\ref{MLreq}).

One may mention  here that by letting  $ a\downarrow 0 $   the Continuity Theorem for  Laplace transforms  tells us  that for the case  $\phi(\theta)=b\theta^\beta,$ the associated $\psi$ is the Laplace exponent of a spectrally negative stable process with stability parameter $1+\beta,$ and its scale function is given by $$W(x)=\frac{1}{b\Gamma(1+\beta)}x^{\beta},\qquad x\geq 0.$$ The associated conjugates are given by $$\phi^*(\theta)=b^{-1}\theta^{1-\beta},\qquad \psi^*(\theta)=b^{-1}\theta^{2-\beta},\qquad \theta\geq 0,$$ and $$W^*(x)=\frac{b}{\Gamma(2-\beta)} x^{1-\beta},\quad x\geq 0.$$ So that $\phi^*$ (respectively $\psi^*$) corresponds to a stable subordinator of parameter $1-\beta,$ zero killing and drift terms (respectively, to a oscillating spectrally negative stable L\'evy process with stability index $2-\beta$), and so its L\'evy measure is given by   
\[
\Pi^*(-\infty, -x) = \frac{\beta(1-\beta)}{b\Gamma(1+\beta)}x^{\beta-2},\qquad x\geq 0.
\]

To complete this example, observe that the change of measure introduced in Lemma \ref{exptilting} allows us to deal with the Bernstein function $$\phi(\theta)=k(\theta+m)^{\beta-\alpha}+b(\theta+m)^\beta,\quad \theta\geq 0,$$ where $m\geq 0$ is a fixed parameter. In this case we get that there exists a spectrally negative L\'evy process whose Laplace exponent is given by $$\psi(\theta)=k\theta(\theta+m)^{\beta-\alpha}+b\theta(\theta+m)^\beta,\quad \theta\geq 0,$$ and its associated scale function is given by $$W(x)=\frac{1}{b}\int^x_{0}\E^{-mt}t^{\beta-1}\mathcal{E}_{\alpha,\beta}(- a  t^{\alpha}/b)d t,\qquad x\geq 0.$$ The respective conjugates can be obtained explicitly but we omit the details given that the expressions found are too involved. 

We can now use the construction in Subsection \ref{genconst}. For, $m,a,b>0,$ $0<\alpha< \beta< 1,$  there exists a parent process drifting to $-\infty$ and with Laplace exponent 
$$\psi(\theta)=(\theta-m)\left(a\theta^{\beta-\alpha}+b\theta^\beta\right),\qquad \theta\geq 0.$$ It follows from the previous calculations that the scale function associated to the parent process with Laplace exponent $\psi$ is given by 
$$W(x)=\frac{\E^{mx}}{b}\int^x_{0}\E^{-mt}t^{\beta-1}\mathcal{E}_{\alpha,\beta}(- a  t^{\alpha}/b)d t,\qquad x\geq 0.$$

\end{example}

\section{Generating  scale functions via an analytical transformation}


In Chazal et al. \cite{KP} it was shown that whenever $\psi$ is the Laplace exponent of a spectrally negative L\'evy process then so is 
\begin{eqnarray*}
\mathcal{T}_{\delta,\beta}\psi^{(q)}(u) &:=& \frac{u+\beta-\delta}{u+\beta}
\psi^{(q)}(u+\beta)-\frac{\beta-\delta}{\beta}\psi^{(q)}(\beta), \quad u\geq -\beta, \,\,u\geq 0,
\end{eqnarray*}
where $\psi^{(q)}(u) = \psi(u)-q$, $\delta, \beta\geq 0$   and  $\psi^{(q)\prime}(0+)=q=0$ if $\beta = 0$. Note that 
  $\psi^{(q)}$ is the Laplace exponent of  a spectrally negative L\'evy process possibly killed at an independent and exponentially distributed random time with rate $q$. In the usual way, when $q=0$ we understand there to be no killing.
Moreover, the characteristics of the L\'evy process with Laplace exponent $\mathcal{T}_{\delta,\beta}\psi$ are also described in the aforementioned paper through the following result.

\begin{proposition} \label{prop:mapping} 
Fix $\delta, \beta\geq 0$   with the additional constraint that $\psi^{(q)\prime}(0+)=q=0$ if $\beta = 0$. 
If $\psi^{(q)}$ has Gaussian coefficient $\sigma$ and jump measure $\Pi$ then $\mathcal{T}_{\delta,\beta}\psi^{(q)}$ also has Gaussian coefficient $\sigma$ and its L\'evy  measure is given by 
\[
\E^{\beta x}\Pi({\rm d}x)+\delta \E^{\beta x}\overline{\Pi}(x){\rm d}x + \delta\frac{\kappa}{\beta}\E^{\beta x}{\rm d}x\,\, \text{ on }(-\infty,0),
\]
where $\overline{\Pi }(x) = \Pi(-\infty, -x)$.
\end{proposition}

\noindent In particular we note that $\mathcal{T}_{\delta,\beta}\psi^{(q)}$ is the Laplace exponent of a spectrally negative L\'evy process without killing, i.e. $\mathcal{T}_{\delta,\beta}\psi^{(q)}(0) =0$.

It turns out that this transformation can be used in a very straightforward way in combination with the definition of $q$-scale functions to generate new examples. Note for example that for a given $\psi$ and $q\geq 0$ we have 
\[
 \int_0^\infty \E^{-\theta x} W^{(q)}(x){\rm d}x = \frac{1}{\psi^{(q)}(\theta)},\text{ for }\theta >\Phi(q),
\]
and hence it is natural to use this transformation to help find the Laplace inverse (if it exists ) of 
\[
 \frac{1}{\mathcal{T}_{\delta,\beta}\psi^{(q)}(\theta)},
\]
for $\beta$ sufficiently large to give an expression for $W_{\mathcal{T}_{\beta, \delta}\psi^{(q)}}$, the 0-scale function associated with the spectrally negative L\'evy process whose exponent is $\mathcal{T}_{\beta, \delta}\psi^{(q)}$. The following result, taken from \cite{KP} does exactly this.
Note that the first conclusion in the theorem gives a similar result to the conclusion of Lemma \ref{exptilting} without the need for the descending ladder height to be special.

\begin{theorem}\label{scalefunctions}
Let $x,\beta\geq0$ such that $\psi'(0+)=q=0$ if $\beta = 0$. Then,
\begin{eqnarray}
W_{\mathcal{T}_{\beta,\beta}\psi^{(q)}}(x)&=& \E^{-\beta x}W^{(q)}(x)
+\beta \int_0^{x}\E^{-\beta y}W^{(q)}(y)\D y.
\label{scale1}
\end{eqnarray}
Moreover, if $\psi'(0+)\leq0$, then  for any $x,\delta, q \geq0$ we have
\begin{eqnarray*}
W_{\mathcal{T}_{\delta,\Phi(0)}\psi^{(q)}}(x)&=&\E^{-\Phi(0)
x}\left(W^{(q)}(x) +\delta \E^{\delta
x}\int_0^{x}\E^{-\delta
y}W^{(q)}(y)\D y\right).
\end{eqnarray*}
\end{theorem}

\begin{proof}
The first assertion is
proved by observing that
\begin{eqnarray*}
 \int_0^{\infty}\E^{-\theta x}W_{\mathcal{T}_{\beta, \beta}\psi^{(q)}}(x)\D x &=& \frac{\theta+\beta}{\theta\psi(\theta+\beta)} \nonumber \\
&=& \frac{1}{\psi(\theta+\beta)} +\frac{\beta}{\theta\psi(\theta+\beta) },
\label{invertthis}
\end{eqnarray*}
which agrees with the Laplace transform of the right hand side of (\ref{scale1}) for which an integration by parts is necessary. As scale functions are right continuous, the result follows by the uniqueness of Laplace transforms.

For the second claim,  first note that $\mathcal{T}_{\delta, \theta}\psi^{(q)} = (\theta+\Phi(0)-\delta)\psi(\theta+\Phi(0))/(\theta+\Phi(0))$. A straightforward calculation shows that 
 for all $\theta+\delta>\Phi(0)$, we have
\begin{eqnarray*} \label{eq:sc}
 \int_0^{\infty}\E^{-\theta x}\E^{(\Phi(0)  -\delta)x}W_{\mathcal{T}_{\delta,\Phi(0)}\psi^{(q)}}(x)\D x &=& \frac{\theta+\delta}{\theta\psi(\theta+\delta)}.
 \end{eqnarray*}
The result now follows from the first part of the theorem.
\hfill$\square$\end{proof}

Below we give an example of how this theory can be easily applied to generate new scale functions from those of (tempered) scale functions.

\begin{example}
Let 
\[
 \psi^{(q)}_{c}(\theta)=(\theta+c)^{\alpha} - c^\alpha- q\text{ for }\theta\geq -c,
\]
 where $1<\alpha<2$ and $q,c\geq 0$.  This is the Laplace exponent of an unbounded variation tempered stable spectrally negative L\'evy process $\xi$ killed at an independent and exponentially distributed time with rate $q$. In the case that $c=0$, the underlying L\'evy process is a
spectrally negative $\alpha$-stable L\'evy process. In that case it is known that
\[
\int_0^\infty \E^{-\theta x} x^{\alpha-1}\mathcal{E}_{\alpha,\alpha}(q x^{\alpha})\D x = \frac{1}{\theta^\alpha-q},
\]
and hence the scale function is given by
\begin{eqnarray*}
W_{\psi^{(q)}_{0}}(x)&=& x^{\alpha-1}\mathcal{E}_{\alpha,\alpha}(qx^{\alpha}),
\end{eqnarray*}
for $x\geq 0$.
(Note in particular that when $q=0$ the expression for the scale function simplifies to $\Gamma(\alpha)^{-1} x^{\alpha-1}$).
Since
\[
\int_0^\infty \E^{- \theta x} \E^{-cx}W_{\psi^{(q+ c^\alpha)}_{0}}(x) \D x = \frac{1}{(\theta+c)^{\alpha} - c^\alpha- q},
\]
it follows that
\[
W_{\psi^{(q)}_{c}}(x)= \E^{-cx}W_{\psi^{(q+ c^\alpha)}_0}(x) = \E^{-cx}x^{\alpha-1}\mathcal{E}_{\alpha,\alpha}((q+ c^\alpha)x^{\alpha}).
\]
Appealing to the first part of Theorem \ref{scalefunctions} we now know that for $\beta \geq 0,$
\[
 W_{\mathcal{T}_{\beta}\psi^{(q)}_{c}}(x) =
\E^{-(\beta +c) x }x^{\alpha-1}\mathcal{E}_{\alpha,\alpha}((q+ c^\alpha)x^{\alpha})
+ \beta \int_0^x \E^{-(\beta +c) y }y^{\alpha-1}\mathcal{E}_{\alpha,\alpha}((q+ c^\alpha)y^{\alpha})\D y.
\]

Note that $\psi^{(q)\prime}_{c}(0+) = \alpha c^{\alpha-1}$ which is zero if and only if $c=0$. We may use the second and third part of Theorem \ref{scalefunctions} in this case. Hence, for any $\delta>0$, the scale function of the spectrally negative L\'evy process with Laplace exponent
$\mathcal{T}_{\delta,0}\psi^{(0)}_0$ is
\begin{eqnarray*}
W_{\mathcal{T}_{\delta,0}\psi^{(0)}_0}(x)&=& \frac{1}{\Gamma(\alpha-1)} \E^{\delta
x}\int_0^x\E^{-\delta y}y^{\alpha-2} \D y\\
&=&\frac{\delta^{\alpha-1}}{\Gamma(\alpha-1)} \E^{\delta
x}\Gamma(\alpha-1,\delta x),
\end{eqnarray*}
where we have used the recurrence relation for the Gamma function and $\Gamma(a,b)$ stands for the incomplete Gamma function of parameters $a,b>0$. Moreover, we have, for any $\beta >0$,
\begin{eqnarray*}
W_{\mathcal{T}^{\beta}_{\delta,0}\psi^{(0)}_0}(x)&=& \frac{1}{\Gamma(\alpha-1)} \left( \frac{\beta^{\alpha} }{\beta-\delta}\Gamma(\alpha-1,\beta x)- \E^{(\beta-\delta)
x}\frac{\delta^{\alpha} }{\beta-\delta}\Gamma(\alpha-1,\delta x)\right).
\end{eqnarray*}
Finally, the scale function of the spectrally negative L\'evy process with Laplace exponent
$\mathcal{T}^{\beta}_{\delta,0}\psi^{(q)}_{0}$ is given by
\begin{eqnarray*}
W_{\mathcal{T}^{\beta}_{\delta,0}\psi^{(q)}_{0}}(x)&=& \frac{\beta}{\beta-\delta}(x/\beta)^{\alpha-1}\mathcal{E}_{\alpha,\alpha-1}\left(x;\frac{q}{\beta}\right) - \frac{\delta}{\beta-\delta}\E^{-(\beta-\delta)
x}(x/\delta)^{\alpha-1}\mathcal{E}_{\alpha,\alpha-1}\left(x;\frac{q}{\delta}\right),
\end{eqnarray*}
where we have used the notation
\begin{eqnarray*}
\mathcal{E}_{\alpha,\beta}\left(x;q\right)=\sum_{n=0}^{\infty}\frac{\Gamma(x;\alpha n+\beta) q^n}{\Gamma(\alpha n+\beta)}.
\end{eqnarray*}
\end{example}

\section{Shifted scale functions}

In this final section, we make some remarks concerning how, in certain circumstances, simply adding unity to an existing scale function creates a new scale function of a completely unrelated spectrally negative L\'evy process.

\begin{theorem}
 Suppose that $W$ is a complete scale function belonging to a spectrally negative L\'evy process of unbounded variation, and whose Laplace exponent $\psi$ satisfies $\psi'(0+)=0$. Then  $1+W(x)$ is a scale function of a spectrally negative L\'evy process, say $X^{\rm s}$, whose Laplace exponent is given by
\[
 \psi^{\rm s}(\theta):=\frac{\theta\psi(\theta)}{\theta + \psi(\theta)} \text{ for }\theta \geq 0.
\]
Moreover, it is necessarily the case that $X^{s}$ has paths of bounded variation and its jump measure, $\Pi^s$, satisfies
\[
 \Pi^{\rm s}(-\infty, -x) = \widetilde{W}'(x)
\]
for $x>0$ where $\widetilde{W}$ is the scale function associated with the spectrally negative L\'evy process whose Laplace exponent is given by $\psi(\theta) + \theta$.

\end{theorem}
\begin{proof}
The assumptions in the theorem allow  us to write $\psi(\theta) = \theta \phi(\theta)$ where $\phi$ is a complete Bernstein function satisfying $\phi(0)=0$.

Let $\widetilde{W}$ be the scale function of the spectrally negative L\'evy process with Laplace exponent $\theta+\psi(\theta).$ The key to the proof is to consider a compound Poisson subordinator, say $S=\{S_t: t\geq 0\}$ with unit arrival rate and jump distribution whose jump measure is given by $\widetilde{W}({\rm d}x)$. Note that the latter is a probability distribution on account of the fact that for $x\geq 0$
\[
 \widetilde{W}[0,\infty)=\lim_{\beta\downarrow 0}\int_{[0,\infty)}\E^{-\beta x}\widetilde{W}({\rm d}x) = 
\lim_{\beta\downarrow 0} \frac{1}{\phi(\beta)+1} = 1.
\]
Suppose that $V({\rm d}x)$ is the renewal measure associated with this compound Poisson subordinator. In particular, a straightforward computation shows that
\begin{eqnarray*}
 V({\rm d}x) &=&\int_0^\infty {\rm d}t \cdot P(S_t \in {\rm d}x) \\
&=&\int_0^\infty {\rm d}t \cdot \sum_{n=0}^\infty \E^{-t} \frac{t^n}{n!}\widetilde{W}^{*n}({\rm d}x)\\
&=& \sum_{n=0}^\infty \widetilde{W}^{*n}({\rm d}x),
\end{eqnarray*}
where we understand $\widetilde{W}^{*0}({\rm d}x) = \delta_0({\rm d}x)$.

Taking Laplace transforms we find, for $\beta>0,$ that 
\[
 \int_{[0,\infty )} \E^{-\beta x}V({\rm d}x) =  \sum_{n=0}^\infty \left(\frac{1}{\phi(\beta)+1}\right)^n = 1+\frac{1}{\phi(\beta)}.
\]
However, since 
\[
\int_{[0,\infty )} \E^{-\beta x}\{\delta_0({\rm d}x)+W({\rm d}x) \} = 1+\frac{1}{\phi(\beta)},
\]
we deduce that 
\[
 V(x) = 1+ W(x).
\]

Our objective is now to show that $V$ is the scale function of the process $X^{\rm s}$. To this end, note that the compound Poisson subordinator behind the measure $V$ has Laplace exponent 
\[
 \int_0^\infty (1-\E^{-\theta x})\widetilde{W}({\rm d}x) = 1 - \frac{1}{1+\phi(\theta)} = \frac{\phi(\theta)}{1+\phi(\theta)}.
\]
Since the jump measure  $\widetilde{W}$ is also the renewal measure associated with the Bernstein function $1+\phi$, which is a complete Bernstein function (because $\phi$ is), it follows that $\widetilde{W}'$ is non-increasing and we may apply Theorem~\ref{th:HK08}
 to conclude that there exists a spectrally negative L\'evy process whose descending ladder height is the aforementioned compound Poisson subordinator. Moreover, this spectrally negative L\'evy process, $X^{\rm s}$, has Laplace exponent 
\[
 \psi^{\rm s}(\theta) = \theta \frac{\phi(\theta)}{1+\phi(\theta)} = \frac{\theta\psi(\theta)}{\theta+\psi(\theta)} ,
\]
and its scale function is given by $V = 1+W$.

Note that the process $X^{\rm s}$ necessarily has paths of bounded variation as its scale function has a discontinuity at the origin.
\hfill\hfill$\square$\end{proof}

\begin{example}
Consider the case of a spectrally negative stable process with index $\alpha\in(1,2)$. Its Laplace exponent is given by $\psi(\theta) = \theta^\alpha$ and its scale function is given by $W(x) =  x^{\alpha-1}/\Gamma(\alpha)$. The above theorem predicts that $1+ x^{\alpha-1}/\Gamma(\alpha)$ is a scale function of a process whose Laplace exponent is given by 
\[
 \psi^{\rm s} (\theta) = \frac{\theta^{1+\alpha}}{\theta + \theta^\alpha} = \frac{\theta^2}{\theta^{2-\alpha} + \theta}.
\]
Moreover, since it is known from Furrer \cite{Fur1998} that the scale function associated with spectrally negative L\'evy process whose Laplace exponent is given $\theta + \theta^\alpha$ is given by 
\[
 \widetilde{W}(x) = 1 - \mathcal{E}_{\alpha-1, 1}(- x^{\alpha-1}),
\]
this gives us the associated jump measure of $X^{\rm s}$.
Note that this example can also be recovered from (\ref{producedanotherway}) with an appropriate choice of constants.
\end{example}

\chapter{Numerical analysis of scale functions}

\section{Introduction}\label{sec_introduction}

The methods presented in the previous chapters for producing closed form expressions for scale functions are generous in the number of examples they provide, but also have their limitations.
This is not surprising, since the scale function is defined via its Laplace transform, 
and in most cases it is not possible to find an explicit expression for the inverse of a Laplace transform.
Our main objective in this chapter is to present several numerical methods which allows one  to compute  scale functions for a general spectrally negative  L\'evy process. 
As we will see, these computations can be done quite easily and efficiently, however there are a few tricks that one should be aware of. 
Our second goal is to discuss two very special families of L\'evy processes, i.e. processes with jumps of rational transform and Meromorphic processes, 
for which the scale function can be computed essentially in closed form.

The problem of numerical evaluation of the scale function and other related quantities has received some attention in the literature.
In particular, Rogers \cite{Rogers} computes the distribution of the first passage time for spectrally negative L\'evy processes by inverting the two-dimensional
Laplace transform. The main tool is the discretization of the Bromwich integral and the application of Euler summation in order to improve convergence. 
We will describe the one dimensional version of this method  in Section \ref{sec_Euler}.
 Surya \cite{Surya} presents an algoritm for evaluating the scale function using exponential dampening followed by Laplace inversion; the latter performed in a similar way as Rogers \cite{Rogers}.  In a recent paper Veillette  and Taqqu \cite{Veilette_Taqqu} compute the distribution of the first passage time for subordinators using two techniques. These are the discretization of the Bromwich integral and Post-Widder formula coupled with Richardson extrapolation. 
Albrecher, Florin and Kortschak \cite{Albrecher} develop algorithms to compute ruin probabilities for completely monotone claim distributions, which is equivalent
 to computing the scale function $W(x)$ due to relation (\ref{kyprianou-palmowski:forces-left-cts}).

In this section we will present some general ideas related to numerical evaluation of the scale function. We adopt the same notation as in the previous chapters. However, in order to avoid excessive technical details, we shall impose the following condition throughout.

\bigskip

\noindent
{\bf Assumption 1.}
 The L\'evy measure $\Pi$ has at most a  finite number of atoms when $X$ has paths of bounded variation.

\bigskip

\noindent Recall that the scale function $W^{(q)}(x)$ is defined by the Laplace transform identity 
 \begin{equation}
 \int_0^{\infty} \E^{-zx} W^{(q)}(x) \d x = \frac{1}{\psi(z)-q}, \;\;\; \re(z) > \Phi(q).
 \end{equation}
 We know from Lemma \ref{C1} and Corollary \ref{corr-for-alexey}  that $W^{(q)}{}'(x)$  exists and is continuous everywhere except when $X$ has paths of bounded variation, in which case the derivative does not exist  at any point $x$ such that  an atom of  of the L\'evy measure occurs at $-x$.  
Since we have assumed that there exists just a finite number of these points, we can apply standard results (such as Theorem 2.2 in \cite{Cohen}) 
and conclude that $W^{(q)}(x)$ can be expressed via the Bromwich integral  
\begin{equation}\label{eqn_Wq_inv_Laplace_transform}
 W^{(q)}(x)=\frac{\E^{cx}}{2\pi} \int\limits_{\r} \frac{\E^{\i u x}}{\psi(c+\i u)-q} \d u, \;\;\; x\in \r.
\end{equation}
where $c$ is an arbitrary constant satisfying $c>\Phi(q)$.

In principle one could use (\ref{eqn_Wq_inv_Laplace_transform}) as a starting point for numerical Laplace inversion to give 
$W^{(q)}(x)$. However, this would not be a good approach from the numerical point of view. The problem here is the exponential factor $\E^{cx}$, which can be very large
and would amplify the errors present in the numerical evaluation of the integral. Ideally we would like to choose $c$ to be a small
positive number, but this is not possible due to the restriction $c>\Phi(q)$.  
Therefore this method would be reasonable from the numerical point of view only when $q=0$ and $\Phi(0)=0$, 
in which case the function $W^{(q)}(x)$ does not increase exponentially fast. Indeed in such cases there is at most linear growth. This follows on account of the fact that when $q=\Phi(0) = 0$
the underlying L\'evy process does not drift to $-\infty$ and $W$ is the renewal measure of the descending ladder height process; recall the discussion preceding (\ref{a}). Thanks to (\ref{a}), this in turn implies that $W(x)$ is a renewal function and hence grows at most linearly  as $x$ tends to infinity.

In all other cases we would have to modify (\ref{eqn_Wq_inv_Laplace_transform}) in order to remove the exponential
growth of $W^{(q)}(x)$. It turns out that this can be done quite easily with the help of the density of the potential measure of the dual process  $\hat X=-X$, defined as
 \begin{equation}\label{def_uq}
 \int_0^{\infty} \E^{-qt } \p(\hat X_t \in \d x) \D t= \hat u^{(q)}(x) \d x.
  \end{equation}
We know from Theorem \ref{resolve} (iv), this function satisfies 
 \begin{equation}\label{eqn_Wq_uq}
  W^{(q)}(x)=\frac{\E^{\Phi(q)x}}{\psi'(\Phi(q))} -\hat u^{(q)}(x), \;\;\; x \ge 0,
  \end{equation}
 therefore $\hat u^{(q)}(x)$ is continuous on $(0,\infty)$ and has finite left and right derivatives at every point $x>0$. 
Theorems \ref{C2}, \ref{II}  and \ref{III} give us more information on the relation between the L\'evy measure of $X$ and the smoothness properties of $\hat u^{(q)}(x)$. 
 As we see from identity (\ref{eqn_Wq_uq}), the problem of computing the scale function $W^{(q)}(x)$ is equivalent to that of computing $\hat u^{(q)}(x)$. Dealing with the latter turns out to be an easier problem from the numerical point of view, providing that $q\geq 0$ and $\Phi(q)>0$, since in this case $u^{(q)}$ is bounded.  To see why this is the case, note from  the earlier  representation of $W^{(q)}$ in (\ref{interesting-representation}) together with (\ref{exp-shift}), we have that
 \begin{eqnarray*}
 \hat{u}^{(q)}(x)& =& \frac{\E^{\Phi(q)x}}{\psi'(\Phi(q))} - W^{(q)}(x)\\
 &=& \frac{\E^{\Phi(q)x}}{\psi'(\Phi(q))} - \E^{\Phi(q)x}\frac{1}{\psi'_{\Phi(q)}(0+)}P_x^{\Phi(q)}(\underline{X}_\infty \geq 0)\\
 &=&\frac{\E^{\Phi(q)x}}{\psi'(\Phi(q))}P_x^{\Phi(q)}(\underline{X}_\infty < 0)\\
 &=&\frac{\E^{\Phi(q)x}}{\psi'(\Phi(q))}P_x^{\Phi(q)}(\tau^-_0<\infty)\\
 &=&\frac{1}{\psi'(\Phi(q))}E_x\left(\E^{\Phi(q) X_{\tau^-_0} - q\tau^-_0} \mathbf{1}_{\{\tau^-_0<\infty\}}\right)\\
 &<&\frac{1}{\psi'(\Phi(q))},
 \end{eqnarray*}
thereby showing boundedness. Note that similar computations to the above can be found in Tak\'acs \cite{tak67} and Bingham \cite{kyprianou-palmowski:bing}.

We need to  characterize the Laplace transform of $\hat u^{(q)}(x)$. 
The proof of the next proposition follows easily from (\ref{eqn_Wq_uq}). 
\begin{proposition}\label{uq_laplace_transform}
 Assume that $\Phi(q)>0$. Then for $\re(z)>0$ 
  \begin{equation}\label{eqn_Laplace_transform_uq}
 \int\limits_0^{\infty} \E^{-zx} \hat u^{(q)}(x) \d x=F^{(q)}(z), 
  \end{equation}
where
 \begin{equation}\label{def_Fqz}
 F^{(q)}(z)=\frac{1}{\psi'(\Phi(q))(z-\Phi(q))}-\frac{1}{\psi(z)-q}.
  \end{equation}
\end{proposition}

Let us summarize our approach to computing the scale functions.  When $q=0$ and $\Phi(0)=0$ we will work with the scale function itself and will use (\ref{eqn_Wq_inv_Laplace_transform}) 
as the starting point for our computations. As earlier noted, in this case $W^{(q)}(x) = W(x)$  grows at worst linearly fast as $x\to +\infty$ and we do not need to worry
about amplifications of numerical errors. When $q\geq 0$ and $\Phi(q)>0$ the scale function grows exponentially fast as $x\to +\infty$, thus it is better to 
 work with the density of the potential measure $\hat u^{(q)}(x)$, and then to recover the scale function via relation (\ref{eqn_Wq_uq}). For convenience however we shall restrict ourselves to the latter case with the following blanket assumption. 
 
 \bigskip
 
\noindent {\bf Assumption 2.} For $q\geq 0$, $\Phi(q)>0$.
 
 \bigskip
 
\noindent The analysis for the case $q= \Phi(q) = 0$, where we work directly with the scale function, is no different.

As we will see later, our algorithms will require the evaluation of $F^{(q)}(z)$ for values of $z$ in the half-plane $\re(z)>0$. 
From (\ref{def_Fqz}) we find that  $F^{(q)}(z)$ has a removable singularity at $z=\Phi(q)$. 
It is not advisable to compute $F^{(q)}(z)$ via (\ref{def_Fqz}) when $z$ is close to $\Phi(q)$, as this procedure would involve 
subtracting two large numbers, which would cause the loss of accuracy. 
There are two solutions to this problem. One should either make sure that $z$ is never too close to $\Phi(q)$ or alternatively one should use the following asymptotic expression for $F^{(q)}(z)$.
\begin{proposition}\label{prop_Fq_near_Phiq}
 Define $a_n=\psi^{(n)}(\Phi(q))$ where $\psi^{(n)}$ is the $n$-th derivative of $\psi$. Then as $z\to \Phi(q)$  
 \begin{equation}\label{F_asymptotic_z_to_Phiq}
  F^{(q)}(z)=\frac12 \frac{a_2}{a_1^2} + \left[ \frac16 \frac{a_3}{a_1^2} - \frac14 \frac{a_2^2}{a_1^3} \right] (z-\Phi(q)) + O((z-\Phi(q))^2).
 \end{equation}
\end{proposition}
\begin{proof}
 The proof follows easily by writing down the Taylor expansion of the right-hand side in (\ref{def_Fqz}) centered at $z=\Phi(q)$.
\end{proof}

As we have seen in Section \ref{section:fluctuation identities} and in many other instances, virtually all fluctuation identities for spectrally negative L\'evy processes can be expressed in terms of 
the following three objects: the scale function $W^{(q)}(x)$, its derivative $W^{(q)}{}'(x)$ and a function $Z^{(q)}(x)$ defined by (\ref{zedque}), which is 
essentially an indefinite integral of the scale function. 
Therefore, in addition to the scale function itself, it is also important 
to be able to compute its derivative and indefinite integral. 
It turns out that computation of all these quantities can be done in exactly the same way.
Again, the first step is to remove the exponential growth as $x\to +\infty$ and express everything in terms 
of the potential density $\hat u^{(q)}(x)$ as follows
 \begin{equation}\label{eqn_Wq_prime_uq}
  W^{(q)}{}'(x)=\frac{\Phi(q)\E^{\Phi(q)x}}{\psi'(\Phi(q))} -\hat u^{(q)}{}'(x), \;\;\; x \ge 0,
  \end{equation}
 and
 \begin{equation}\label{eqn_Zq_vq}
  Z^{(q)}(x)=1+\frac{q}{\Phi(q)}  \frac{\E^{\Phi(q)x}-1}{\psi'(\Phi(q))} - q v^{(q)}(x),
 \end{equation}
 where we have defined
\begin{equation}\label{def_vqx}
 v^{(q)}(x)=\int_{0}^x \hat u^{(q)}(y) \d y.
\end{equation}
 We see that the problem of computing $W^{(q)}{}'(x)$ and $Z^{(q)}(x)$ is equivalent to the problem of computing $\hat u^{(q)}{}'(x)$ and $v^{(q)}(x)$. 
The following result is an analogue of proposition \ref{uq_laplace_transform} and it gives us Laplace transforms of $\hat u^{(q)}{}'(x)$ and $v^{(q)}(x)$.
\begin{proposition}\label{prop_Laplace_uq_vq}
 Assume that $q\geq 0$  and $\Phi(q)>0$. Then for $\re(z)>0$
\begin{equation}\label{Laplace_uq_prime}
 \int_0^{\infty} \hat u^{(q)}{}'(x) \E^{-zx} \d x=  z F^{(q)}(z) +W^{(q)}(0^+)-\frac{1}{\psi'(\Phi(q))}, 
\end{equation}
and
\begin{equation}\label{Laplace_vq}
 \int_0^{\infty} v^{(q)}(x) \E^{-zx} \d x = \frac{F^{(q)}(z)}{z}.
\end{equation}
\end{proposition}
\begin{proof}
 The proof follows from proposition \ref{uq_laplace_transform} and integration by parts.
\end{proof}

Now we have reduced all three problems to an equivalent ``standardized'' form. 
Equations (\ref{eqn_Wq_uq}), (\ref{eqn_Wq_prime_uq}), (\ref{eqn_Zq_vq}) and  Propositions (\ref{uq_laplace_transform}) and (\ref{prop_Laplace_uq_vq}) show that 
the problem of computing  of  $W^{(q)}(x)$, $W^{(q)}{}'(x)$ and $Z^{(q)}(x)$  is equivalent to a standard Laplace inversion problem. 
One has an explicit expression for the function $g(z)$, which is analytic in the half-plane $\re(z)>0$ and is equal to the Laplace transform of $f(x)$
\begin{equation}\label{eqn_f_g_Laplace}
  \int_0^{\infty} \E^{-zx} f(x) \d x= g(z), \;\;\; \re(z)>0,
\end{equation}
and one wants to compute the function $f(x)$ for several values of $x>0$. 
From now on we will concentrate on solving this problem.

This Laplace inversion problem has been well-studied and it has generated an enormous amount of literature. We will just mention here an excellent textbook
by Cohen \cite{Cohen}, very helpful reviews and research articles by Abate, Choudhurry, Whitt and Valko \cite{AbateWhitt2006, AbateChWhitt1999, AbateValko2006}
and works of Filon \cite{Filon}, Bailey and Swarztrauber \cite{Bailey_FFT} and Iserles \cite{Iserles}.

 This chapter is organized as follows. In Sections \ref{sec_Filon} and \ref{sec_mp_methods} we present four general numerical methods for computing $f(x)$. The first approach starts with the Bromwich integral and is based on 
Filon's method coupled with fractional discrete Fourier transform and Fast Fourier Transform techniques. The next three
methods, i.e. Gaver-Stehfest, Euler and Talbot algorithms, also start with the Bromwich integral, but the discretization of this integral is done in a 
different way, and in every case (except for Talbot method) some  acceleration procedure is applied. 
These three methods typically require multi-precision arithmetic. 
In Sections \ref{sec_rational} and \ref{sec_meromorphic} we present two families of processes for which the scale function can be computed explicitly, 
as a finite sum or an infinite series, which involve exponential functions, derivative of the Laplace exponent $\psi(z)$ and the solutions to equation $\psi(z)=q$. In this case computing the scale function can be done in an extremely efficient and accurate way, and later we will use these processes as benchmarks to test the 
performance of the four general methods. In Section \ref{sec_numerical} we discuss the results of several
numerical experiments and in Section \ref{sec_conclusion} we present our conclusions and provide some recommendations on how to choose the right numerical algorithm.

\section{Filon's method and fractional fast Fourier transform}\label{sec_Filon}

The starting point for Filon's method is an expression for $f(x)$ which identifies it as a Bromwich integral. That is,
\begin{equation}\label{uqx_bromwich_1}
f(x) = \frac{1}{2\pi \i } \int\limits_{c+\i \r} g(z) \E^{zx} \d z, \;\;\; x > 0,
\end{equation}
where $c$ is an arbitrary positive constant.  
 Since we are only interested in $f(x)$ for positive values of $x$, e	equation (\ref{uqx_bromwich_1}) can be written in
terms of the cosine transform as follows,
\begin{equation}\label{eq_cosine_transform}
 f(x)=\frac{2\E^{cx}}{\pi } \int_0^{\infty} \re\left[ g(c+\i u) \right] \cos(ux) \d u, \;\;\; x\ge 0.
\end{equation}
Now our plan is to compute the above integral numerically. 

The integral in (\ref{eq_cosine_transform}) is an {\it oscillatory integral}, which means that the integrand is an oscillating function.
Evaluating oscillatory integrals is not as straightforward as it may seem, and one should be careful to choose the right numerical method. For example, in 
 (\ref{eq_cosine_transform}), the cosine function can cause problems. When $x$ is large it oscillates rapidly with the period $2\pi /x$. 
Therefore if we simply 
discretize this integral using the trapezoid or Simpson's rule, we have to make sure that the spacing of the 
discretization $h$ satisfies  $h\ll 2\pi/x$. When $x$ is large, this restriction forces us to take $h$ to be a very small number, 
therefore we need a huge number of discretization points and the algorithm becomes slow and inefficient. 
The main benefit of Filon's method is that it helps one to avoid all these problems. 

Filon's method applies to computing oscillatory integrals over a finite interval $[a,b]$ of the form 
\begin{equation}\label{def_Fcfy}
  {\mathcal F}_c G (x)=\int_a^b G(u) \cos(ux) \d u.
\end{equation}
Here ${\mathcal F}_c G$ stands for {\it cosine transform} of the function $G$. In order to describe the intuition behind Filon's method, let us revisit Simpson's rule. It is well known that Simpson's rule for the integral in (\ref{def_Fcfy}) can be obtained in the following way. 
We discretize the interval $[a,b]$, approximate function $u \mapsto G(u) \cos(ux)$ by the second order Lagrange interpolating polynomials in the $u$-variable on each subinterval, and then integrate this approximation over $[a,b]$. Filon's method goes along the same steps,
except for one crucial difference. The function $G(u)$ is approximated by second order Lagrange interpolating polynomials,  which is then multiplied by $\cos(ux)$ and integrated over $[a,b]$. Due to the fact that the product of trigonometric functions and polynomials can be integrated explicitly, we still have an explicit formula. In doing this we separate the effects of oscillation and approximation. The approximation for the function $G(u)$ is usually  quite smooth and does not change very fast. Then any effect of oscillation disappears since we compute the integral against $\cos(ux)$ explicitly.

Let us describe this algorithm in full detail. Our main references are \cite{Filon}, \cite{Fosdick} and \cite{Iserles}. We take $N$ to be an integer number and define $h=(b-a)/(2N)$ and $u_n=a+nh$, $0\le n \le 2N$. Then we 
denote $g_{n}=G(u_n)$, $0\le n \le 2N$ and introduce the vectors ${\bf u}=[u_0,u_1,\dots,u_{2N}]$ and 
${\bf g}=[g_0,g_1,\dots,g_{2N}]$. For $k \in \{1,2\}$ we define
\begin{equation}\label{def_cfxy}
 {\mathcal C}_k({\bf g},{\bf u},x)=\sum\limits_{n=0}^{N-1} g_{2n+k} \cos(x u_{2n+k}).
\end{equation}
Next, on each subinterval $[u_{2n}, u_{2n+2}]$, 
$0\le n < N$ we approximate $G(u)$ by a Lagrange polynomial of degree two, which gives us a composite approximation of the following form
\begin{eqnarray*}
 G(u;N)&=& \sum\limits_{n=0}^{N-1} {\mathbf 1}_{\{u_{2n} \le u < u_{2n+2}\}}
 \bigg[ g_{2n+1} + \frac{1}{2h} (g_{2n+2}-g_{2n}) (u-u_{2n+1}) \\ &&\qquad \qquad + \frac{1}{2h^2} (g_{2n+2}-2g_{2n+1} + g_{2n} ) (u-u_{2n+1})^2 \bigg].
\end{eqnarray*}
Multiplying  this function by $\cos(ux)$, integrating over the interval $[a,b]$ and simplifying the resulting expression we obtain
the final expression for Filon's method,
\begin{eqnarray}\label{Filons_method}
{\mathcal F}_c G(x;N)&=&\int\limits_a^b G(u;N) \cos(ux) \d u \\ \nonumber 
 &=& h A(hx) (G(b) \sin(b x)-G(a) \sin(a x)) \\ \nonumber
  &+&
h B(hx) \left[ {\mathcal C}_2({\bf g},{\bf u},x) -\frac12 (G(b) \cos(bx) - G(a) \cos(ax)) \right] \\ 
\nonumber &+& h C(hx)  {\mathcal C}_1({\bf g},{\bf u},x),
\end{eqnarray}
where 
\begin{eqnarray}
 A(\theta)&=&\frac{1}{\theta} + \frac{\sin(2\theta)}{2\theta^2}-\frac{2\sin(\theta)^2}{\theta^3}, \\
 B(\theta)&=&2\left[ \frac{1+\cos(\theta)^2}{\theta^2}-\frac{\sin(2\theta)}{\theta^3} \right], \\
 C(\theta)&=& 4 \left[ \frac{\sin(\theta)}{\theta^3}-\frac{\cos(\theta)}{\theta^2} \right].
\end{eqnarray}
Note that by construction,  Filon's approximation
is exact for polynomials of degree two or less. It is also known that the error of Filon's approximation is $O(h^3)$, provided that $G^{(3)}(u)$ is continuous,
see \cite{Fosdick}.

We see that in order to evaluate Filon's approximation (\ref{Filons_method}) we have to compute two finite sums 
${\mathcal C}_k({\bf g},{\bf u},x)$ defined by (\ref{def_cfxy}). If we want to compute $ {\mathcal F}_c G (x;N)$ for just a single value of $x$ then
this obviously requires $O(N)$ operations. By the same reasoning, if we want to compute
$ {\mathcal F}_c G (x;N)$ for $N$ equally spaced points ${\bf x}=[x_0,x_1,...,x_{N-1}]$, then we would have to perform $O(N^2)$ operations. 
However, the special structure of these sums makes it possible to perform the latter computations in just $O(N\ln(N))$ computations. The main idea 
 is to use {\it fast Fourier transform} (FFT). This works as follows. Assume that we want 
to compute $ {\mathcal F}_c G (x;N)$ for $N_x=N$ values $x_m=x_0+m \delta_x$, $0\le m < N$. We rewrite the finite sums in (\ref{Filons_method}) in the following form
\begin{equation}\label{formula_FFT}
{\mathcal C}_k({\bf g},{\bf u},x_m)=\re\left[ \E^{\i (a+kh) x_m} \sum\limits_{n=0}^{N-1} g_{k,n}  \E^{ \i (2h \delta_x) n m } \right],
\end{equation}
where we have defined $g_{k,n}=g_{2n+k}  \E^{ \i 2h n x_0}$. Then, assuming that parameters $h$ and $\delta_x$ satisfy 
\begin{equation}\label{h_delta_x_restriction}
h\delta_x=\frac{\pi}{N},
\end{equation}
 the expression in the right-hand side of (\ref{formula_FFT}) is exactly in the form of the discrete Fourier transform (see \cite{Bailey_FFT}), and it is well-known
that it can be evaluated {\it for all } $m=0,1,..,N-1$  in just $O(N\ln(N))$ using the fast Fourier transform technique. 

The restriction (\ref{h_delta_x_restriction}) is quite unpleasant as it does not allow us to choose the spacing between the discretization points in the $u$-domain in 
(\ref{def_Fcfy}) independently of the spacing in the $x$-domain. However there is an easy solution to this problem, namely the fractional discrete Fourier transform (see \cite{Bailey_FFT}), which is defined as a linear transformation, which maps a vector ${\mathbf v}=[v_0,v_1,\dots,v_{N-1}]$ into a vector ${\mathbf V}=[V_0,V_1,\dots,V_{N-1}]$
\begin{equation}\label{dfrac}
 V_m=\sum\limits_{n=0}^{N-1} v_{n}  \E^{ \i \alpha n m }, \;\;\; m=0,1,\dots,N-1.
\end{equation}
It turns out that for any $\alpha \in \r$ one can still compute the values of $V_m$ for $m=0,1,..,N-1$ in just $O(N\ln(N))$ operations, see \cite{Bailey_FFT}
 for all the details.

Let us summarize the main steps of algorithm. 
 First, choose a small number $c>0$, so that the factor $\E^{cx}$ is not too large for the values of $x$ that interest us. 
 Then set $a=0$ and choose the value of the cutoff, i.e. a large number $b>0$ such that the integral 
\begin{equation*}
\int_b^{\infty} \re\left[ g(c+\i u) \right] \cos(ux) \d u
\end{equation*}
is sufficiently small. Then, choose the number of discretization points $N$ in the $u$-domain and define $u_n=nb/(2N)$, $0\le n \le 2N$ and 
$g_n=\re\left[ g(c+\i u_n) \right]$. Choose $\delta_x$ and $x_0$ and compute the approximation (\ref{Filons_method}) using the fractional Fast Fourier Transform. 
This gives us $N$ values of  $f(x_0+m \delta_x)$ for $0\le m < N$, with the error bound $O(N^{-3})$, at the computational cost of $O(N\ln(N))$ operations.

There is another trick that might be very useful when implementing Filon's method.  In many examples the integrand $\re\left[ g(c+\i u) \right]$ in
(\ref{eq_cosine_transform}) changes quite rapidly when $u$ is small while it changes slowly for large values of $u$. This means that we would have better precision and would need fewer discretization points if we were able to place more of them near $u=0$ and fewer of them for large $u$. This can be easily achieved by dividing the domain of integration
\begin{equation}\label{subdividing_domain}
\int_0^{b} \re\left[ g(c+\i u) \right] \cos(ux) \d u= \sum \limits_{j=0}^{n-1} \int_{b_j}^{b_{j+1}} \re\left[ g(c+\i u) \right] \cos(ux) \d u,
\end{equation}
where $0=b_0<b_1<\dots<b_n=b$. Each integral over $[b_j,b_{j+1}]$ can be evaluated using Filon's method to produce results on the same grid of the $x$-variable.
Choosing $b_j$ so that the spacing $b_{j+1} - b_j$ increases allows us to concentrate more points where they are needed (near $u=0$) and fewer points
in the regions far away from $u=0$.

\section{Methods requiring multi-precision arithmetic}\label{sec_mp_methods}

The methods presented in this section can give excellent performance, but the price that one has to pay is that they all require multi-precision arithmetic.
Our main references for this section are \cite{AbateWhitt2006}, \cite{AbateValko2006}, \cite{AbateChWhitt1999} and \cite{Cohen}. These methods are grouped together since they all give a similar expression for the approximation
to $f(x)$ 
\begin{equation*}
 f(x) \sim \frac{1}{x} \sum\limits_{n=0}^M a_n g\left(\frac{b_n}{x} \right),
\end{equation*}
where the coefficients $a_n$ and $b_n$ depend only on $M$ and do not depend on functions $f$ and $g$.

\subsection{The Gaver-Stehfest algorithm}\label{subsec_Gaver}
The first method that we discuss is the Gaver-Stehfest algorithm, see \cite{AbateWhitt2006} and Section 7.2 in \cite{Cohen}. 
This algorithm is based on the Gaver's approximation \cite{Gaver}, which can be considered as a discrete analogue of the Post-Widder formula (see Section 2.3 in \cite{Cohen})
\begin{equation*}
f(x)=\lim\limits_{k\to \infty} \frac{(-1)^k}{k!} \left(\frac{k}{x}\right)^{k+1} g^{(k)}\left( \frac{k}{t} \right).
\end{equation*}
It turns out that both Gaver's and Post-Widder formulas have a very slow convergence rate, therefore one has to apply some acceleration algorithm, 
such as Salzer transformation for the Gaver's formula, which was proposed by Stehfest \cite{Stehfest}, or Richardson extrapolation 
for Post-Widder formula which was used by Veillette and Taqqu \cite{Veilette_Taqqu}. 

We refer to \cite{AbateWhitt2006} for all the details and background on the Gaver-Stehfest algorithm, here we just present the final expression.
The function $f(x)$ is approximated by $f^{GS}(x;M)$, which depends on a single integer parameter $M$ and is defined as
\begin{equation}\label{Gaver_method}
f^{GS}(x;M)=\frac{\ln(2)}{x} \sum\limits_{n=1}^{2M} a_n g\left( n \ln(2)x^{-1} \right),
\end{equation}
where the coefficients $a_n$ are given by
\begin{equation*}
 a_n=(-1)^{M+n} \sum\limits_{j=[(n+1)/2]}^{n \wedge M} \frac{j^{M+1} }{ M!} \binom{M}{j} \binom{2j}{j} \binom {j}{n-j}.  
\end{equation*}
Note that the coefficients $a_n$ are defined as finite sums of possibly very large numbers, and that these coefficients have alternating signs. 
This means that we will lose accuracy in (\ref{Gaver_method}) due to subtracting very large numbers, thus we have to use multi-precision arithmetic
to avoid this problem. 
Abate and Valko \cite{AbateValko2006} (see also \cite{AbateWhitt2006}) recommend the following ``rule of thumb''. If we want $j$ significant digits in our approximation, we should set $M=\lceil 1.1 j \rceil$ (the least integer greater than or equal to $1.1j$)  and set the system precision at $\lceil 2.2M \rceil$.  It is also useful to check the accuracy of computation of $a_n$ using the fact that
\begin{equation*}
 \sum\limits_{n=1}^{2M} a_n=0.
\end{equation*}

We see that Gaver-Stehfest algorithm should have efficiency around $0.9/2.2 \sim 0.4$, which is the ratio of the system precision to the 
number of significant digits produced by 
the approximation. While algorithms presented below all have higher efficiency, the Gaver-Stehfest algorithm has a big advantage that it does not require the use
of complex numbers. This can improve performance, since it is faster to perform computations with real numbers than with complex numbers. This feature is also useful, because not every multi-precision software has a built-in support for complex numbers.

\subsection{The Euler algorithm}\label{sec_Euler}

The next two algorithms are based on the Bromwich integral representation, which follows from (\ref{uqx_bromwich_1}) after the change of variables of 
integration $u \mapsto v/x$ 
\begin{equation}\label{uq_Bromwich_2}
 f(x)=\frac{\E^{c}}{2\pi x} \int\limits_{\r} g \left(\frac{c+\i v}{x} \right) \E^{\i v} \d v, \;\;\; x\in \r.
\end{equation}
Note that we have also changed $c \mapsto c/x$. 
The main idea behind Euler algorithm is to approximate the integral in (\ref{uq_Bromwich_2}) by a trapezoid rule and then apply Euler acceleration method 
to improve the convergence rate. Again, all the details can be found in  \cite{AbateChWhitt1999} and \cite{AbateWhitt2006}. We present here only 
the final form of this approximation 
\begin{equation}\label{Euler_method}
f^{E}(x;M)=\frac{10^{\frac{M}{3}}}{x} \sum\limits_{n=0}^{2M} (-1)^n a_n \re \left[  g\left( 
\left(\ln\left(10^\frac{M}{3}\right)+\pi \i n \right) x^{-1} \right) \right],
\end{equation}
where the coefficients $a_n$ are defined as
\begin{eqnarray*}
&& a_0=\frac{1}{2}, \;\;\; a_n=1, \; {\textnormal{ for }} \; 1\le n \le M, \;\;\; a_{2M}=2^{-M}, \\
 && a_{2M-k}=a_{2M-k+1}+2^{-M} \binom{M}{n}, \; {\textnormal{ for }} \; 1 \le  n < M.
\end{eqnarray*}
Note that while coefficients $a_n$ are real, formula (\ref{Euler_method}) still requires evaluation of $g(z)$ for complex values of $z$. 

In the case of Euler algorithm, Abate and Whitt \cite{AbateWhitt2006} recommend to set $M=\lceil 1.7 j \rceil$ if $j$ significant digits of  are required, and 
then set the system precision at $M$. Coefficients $a_n$ can be precomputed, and the accuracy can be verified via condition
\begin{equation*}
 \sum\limits_{n=0}^{2M} (-1)^n a_n=0.
\end{equation*}
We see that the efficiency of the Euler algorithm should be around $1/1.7 \sim 0.6$. 

\subsection{The fixed Talbot algorithm}\label{sec_Talbot}

The Talbot algorithm \cite{Talbot} also starts with integral representation (\ref{uq_Bromwich_2}), but then the contour of integration is transformed so that $\re(z) \to -\infty$ on this contour. Note that $\re(z)$ is constant on the contour of integration in the Bromwich integral (\ref{uq_Bromwich_2}). 
This transformation of the contour of integration has a great benefit in that the integrand $g(z) \exp(zx)$ converges to zero much faster. On the negative side, this method 
only works when (i) $g(z)$ can be analytically continued into the domain $|{\textnormal{arg}}(z)|<\pi$, and  (ii)  $g(z)$ has no singularities far away from the negative half-line. 
In our case these two conditions are satisfied for processes whose jumps have completely monotone density, as in this case all of the singularities
of $F^{(q)}(z)$ lie on the negative half-line. Meromorphic L\'evy  processes, presented in Section \ref{sec_meromorphic}, exhibit this property for example. It is also worthy of note for future reference that such processes are dense in the class of 
processes with completely monotone jumps. As we will see later, Talbot method gives an excellent performance for processes with completely monotone jumps.

Again, we refer to \cite{AbateWhitt2006}, \cite{AbateValko2006} and \cite{Cohen} for all the details, and present here only the final form of the 
approximation
\begin{equation}\label{Talbot_method}
f^{T}(x;M)=\frac{1}{x} \sum\limits_{n=0}^{M-1} \im \left[ a_n 
g\left( b_n x^{-1} \right) \right],
\end{equation}
where 
\begin{equation*}
b_0=\frac{2M}{5}, \;\;\; b_n=\frac{2 \pi n}{5} \left( \cot\left(\frac{\pi n}{M} \right)+\i\right),  \; {\textnormal{ for }} \; 1 \le  n < M,
\end{equation*}
and 
\begin{equation*}
a_0=\frac{\i}{5} \E^{b_0}, \;\;\; a_n= \left(b_n- \frac{5}{2M} |b_n|^2\right) \frac{\E^{b_n}}{n \pi},  \; {\textnormal{ for }} \; 1 \le  n < M.
\end{equation*}
For this algorithm, Abate and Valko \cite{AbateValko2006} (see also \cite{AbateWhitt2006}) recommend using 
the same precision parameters as for the Euler algorithm:  one should set $M=\lceil 1.7 j \rceil$ if $j$ significant digits of  are required, and 
then set the system precision at $M$. The efficiency of this algorithm should also be close to $0.6$.

\section{Processes with jumps of rational transform}\label{sec_rational}

It is well known that the Wiener-Hopf factorization and many related fluctuation identities can be obtained in closed form 
for processes with jumps of rational transform. The Wiener-Hopf factorization for the two-sided processes having positive and/or negative phase-type jumps  
(a subclass of jumps of rational transform) was studied by
Mordecki \cite{Mordecki2},  Asmussen, Avram and Pistorius \cite{Asmussen}  and Pistorius \cite{Pistorius_2006}, while Wiener-Hopf factorization for a more general class of processes having positive jumps of rational transform was obtained by Lewis and Mordecki \cite{Mordecki}. 
Expressions for the scale functions for spectrally-negative processes with jumps of rational transform follow implicitly from these papers, 
their explicit form was obtained in a recent paper by Egami and Yamazaki \cite{Egami}. 

In order to specify these processes, let us define the density of the L\'evy measure as follows
\begin{equation}\label{def_pi_rational}
\pi(x)={\mathbf 1}_{\{x<0\}} \sum\limits_{j=1}^m a_j |x|^{m_j-1} \E^{\rho_j x},
\end{equation}
where $m_j \in {\mathbb N}$ and $\re(\rho_j)>0$. It is easy to see that the Laplace exponent is given by
\begin{equation}\label{rational_Laplace_exponent}
 \psi(z)=\frac{\sigma^2}2 z^2 + \mu z + \sum\limits_{j=1}^m a_j (m_j-1)! \left[ (\rho_j+z)^{-m_j}- \rho_j^{-m_j} \right]
\end{equation}
and that $\psi(z)$ is a rational function
\begin{equation}\label{eqn_psi_rational}
 \psi(z) = \frac{P(z)}{Q(z)},
\end{equation}
where ${\textnormal {deg}}(Q)=M=\sum_{j=1}^m m_j$ and ${\textnormal {deg}}(P)=N$, where $N=M+2$ (resp. $M+1$) if $\sigma>0$ (resp. $\sigma=0$). The next proposition
gives us valuable information about solutions of the equation $\psi(z)=q$.
\begin{proposition}\label{prop_roots_rational}
\indent
\begin{description}[(iii)]
 \item[(i)]   For $q > 0$ or $q=0$ and $\psi'(0)<0$ equation $\psi(z)=q$ has one solution $z=\Phi(q)$ in the half-plane $\re(z) > 0$ and $N-1$ solutions in
the half-plane $\re(z)<0$. 
 \item[(ii)]  For $q=0$ and $\psi'(0)>0$ (resp. $\psi'(0)=0$) equation $\psi(z)=q$ has a solution $z=0$ of multiplicity one (resp. two)
and $N-1$ (resp. $N-2$) solutions in the half-plane $\re(z)<0$. 
 \item[(iii)]  There exist at most $M+N-1$ complex numbers $q$ such that the equation $\psi(z)=q$ has solutions of multiplicity greater than one.
\end{description}
\end{proposition}
\begin{proof}
 The proof of (i) and (ii) is trivial and follows from (\ref{eqn_psi_rational}) and the general theory of scale functions. 
Let us prove (iii). Assume that the equation $\psi(z)=q$ has a solution $z=z_0$
of multiplicity greater than one, therefore $\psi'(z_0)=0$. Using (\ref{eqn_psi_rational}) we find that $\psi'(z_0)=0$ implies $P'(z_0)Q(z_0)-P(z_0)Q(z_0)=0$. 
The polynomial $H(z)=P'(z)Q(z)-P(z)Q(z)$ has degree $M+N-1$, thus there exist at most $M+N-1$ distinct points $z_k$ for which $\psi'(z_k)=0$, 
which implies that there exist at most $M+N-1$ values $q$, given by $q_k=\psi(z_k)$, for which the equation $\psi(z)=q$ has solutions of multiplicity greater than one. 
\end{proof}

The statement in Proposition \ref{prop_roots_rational} (iii) is quite important for numerical calculations. While it's proof is quite elementary, we were not able to locate this result in the existing literature. The implication of this result is that for a generic 
L\'evy measure defined by (\ref{def_pi_rational}) and general $q\ge 0$ it is extremely unlikely that equation $\psi(z)=q$ has solutions
of multiplicity greater than one.
So for all practical purposes (unless we are dealing with the case $\psi'(0)=q=0$) we can assume that all the solutions of $\psi(z)=q$ have multiplicity equal to one,
as we will see later this will considerably simplify all the formulas. 

Computing the scale function, or equivalently, the potential density $\hat u^{(q)}(x)$ is a trivial task for this class of processes. Let us consider the general case, and assume that equation $\psi(z)=q$ has $n$ distinct solutions $-\zeta_1,-\zeta_2,...,-\zeta_n$ in the half-plane $\re(z)<0$, and the multiplicity of 
 $z=-\zeta_j$ is equal to $n_j$.  Due to (\ref{eqn_psi_rational}) it is clear that $N-1=\sum_{j=1}^{n} n_j$. Rewriting the rational function $1/(\psi(z)-q)$ as partial fractions
\begin{equation}\label{eqn_partial_fractions}
\frac{1}{\psi(z)-q}=\frac{1}{\psi'(\Phi(q))(z-\Phi(q))}+\sum\limits_{j=1}^n \sum\limits_{k=1}^{n_j} \frac{c_{j,k}}{(z+\zeta_j)^k},
\end{equation}
and using Proposition \ref{uq_laplace_transform} we can identify $\hat u^{(q)}(x)$ as follows
 \begin{equation}\label{uqx_rational}
 \hat u^{(q)}(x)=-\sum\limits_{j=1}^n \E^{-\zeta_j x} \sum\limits_{k=1}^{n_j} \frac{c_{j,k} }{(k-1)!} x^{k-1}, \;\;\; x>0.
 \end{equation}

Note, that if all the solutions $\zeta_i$ have multiplicity one, i.e. $n_j=1$ and $n+1=N={\textnormal{deg}}(P)$, then (\ref{eqn_partial_fractions}) implies that
\begin{equation*}
 c_{j,1}=\frac{1}{\psi'(-\zeta_j)},
\end{equation*}
and we have a much simpler expression for the potential density
\begin{equation}\label{uqx_rational_multiplicity1}
  \hat u^{(q)}(x)=-\sum\limits_{j=1}^{N-1} \frac{\E^{-\zeta_j x}}{\psi'(-\zeta_j)}, \;\;\; x \ge 0.
\end{equation}
We would like to stress again that  Proposition \ref{prop_roots_rational} (iii) tells us that $\hat u^{(q)}(x)$ can be computed with the help of (\ref{uqx_rational_multiplicity1}) for all but a finite number of $q$, unless we are dealing with the case $\psi'(0)=q=0$.

\section{Meromorphic L\'evy processes}\label{sec_meromorphic}

The main advantage of processes with jumps of rational transform is that the numerical computations are very simple and straightforward. Everything
boils down to solving a polynomial equation $\psi(z)=q$ and performing a partial fraction decomposition of a rational function $1/(\psi(z)-q)$. 
On the negative side, it is clear that these processes can only have compound Poisson jumps, 
while  in applications it is often necessary to have processes with jumps of infinite activity or even of infinite variation. 
Meromorphic L\'evy processes, which were recently introduced  in \cite{meromorphic}, solve precisely this problem.  They allow for much more flexible modeling of the small jump behavior, yet all the computations can be done with the same efficiency as for processes with jumps of rational transform. 

In order to define a spectrally-negative Meromorphic process $X$, let us consider the function 
\begin{equation}\label{def_Levy_measure}
\pi(x)={\mathbf 1}_{\{x<0\}}\sum\limits_{j=1}^{\infty} a_j \E^{\rho_j x},
\end{equation}
where the coefficients $a_j$ and $\rho_j$ are positive and $\rho_j$ increase to $+\infty$ as $j \to +\infty$.  
As was shown in \cite{Kuz_theta}, convergence of the series 
\begin{equation*}
 \sum_{j\ge 1 } \frac{ a_j }{ \rho_j^{3}}<\infty
\end{equation*}
implies convergence of the integral
\begin{equation*}
 \int_{-\infty}^0 x^2 \pi (x) \d x < \infty,
\end{equation*}
thus $\pi(x)$ can be used to define the density of the L\'evy measure. 
Note that by constuction $\pi(-x)$ is a completely monotone function.

Using the L\'evy-Khintchine formula we find that the Laplace exponent is given by 
 \begin{equation}\label{eqn_psi_partial_fractions}
 \psi(z)=\frac 12 \sigma^2 z^2 + \mu z+z^2 \sum\limits_{j\ge 1} \frac{ a_j}{\rho_j^2 ( \rho_j+z)}
, \;\;\; z\in \c,
 \end{equation}
and we see that $\psi(z)$ is a meromorpic function which has only negative poles at points $z=-\rho_j$.  From the general theory we know that for $q\ge 0$ 
equation $\psi(z)=q$
has a unique solution $z=\Phi(q)$ in the half-plane $\re(z)>0$, and we also know that this solution is real.  
It can be proven (see \cite{Kuz_theta}) that the same is true for equation $\psi(-z)=q$. For $q\ge 0$ all the solutions $z=\zeta_j$ of $\psi(-z)=q$ in the halfplane $\re(z) \ge  0$ are real and they satisfy the interlacing property  
 \begin{equation}\label{interlacing_property}
  0\le \zeta_1<\rho_1 < \zeta_2 < \rho_2 < \dots.
 \end{equation}
When $q=0$ and $\psi'(0)=\e[X_1]\le 0$ we have $\zeta_1=0$, otherwise $\zeta_1>0$.

The following proposition gives an explicit formula for the potential density, generalizing (\ref{uqx_rational_multiplicity1}), which is expressed in terms of the roots $\zeta_j$ and the first derivative of the Laplace exponent. 
 \begin{proposition}\label{prop_meromorphic}

\indent
 \begin{description}[(iii)]
  \item[(i)] If $X$ is Meromorphic, then for all $q>0$ function $\hat u^{(q)}(x)$ is an infinite mixture of exponential functions with positive coefficients 
  \begin{equation}\label{uqx_meromorphic}
  \hat u^{(q)}(x)=-\sum\limits_{j=1}^{\infty} \frac{\E^{-\zeta_j x}}{\psi'(-\zeta_j)}, \;\;\; x \ge 0.
\end{equation}
 \item[(ii)] $\;$ If for some $q>0$ function $\hat u^{(q)}(x)$ is an infinite mixture of exponential functions with positive coefficients, then
 $X$ is a Meromorphic process.
 \end{description}  
 \end{proposition}

The first statement of the above proposition can be proved using Proposition \ref{uq_laplace_transform} and standard analytical techniques. See for example Corollary 2 and Remark 2  in \cite{meromorphic} as well as 
\cite{Kuz_Morales}. The second statement can be established using the same technique as in the proof of Theorem 1 in 
\cite{Kuz_theta}. In both cases  we leave all the details to the reader.

Proposition \ref{prop_meromorphic} shows that we can compute the potential density and the scale function very easily, provided that we know $\zeta_j$ and $\psi'(-\zeta_j)$. There seems little hope to find $\zeta_j$ explicitly for a general $q\geq 0$ and hence 
these quantities have to be computed numerically, which in turn requires multiple evaluations of $\psi(z)$. 
Computing $\psi(z)$ with the help of the partial fraction decomposition (\ref{eqn_psi_partial_fractions}) is not the best way to do it, as in general the series will converge rather slowly. 
Therefore it is important to find examples of Meromorphic processes for which $\psi(z)$ can be computed explicitly. Below we present several such examples,
the details can be found in \cite{Kuz_beta, Kuz_theta}.

\begin{description}[(iii)]
 \item[(i)]    $\;$$\theta$-process with parameter $\lambda\in\{3/2,5/2\}$
 \begin{eqnarray}
\nonumber
 \psi(z)=\frac{1}{2} \sigma^2 z^2 +\mu z &+& c(-1)^{\lambda-1/2} \left(\alpha+z/\beta \right)^{\lambda-1} 
 \coth\left( \pi \sqrt{\alpha+z/\beta} \right) \\  &-& c (-1)^{\lambda-1/2}\alpha^{\lambda-1} 
 \coth\left( \pi \sqrt{\alpha} \right) \;.\label{eq:theta_2}
 \end{eqnarray}
 \item[(ii)] $\;$ $\beta$-process with parameter $\lambda \in (1,2) \cup (2,3)$
 \begin{equation}
 \psi(z)=\frac{1}{2} \sigma^2 z^2 +\mu z + c {\textnormal{B}}(1+\alpha+z/\beta,1-\lambda)-c {\textnormal{B}}(1+\alpha,1-\lambda) \;, \label{eq:beta}
 \end{equation}
 where ${\textnormal{B}}(x,y)=\Gamma(x)\Gamma(y)/\Gamma(x+y)$ is the Beta function.
\end{description}

The admissible set of parameters is  $\sigma\ge 0$, $\mu \in \r$, $c>0$, $\alpha>0$ and $\beta>0$. 
The parameters $a_j$ and $\rho_j$, which define the L\'evy measure via (\ref{def_Levy_measure}), are given as follows (see \cite{Kuz_theta,Kuz_beta}): 
in the case of $\theta$-process we have
\begin{equation}\label{aj_rhoj_theta}
a_j=\frac{2}{\pi} c \beta j^{2\lambda-1} , \;\;\; \rho_j=\beta(\alpha+j^2)\;,
\end{equation}
and in the case of $\beta$-process 
 \begin{equation}\label{aj_rhoj_beta}
 a_j=c\beta \binom{j+\lambda-2}{j-1}, \;\;\; \rho_j=\beta(\alpha+j)\;.
\end{equation}
  In the case of $\beta$-process we also have an explicit formula for the density of the L\'evy measure,
 \begin{equation*} 
 \pi(x)= c\beta \frac{\E^{(1+\alpha) \beta x}}{(1-\E^{\beta x})^{\lambda}}, \;\;\; x<0\;.
 \end{equation*}

Moreover, it can be shown  
that the density of the L\'evy measure $\pi(x)$ satisfies  (up to a multiplicative constant)
 \begin{eqnarray*}
 \pi(x) &\sim& |x|^{-\lambda}, \;\; \textnormal{ as } \; x \to 0^- \;,\\
  \pi(x) &\sim& \E^{\beta(1+\alpha)x}, \;\; \textnormal{ as } \; x \to -\infty\;.  
 \end{eqnarray*}
In particular, we see that the L\'evy measure always has exponential tails and the rate of decay of $\pi(x)$ is controlled by parameters $\alpha$ and $\beta$.  See \cite{Kuz_theta,Kuz_beta}.
The parameter $c$ controls the overall ``intensity'' of jumps,  while $\lambda$ is responsible for the behavior of small jumps: if $\lambda \in (1,2)$ (resp. $\lambda \in (2,3)$) then the jump part of the process is of infinite activity and finite (resp. infinite) variation.

The beta family of processes is more general than the theta family, as it allows for a greater range of the parameter $\lambda$, which
gives us more flexibility in modeling the behavior of small jumps. 
However, processes in the theta family have the advantage that the Laplace exponent is given in terms of elementary functions, 
which helps to implement faster numerical algorithms.

Computing the scale function for Meromorphic processes is very simple. 
The first step is to compute the values of $\zeta_j$, which are defined as the solutions 
to $\psi(-z)=q$. The interlacing property (\ref{interlacing_property}) gives us left/right bounds for each $\zeta_j$ (recall that $\rho_j$ are known explicitly), thus
we know that on each interval $(\rho_j, \rho_{j+1})$ there is a unique solution to $\psi(-z)=q$, and this solution can be found very efficiently 
using bisection/Newton's method. 

There is one additional trick that can greatly reduce the computation time for this first step. We will explain the main idea on the example of a $\beta$-process. 
Formula  (\ref{aj_rhoj_beta}) shows that the spacing between the poles $\rho_j$ is constant and is equal to $\beta$. 
This fact and the behavior of $\psi(z)$ as $\re(z) \to \infty$ implies that for $j$ large, the difference $\zeta_{j+1}-\zeta_j$ would also be very close to $\beta$. Therefore, once we have computed $\zeta_j$ for a sufficiently large $j$, we can use $\zeta_j+\beta$ as a starting point for our search for $\zeta_{j+1}$. In practice, 
starting Newton's method from this point gives the required accuracy in just one or two iterations. A corresponding strategy can be 
developed for  the $\theta$-processes: one should use the fact that $\sqrt{\rho_j -\alpha \beta }$ have constant spacing equal to $\sqrt{\beta}$. 

Once we have precomputed the values of $\zeta_j$, we compute the coefficients $\psi'(-\zeta_j)$, which is an easy task since we have 
an explicit formula
for $\psi'(z)$. Now we can evaluate $\hat u^{(q)}(x)$ by truncating the infinite 
series in (\ref{uqx_meromorphic}). Note that the infinite series converges exponentially fast (and much faster in the case of $\theta$-process), 
so unless $x$ is very small, one really needs just a few  terms to have  good precision. 
On the other hand, when $x$ is extremely small, it is better to compute the scale function by
using the information about $W^{(q)}(0^+)$ and $W^{(q)}{}'(0^+)$  given in  Lemmas \ref{W(0)} and \ref{Wprime} and applying interpolation. 

\section{Numerical examples}\label{sec_numerical}

In this section we present the results of several numerical experiments. Before we start discussing the details of these experiments, let us 
describe the computing environment. All the code was written in Fortran90 and compiled with the help of Intel$\textsuperscript{\textregistered}$ Fortran compiler. For the methods which require multi-precision arithmetic we have the following two options. 
The first one is the {\it quad} data type (standard on Fortran90),  which uses 128 bits to represent a real number and 
allows for the precision of approximately 32 decimal digits. 
Recall that the {\it double} data type (standard on C/C++/Fortran/Matlab and other programming languages) 
uses 64 bits and gives approximately 16 decimal digits.  The second option is to use MPFUN90 multi-precision library
developed by Bailey \cite{BaileyMPFUN90}. This is an excellent library for Fortran90, which is very 
efficient and easy to use. It allows one to perform computations with precision of
thousands of digits, though in our experiments we've never used more than two hundred digits. 
All the computations were performed on a standard 2008 laptop (Intel Core 2 Duo 2.5 GHz processor and 3 GB of RAM).

In our first numerical experiment we will compare the performance of the four methods described in Sections \ref{sec_Filon} and \ref{sec_mp_methods}
 in the case of a $\theta$-process with $\lambda=3/2$ (process with jumps of finite variation and infinite activity), whose Laplace exponent is given by (\ref{eq:theta_2}). We will consider two cases $\sigma=0$ or $\sigma=0.25$ and will fix the other parameters as follows 
\begin{equation*}
  \mu=2, \; c=1,  \; \alpha=1, \; \beta=0.5.
\end{equation*}
Moreover, {\bf everywhere in this section we fix $q=0.5$}. 

As the benchmark for this experiment, we compute the values of the scale function $W^{(q)}(x)$ for one hundred values of $x$, given by $x_i=i/20$, $i=1,2,\dots,100$. 
This part is done using the algorithm described in Section \ref{sec_numerical}.  The infinite series (\ref{uqx_meromorphic}) is truncated at $j=1000$
and the system precision is set at 200 digits (using MPFUN90 library). This guarantees that our benchmark is within 1.0e-150 of the exact value.

Next we compute the approximation $\tilde W^{(q)}(x)$ using Filon/ Gaver-Stehfest/ Euler/ Talbot algorithms for the same set of values of $x$. 
As the measure of accuracy of the approximation we will take the maximum
of the relative error
\begin{equation}
{\textnormal{relative error}}=\max\limits_{1\le i \le 100} \frac{|\tilde W^{(q)}(x_i)-W^{(q)}(x_i)|}{ W^{(q)}(x_i)}.
\end{equation}

Let us specify the parameters for the Filon's method, as described in Section \ref{sec_Filon}. 
We truncate the integral at $b=10^7$ (or $b=10^9$)  in (\ref{subdividing_domain}) and divide the domain of integration into ten sub-intervals, i.e.
$n=10$. Moreover, we set $b_j=b (j/n)^5$ for $j=1,2,\dots,n$ 
so that there are more discretization points 
 close to $u=0$. The number of discretization points per each sub-interval $[b_j,b_{j+1}]$ is fixed at  $N_x=10^4$ (or $N_x=10^5$). 
Note that Filon's method as described in Section \ref{sec_Filon} produces $N_x$ values of $W^{(q)}(x)$. 
The spacing between the points in $x$-domain is chosen at 
$\delta_x=5/N_x$. In this way the grid $x_m=m\delta_x$ covers the whole interval $[0,5]$. 
In order to perform the Fast Fourier Transform we use Intel$\textsuperscript{\textregistered}$ MKL library.

\begin{table}[!ht]
\centering
\bigskip
\begin{tabular}{|c||c|c|c|c|c|}
\hline \rule[0pt]{0pt}{11pt}  
   \;\; algorithm \;\; & \;\;\; rel. error  \;\;\;  & \;\; time (seconds) \;\; & \;\;\; $N_x$ \;\;\; & \;\;\; $M$ \;\;\;  & \;\; system precision \;\;  \\ [0.5ex] \hline \hline
   \rule[0pt]{0pt}{11pt}    Filon ($b=10^7$)  &  3.0e-12  & 0.18 & 10 000 & -- & double (64 bits)     \\ [0.5ex]\hline
   \hline 
   \rule[0pt]{0pt}{11pt}    Gaver-Stehfest \; &  9.2e-12  & 0.4 & 100 & 20 & 44 (MPFUN90)     \\ [0.5ex]\hline
   \rule[0pt]{0pt}{11pt}    Euler  &  9.0e-14  & 0.05 & 100  & 20 & quad (128 bits)     \\ [0.5ex]\hline
   \rule[0pt]{0pt}{11pt}    Talbot  &  2.1e-13  & 0.025 & 100  & 20 & quad (128 bits)     \\ [0.5ex]\hline
   \hline 
   \rule[0pt]{0pt}{11pt}    Gaver-Stehfest \; &  2.9e-21  & 1.5 & 100 & 40 & 88 (MPFUN90)     \\ [0.5ex]\hline
   \rule[0pt]{0pt}{11pt}    Euler  &  2.7e-25  & 1.5 & 100  & 40 & 40 (MPFUN90)    \\ [0.5ex]\hline
   \rule[0pt]{0pt}{11pt}    Talbot  &  2.4e-25  & 0.7 & 100  & 40 & 40 (MPFUN90)     \\ [0.5ex]\hline
   \hline 
   \rule[0pt]{0pt}{11pt}    Gaver-Stehfest \; &  2.2e-39  & 6.3  & 100  & 80 & 176 (MPFUN90)     \\ [0.5ex]\hline
   \rule[0pt]{0pt}{11pt}    Euler 	      &  2.1e-48  & 5.0  & 100  & 80 & 80 (MPFUN90)    \\ [0.5ex]\hline
   \rule[0pt]{0pt}{11pt}    Talbot  	      &  3.1e-49  & 2.5  & 100  & 80 & 80 (MPFUN90)     \\ [0.5ex]\hline
\end{tabular}
\vspace{0.2cm}
\caption{Computing $W^{(q)}(x)$ for the $\theta$-process with $\sigma=0.25$.}\label{tab_theta1}
\end{table}

The results of this numerical experiment for computing $W^{(q)}(x)$ are presented in Table \ref{tab_theta1}. We see that each of the Gaver-Stehfest/Euler/Talbot algorithms produce very accurate results. When $M=40$ we obtain more than twenty significant digits and when $M=80$ we get almost fifty significant digits. 
The computation time varies with each algorithm and also depends on parameters. 
Note that when we use the quad date type for Euler/Talbot methods with $M=20$, the computations are very  fast, on the order of one hundredth of a second. When we use MPFUN90 library the computation time increases significantly. This is due to the fact that quad is a native data type in Fortran90 and computations done with an external multi-precision library are slower. 
Next we see that Talbot algorithm is always about two times faster 
than Euler method. This is due to the fact that the finite sum in (\ref{Talbot_method}) has $M$ terms while Euler method (\ref{Euler_method}) requires
summation of $2M+1$ terms. We would like to stress again that the time values presented in this table are for computing the scale function for $N_x$ different values
of $x$. For example, if one wants to compute the scale function for a {\it single} value of $x$ using Talbot method with $M=40$, it would take just {\it seven milliseconds}. 

 Filon's method also performs well, computing the scale function for $10 000$ values of $x$ with accuracy of around 12 decimal digits
 in just 0.18 seconds.
As we see from the table, the Talbot algorithm appears to be the most efficient. This is not surprising given the discussion in Section \ref{sec_Talbot} where it was pointed out that this algorithm is well suited for processes  with completely monotone jumps (see also results by Albrecher, Avram and Kortschak \cite{Albrecher} on ruin probabilities for completely monotone claim distributions).

\begin{table}[!ht]
\centering
\bigskip
\begin{tabular}{|c||c|c|c|c|c|}
\hline \rule[0pt]{0pt}{11pt}  
   \;\; algorithm \;\; & \;\;\; rel. error  \;\;\;  & \;\; time (seconds) \;\; & \;\;\; $N_x$ \;\;\; & \;\;\; $M$ \;\;\;  & \;\; system precision \;\;  \\ [0.5ex] \hline \hline
   \rule[0pt]{0pt}{11pt}    Filon ($b=10^9$)  &  8.0e-5  & 2.9 & $10^5$ & -- & double (64 bits)     \\ [0.5ex]\hline
   \hline 
   \rule[0pt]{0pt}{11pt}    Gaver-Stehfest \; &  4.6e-12  & 0.45 & 100 & 20 & 44 (MPFUN90)     \\ [0.5ex]\hline
   \rule[0pt]{0pt}{11pt}    Euler  &  1.7e-13  & 0.05 & 100  & 20 & quad (128 bits)     \\ [0.5ex]\hline
   \rule[0pt]{0pt}{11pt}    Talbot  &  1.5e-12  & 0.023 & 100  & 20 & quad (128 bits)     \\ [0.5ex]\hline
   \hline 
   \rule[0pt]{0pt}{11pt}    Gaver-Stehfest \; &  7.6e-21  & 1.5 & 100 & 40 & 88 (MPFUN90)     \\ [0.5ex]\hline
   \rule[0pt]{0pt}{11pt}    Euler  &  4.3e-25  & 1.4 & 100  & 40 & 40 (MPFUN90)    \\ [0.5ex]\hline
   \rule[0pt]{0pt}{11pt}    Talbot  &  2.9e-24  & 0.71 & 100  & 40 & 40 (MPFUN90)     \\ [0.5ex]\hline
   \hline 
   \rule[0pt]{0pt}{11pt}    Gaver-Stehfest \; &  1.7e-38  & 6.3  & 100  & 80 & 176 (MPFUN90)     \\ [0.5ex]\hline
   \rule[0pt]{0pt}{11pt}    Euler 	      &  2.8e-48  & 4.9  & 100  & 80 & 80 (MPFUN90)    \\ [0.5ex]\hline
   \rule[0pt]{0pt}{11pt}    Talbot  	      &  9.1e-48  & 2.4  & 100  & 80 & 80 (MPFUN90)     \\ [0.5ex]\hline
\end{tabular}
\vspace{0.2cm}
\caption{Computing $W^{(q)}{}'(x)$ for the $\theta$-process with $\sigma=0$.}\label{tab_theta2}
\end{table}

In Table \ref{tab_theta2} we present the results of computing the first derivative of the scale function. 
The results are very similar to those presented in 
Table \ref{tab_theta1}, the major difference now is that Filon's method is not producing a very good accuracy. 
This happens because the integrand in 
(\ref{eq_cosine_transform}) decays very slowly (see Proposition \ref{prop_Laplace_uq_vq}), 
and even though we have increased $b$ and $N_x$, we still  get only five significant digits.

For our second numerical experiment we will consider a process with jumps of rational transform, which are not completely monotone. 
We define the density of the L\'evy measure as follows
\begin{equation}
 \pi(x) = {\mathbf 1}_{\{x<0\}} \E^{-x} (1+\cos(ax)).
\end{equation}
Note that this is an example of a process which has jumps of rational transform but whose distribution does not belong to the phase-type class. Recall that any distribution from the phase-type class can be characterized as the time to absorption  in a finite state Markov chain with one absorbing state when the chain is started from a particular state.
One can easily see that $\pi(-x)$ can not be the density of such an absorption time, since it is equal to zero for $x=(2k+1)\pi/(a)$. 
The Laplace exponent of this process can be computed using formula (\ref{rational_Laplace_exponent})
\begin{equation*}
 \psi(z)=\frac12 \sigma^2 z^2 + \mu z + \frac{1}{z+1} + \frac{z+1}{(z+1)^2+a^2}-1 - \frac{1}{1+a^2}.
\end{equation*}
We fix the parameters at
\begin{equation*}
\sigma=0.25, \; \mu=2, \; a=4.
\end{equation*}
For this set of parameters we have $\Phi(q) \sim 0.37519$ (recall that we have fixed $q=0.5$ throughout this section) and all the roots of $\psi(-z)=q$ are simple
and are located at 
\begin{equation*}
 \zeta_{1,2} \sim 0.97265 \pm 4.0518\i, \; \zeta_3 \sim 0.64448, \;\;\; \zeta_4 \sim 64.7854
\end{equation*}

\begin{table}[!ht]
\centering
\bigskip
\begin{tabular}{|c||c|c|c|c|c|}
\hline \rule[0pt]{0pt}{11pt}  
   \;\; algorithm \;\; & \;\;\; rel. error  \;\;\;  & \;\; time (seconds) \;\; & \;\;\; $N_x$ \;\;\; & \;\;\; $M$ \;\;\;  & \;\; system precision \;\;  \\ [0.5ex] \hline \hline
   \rule[0pt]{0pt}{11pt}    Filon ($b=10^7$)  &  1.2e-11  & 0.14 & 10 000 & -- & double (64 bits)     \\ [0.5ex]\hline
   \hline 
   \rule[0pt]{0pt}{11pt}    Gaver-Stehfest \; &  1.2e-5  & 0.06   & 100  & 20 & 44 (MPFUN90)     \\ [0.5ex]\hline
   \rule[0pt]{0pt}{11pt}    Euler  	      &  9.1e-14 & 0.048  & 100  & 20 & quad (128 bits)     \\ [0.5ex]\hline
   \rule[0pt]{0pt}{11pt}    Talbot  	      &  9.2e-5  & 0.022   & 100  & 20 & quad (128 bits)     \\ [0.5ex]\hline
   \hline 
   \rule[0pt]{0pt}{11pt}    Gaver-Stehfest \; &  5.5e-10  & 0.23 & 100  & 40 & 88 (MPFUN90)     \\ [0.5ex]\hline
   \rule[0pt]{0pt}{11pt}    Euler  	      &  2.6e-25  & 0.44 & 100  & 40 & 40 (MPFUN90)    \\ [0.5ex]\hline
   \rule[0pt]{0pt}{11pt}    Talbot  	      &  2.8e-13  & 0.22 & 100  & 40 & 40 (MPFUN90)     \\ [0.5ex]\hline
   \hline 
   \rule[0pt]{0pt}{11pt}    Gaver-Stehfest \; &  1.5e-27  & 1.2   & 100  & 80 & 176 (MPFUN90)     \\ [0.5ex]\hline
   \rule[0pt]{0pt}{11pt}    Euler 	      &  2.0e-48  & 1.2   & 100  & 80 & 80 (MPFUN90)    \\ [0.5ex]\hline
   \rule[0pt]{0pt}{11pt}    Talbot  	      &  2.9e-49  & 0.6   & 100  & 80 & 80 (MPFUN90)     \\ [0.5ex]\hline
\end{tabular}
\vspace{0.2cm}
\caption{Computing $W^{(q)}(x)$ for the process with jumps of rational transform.}\label{tab_rational}
\end{table}

Once we have these roots 
 the potential density $\hat u^{(q)}(x)$ can be computed from  (\ref{uqx_rational_multiplicity1}) and the scale function
using relation (\ref{eqn_Wq_uq}). The parameters for Filon's method are the same as in our previous experiment with $\theta$-process. The results are presented in 
Table \ref{tab_rational} and are seen to be  similar to those of the  $\theta$-process, with the exception  that the Talbot method is not 
performing as well for small values of $M$. This is most likely due to the fact that function $F^{(q)}(z)$ has non-real singularities in the half-plane $\re(z)<0$
 and,  as  discussed in Section \ref{sec_Talbot}, this may cause problems for the Talbot method. 
However, when $M=80$, the Talbot method is performing very well just as it did for the case of the  $\theta$-process.

For our final numerical experiment, we consider a spectrally-negative L\'evy process with the L\'evy measure  defined by
\begin{eqnarray}\label{def_Pi_example}
 \Pi(\d x)&=&c_1 {\mathbf 1}_{\{x<0\}}\frac{\E^{\lambda_1 x}}{|x|^{1+\alpha_1}}\d x+
	c_2  {\mathbf 1}_{\{x<0\}}\frac{1}{|x|^{1+\alpha_2}}\d x \\ \nonumber
  &+& c_3 \delta_{-a_3}(\d x)+ c_4   {\mathbf 1}_{\{x<-a_4\}} \E^{x} \d x.
\end{eqnarray}
Here the admissible range of parameters is 
$c_i \ge 0$, $\lambda_i>0$, $\alpha_1 \in (-\infty,2)\setminus \{ 0, 1\}$, $\alpha_2 \in (0,2)\setminus \{ 1\}$ and $a_i>0$. 
Using the L\'evy-Khintchine formula and standard results on tempered stable processes (see proposition 4.2 in \cite{Cont}) we find that the Laplace exponent of $X$ can be computed as follows  
\begin{eqnarray}\label{def_psi_example}
 \psi(z)&=& \frac{1}{2}\sigma^2 z^2 + \mu z + c_1 \Gamma(-\alpha_1) \left[ (z+\lambda_1)^{\alpha_1}-\lambda_1^{\alpha_1} - z \alpha_1 \lambda_1^{\alpha_1-1} \right] \\
 \nonumber 
 &+&  c_2 \Gamma(-\alpha_2) z^{\alpha_2} + c_3 \left[ \E^{-a_3 z}-1 \right] + c_4 \left[ \frac{\E^{-a_4 z}}{z+1} -1 \right].
\end{eqnarray}

\begin{table}[!ht]
\centering
\bigskip
\begin{tabular}{|c||c|c|c|c|c|c|c|c|c|c|c|}
\hline \rule[0pt]{0pt}{11pt}  
    & \; $\sigma$ \; & \; $\mu$\; & \;$c_1$ \;  & \; $\lambda_1$ \; & \; $\alpha_1$ \; & \; $c_2$ \;  & \; $\alpha_2$ \; & \; $c_3$ \; & \; $a_3$ \; & \; $c_4$ \; & \; $a_4$ \;   \\ [0.5ex] \hline \hline
   \rule[0pt]{0pt}{11pt}   
    Set 1 \;&  0 & 2 & 1 & 1 & 0.5 & 1 & 1.5 & 0 & -- & 0 & --     \\ [0.5ex]\hline
   \rule[0pt]{0pt}{11pt}   
    Set 2 \;&  0 & 2 & \;0.05\; & 1 & -5 & \;0.25\; & 1.5 & 0 & -- & 0 & --     \\ [0.5ex]\hline
   \rule[0pt]{0pt}{11pt}   
    Set 3 \;&  \;0.25\; & 2 & 1 & 1 & 0.5 & 0 & -- & 0 & -- & 1 & 1     \\ [0.5ex]\hline
   \rule[0pt]{0pt}{11pt}   
    Set 4 \;&  0 & 2 & 1 & 1 & 0.5 & 0 & -- & 1 & 1 & 0 & --     \\ [0.5ex]\hline
\end{tabular}
\vspace{0.2cm}
\caption{Parameter sets.}\label{tab_parameter_sets}
\end{table}

We see that the L\'evy measure (\ref{def_Pi_example}) is a mixture of (a) the L\'evy measure of a tempered stable process, 
(b) the L\'evy measure of a stable process, 
(c) an atom at $-a_3$ and (d) an exponential jump distribution shifted by $-a_4$.  Table \ref{tab_parameter_sets}  specifies parameter choices for four different sets of L\'evy processes
 which capture the following characteristics.
\begin{description}[Set 1:]
 \item[Set 1:] Jumps of infinite variation with a completely monotone L\'evy density; no Gaussian component.
 \item[Set 2:] Jumps of infinite variation which have smooth, but not completely monotone L\'evy density; no Gaussian component.
 \item[Set 3:] Jumps of infinite activity, finite variation and discontinuous L\'evy density; non-zero Gaussian component.
 \item[Set 4:]  Jumps of infinite  activity, finite variation, the L\'evy measure has an atom; no Gaussian component.
\end{description}

In order to compare the performance of the algorithms we need to have a benchmark. In this experiment we don't have any explicit formula for
the scale function, thus we have to compute the benchmark numerically. We have chosen to use Filon's method for this purpose 
as it is quite robust. For each case we have tried different parameters to ensure that we have at least 10-11 digits of accuracy. For example, for the parameter
Set 4 we fix $b=10^8$, $c=0.1$, divide the interval of integration
in (\ref{subdividing_domain}) into 
1000 subintervals and set $b_j=b (j/n)^2$ for $j=1,2,\dots,n$ and $n=1000$. Finally we fix the number of discretization points for each subinterval at 
$N=N_x=5 \times 10^6$. 

As in our previous numerical experiment, we compute the scale function for one hundred values of $x$, given by $x_i=i/20$, $i=1,2,\dots,100$. The relative
errors presented below are maximum relative errors over this range of $x$-values.

\begin{table}[!ht]
\centering
\bigskip
\begin{tabular}{|c||c|c|c|c|c|}
\hline \rule[0pt]{0pt}{11pt}  
   \;\; algorithm \;\; & \;\;\; rel. error  \;\;\;  & \;\; time (seconds) \;\; & \;\;\; $N_x$ \;\;\; & \;\;\; $M$ \;\;\;  & \;\; system precision \;\;  \\ [0.5ex] \hline \hline
   \rule[0pt]{0pt}{11pt}    Filon ($b=10^7$)  &  2.6e-10  & 0.46 & 10 000 & -- & double (64 bits)     \\ [0.5ex]\hline
   \hline 
   \rule[0pt]{0pt}{11pt}    Gaver-Stehfest \; &  1.2e-12  & 1.1 & 100 & 20 & 44 (MPFUN90)     \\ [0.5ex]\hline
   \rule[0pt]{0pt}{11pt}    Euler  	      &  1.3e-12  & 0.05 & 100  & 20 & quad (128 bits)     \\ [0.5ex]\hline
   \rule[0pt]{0pt}{11pt}    Talbot  	      &  7.3e-13  & 0.024 & 100  & 20 & quad (128 bits)     \\ [0.5ex]\hline
\end{tabular}
\vspace{0.2cm}
\caption{Computing $W^{(q)}(x)$ for the process with parameter Set 1.}\label{tab_compl_monotone}
\end{table}

First we compute the scale function for the parameter Set 1 and present the results in Table \ref{tab_compl_monotone}. In this case we have a process
with completely monotone jumps, and we see that the situation is exactly the same as in our earlier experiment with $\theta$-process. 
All four methods provide good accuracy, although the Talbot algorithm is faster and gives better accuracy.

\begin{table}[!ht]
\centering
\bigskip
\begin{tabular}{|c||c|c|c|c|c|}
\hline \rule[0pt]{0pt}{11pt}  
   \;\; algorithm \;\; & \;\;\; rel. error  \;\;\;  & \;\; time (seconds) \;\; & \;\;\; $N_x$ \;\;\; & \;\;\; $M$ \;\;\;  & \;\; system precision \;\;  \\ [0.5ex] \hline \hline
   \rule[0pt]{0pt}{11pt}    Filon ($b=10^7$)  &  8.4e-10  & 0.47 & 10 000 & -- & double (64 bits)     \\ [0.5ex]\hline
   \hline 
   \rule[0pt]{0pt}{11pt}    Gaver-Stehfest \; &  3.9e-12  & 1.1 & 100 & 20 & 44 (MPFUN90)     \\ [0.5ex]\hline
   \rule[0pt]{0pt}{11pt}    Euler 	      &  3.9e-12  & 0.05 & 100  & 20 & quad (128 bits)     \\ [0.5ex]\hline
   \rule[0pt]{0pt}{11pt}    Talbot  	      &  3.6e-12  & 0.024 & 100  & 20 & quad (128 bits)     \\ [0.5ex]\hline
\end{tabular}
\vspace{0.2cm}
\caption{Computing $W^{(q)}(x)$ for the process with parameter Set 2.}\label{tab_smooth}
\end{table}

In Table \ref{tab_smooth} we show the results for the parameter Set 2. In this case the jumps have smooth, but not completely monotone density. However, this 
does not seem to affect the results, and the performance of all four algorithms is similar to the case of parameter Set 1.

\begin{table}[!ht]
\centering
\bigskip
\begin{tabular}{|c||c|c|c|c|c|}
\hline \rule[0pt]{0pt}{11pt}  
   \;\; algorithm \;\; & \;\;\; rel. error  \;\;\;  & \;\; time (seconds) \;\; & \;\;\; $N_x$ \;\;\; & \;\;\; $M$ \;\;\;  & \;\; system precision \;\;  \\ [0.5ex] \hline \hline
   \rule[0pt]{0pt}{11pt}    Filon ($b=10^7$)  &  4.6e-11  & 0.46 & 10 000 & -- & double (64 bits)     \\ [0.5ex]\hline
   \hline 
   \rule[0pt]{0pt}{11pt}    Gaver-Stehfest \; &  3.3e-5  & 1.1   & 100  & 20 & 44 (MPFUN90)     \\ [0.5ex]\hline
   \rule[0pt]{0pt}{11pt}    Euler  	      &  1.4e-6  & 0.06  & 100  & 20 & quad (128 bits)     \\ [0.5ex]\hline
   \rule[0pt]{0pt}{11pt}    Talbot  	      &  5.0e-3  & 0.028 & 100  & 20 & quad (128 bits)     \\ [0.5ex]\hline
   \hline 
   \rule[0pt]{0pt}{11pt}    Gaver-Stehfest \; &  7.0e-6  & 4.5 & 100  & 40 & 88 (MPFUN90)     \\ [0.5ex]\hline
   \rule[0pt]{0pt}{11pt}    Euler  	      &  9.1e-7  & 4.7 & 100  & 40 & 40 (MPFUN90)    \\ [0.5ex]\hline
   \rule[0pt]{0pt}{11pt}    Talbot  	      &  1.7e-3  & 2.3 & 100  & 40 & 40 (MPFUN90)     \\ [0.5ex]\hline
   \hline 
   \rule[0pt]{0pt}{11pt}    Gaver-Stehfest \; &  1.1e-6  & 22  & 100  & 80 & 176 (MPFUN90)     \\ [0.5ex]\hline
   \rule[0pt]{0pt}{11pt}    Euler 	      &  3.0e-6  & 18  & 100  & 80 & 80 (MPFUN90)    \\ [0.5ex]\hline
   \rule[0pt]{0pt}{11pt}    Talbot  	      &  2.9e-4  & 8.8 & 100  & 80 & 80 (MPFUN90)     \\ [0.5ex]\hline
\end{tabular}
\vspace{0.2cm}
\caption{Computing $W^{(q)}(x)$ for the process with parameter Set 3.}\label{tab_C3}
\end{table}

In Table \ref{tab_C3} we see the first example where the jump density (and therefore the scale function) is non-smooth. 
In this case we have $W^{(q)}(x) \in C^3(0,\infty)$ (see Theorem \ref{II}), and we would expect that the discontinuity in the fourth derivative of 
$W^{(q)}(x)$ shouldn't affect 
the performance of our numerical methods. Unfortunately, this is not the case. We see that,  while Filon's method is still performing well, 
the Talbot algorithm provides only 2-3 digits of precision. The Gaver-Stehfest and Euler algorithms do a slightly better job and produce 5-6 digits of precision. 
However, when we increase $M$ we do not see a significant decrease in the error, and, in fact, it is not clear whether these three methods 
would converge to the right values at all. This behaviour, specifically problems with Laplace inversion when the target function is non-smooth, is quite well-known. See for example the discussion in \cite{AbateChWhitt1999} and Section 14 of \cite{AbateWhitt1992}. 

\begin{table}[!ht]
\centering
\bigskip
\begin{tabular}{|c||c|c|c|c|c|}
\hline \rule[0pt]{0pt}{11pt}  
   \;\; algorithm \;\; & \;\;\; rel. error  \;\;\;  & \;\; time (seconds) \;\; & \;\;\; $N_x$ \;\;\; & \;\;\; $M$ \;\;\;  & \;\; system precision \;\;  \\ [0.5ex] \hline \hline
   \rule[0pt]{0pt}{11pt}    Filon ($b=10^7$)  & \; 5.3e-8 (4.6e-11) \;  & 5.1 & $10^5$ & -- & double (64 bits)     \\ [0.5ex]\hline
   \hline 
   \rule[0pt]{0pt}{11pt}    Gaver-Stehfest \; &  1.6e-3  & 1.1   & 100  & 20 & 44 (MPFUN90)     \\ [0.5ex]\hline
   \rule[0pt]{0pt}{11pt}    Euler  	      &  3.2e-4  & 0.05  & 100  & 20 & quad (128 bits)     \\ [0.5ex]\hline
   \rule[0pt]{0pt}{11pt}    Talbot  	      &  4.6e-3  & 0.024 & 100  & 20 & quad (128 bits)     \\ [0.5ex]\hline
   \hline 
   \rule[0pt]{0pt}{11pt}    Gaver-Stehfest \; &  8.0e-4  & 4.5 & 100  & 40 & 88 (MPFUN90)     \\ [0.5ex]\hline
   \rule[0pt]{0pt}{11pt}    Euler  	      &  3.2e-4  & 4.6 & 100  & 40 & 40 (MPFUN90)    \\ [0.5ex]\hline
   \rule[0pt]{0pt}{11pt}    Talbot  	      &  2.4e-3  & 2.2 & 100  & 40 & 40 (MPFUN90)     \\ [0.5ex]\hline
   \hline 
   \rule[0pt]{0pt}{11pt}    Gaver-Stehfest \; &  3.9e-4  & 22  & 100  & 80 & 176 (MPFUN90)     \\ [0.5ex]\hline
   \rule[0pt]{0pt}{11pt}    Euler 	      &  8.1e-3  & 18  & 100  & 80 & 80 (MPFUN90)    \\ [0.5ex]\hline
   \rule[0pt]{0pt}{11pt}    Talbot  	      &  1.2e-3  & 8.7 & 100  & 80 & 80 (MPFUN90)     \\ [0.5ex]\hline
\end{tabular}
\vspace{0.2cm}
\caption{Computing $W^{(q)}(x)$ for the process with parameter Set 4.}\label{tab_non_smooth}
\end{table}

Finally, in Table \ref{tab_non_smooth} we present the most extreme case of parameter Set 4. 
In this case the scale function is not even continuously
differentiable, so we should expect to see some interesting behavior of our numerical algorithms. We see that Filon's method is still performing very well. Indeed, the maximum relative error is around 5.3e-8. The error achieves its maximum at $x=1$, which is  not surprising since this is the point where
$W^{(q)}(x)$ is not differentiable (see figure \ref{fig_Set4}). If we remove this point, then the maximum relative error drops down to 4.6e-11, 
which is an extremely 
good degree of accuracy for such non-smooth function. By contrast, the other three algorithms
are all struggling to produce even four digits of precision. Again, it is not clear whether these algorithms converge 
as $M$ increases. In particular we see that the error of the Euler method appears to increase. 

We obtain an overview of what is happening in this case by inspecting the graph of $W^{(q)}{}'(x)$ (see Figure \ref{fig_Set4}). 
We see that $W^{(q)}{}'(x)$ has a jump at $x=1$, which is the 
point where the L\'evy measure has an atom. This behavior is expected due to Corollary \ref{corr-for-alexey}. 
Then we observe that while $W^{(q)}{}'(x)$ is continuous at all points $x\ne 1$, the second derivative
$W^{(q)}{}''(x)$ has a jump at $x=2$. Although there is no theoretical justification for the appearance of such a jump in this text, this observation  is consistent with the results on discontinuities in higher derivatives of  potential measures of 
 subordinators obtained by D\"{o}ring and Savov \cite{Doring_Savov}. 
 
\begin{figure}
\centering
\subfloat[][ $W^{(q)}{}'(x)$]{\label{fig_set1_2d}\includegraphics[height =4cm]{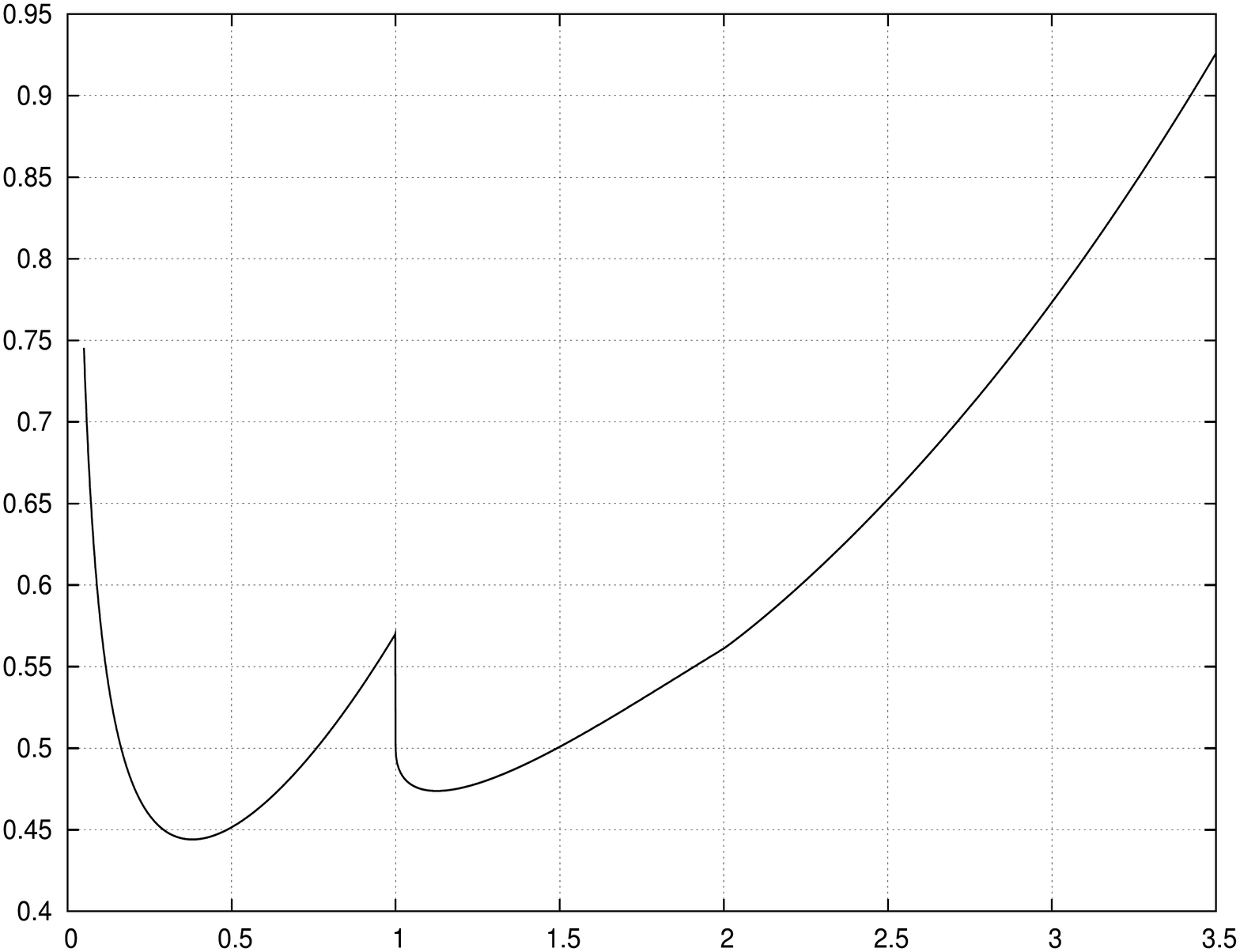}} 
\subfloat[][ $\log_{10}($relative error) with $M=20$]{\label{fig_set1_1d}\includegraphics[height =4cm]{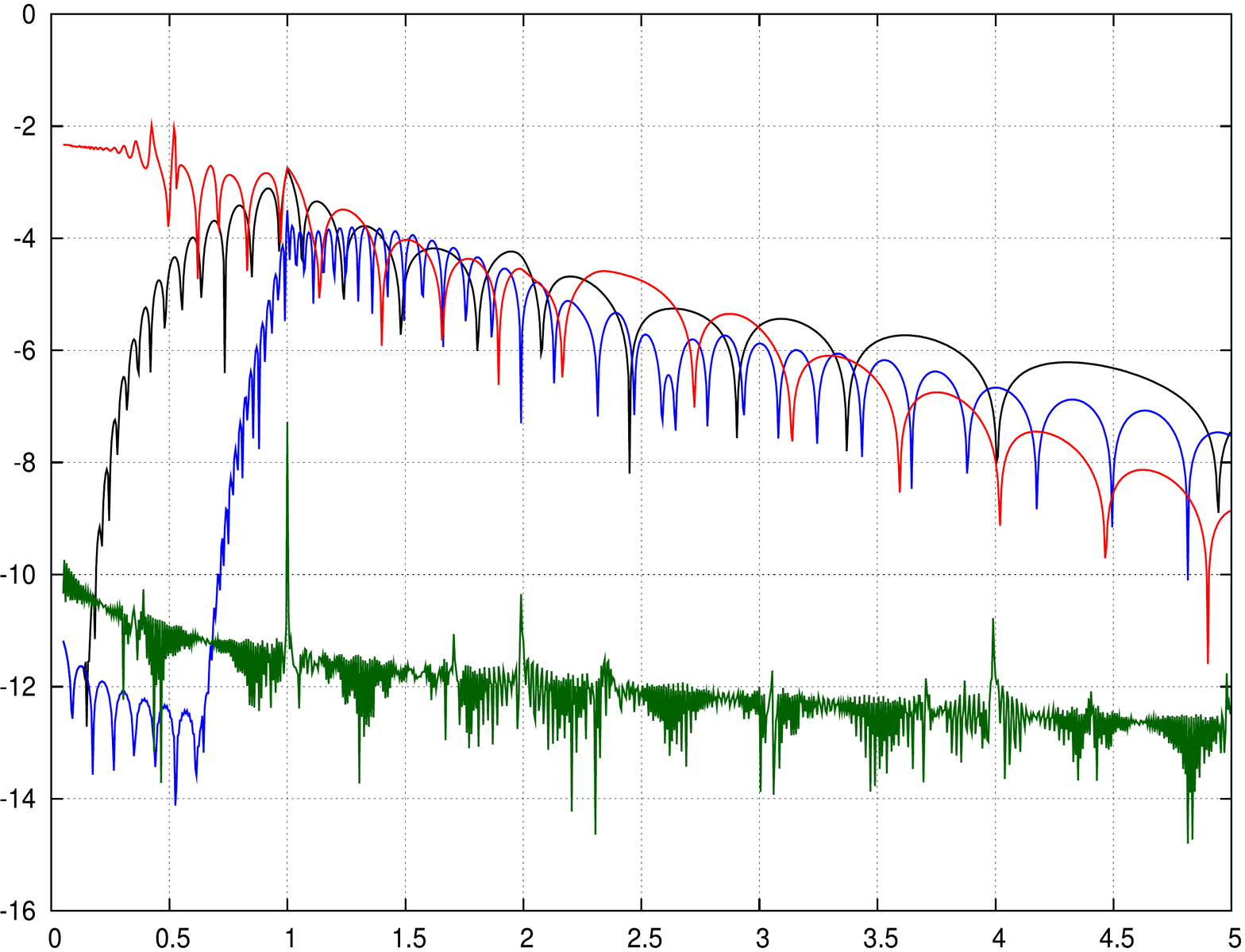}} \\
\caption{Computing $W^{(q)}(x)$ for the process with parameter Set 4. The relative errors produced by the Filon, Gaver-Stehfest, Euler and Talbot methods correspond to the green, black, blue and red curves respectively.} 
\label{fig_Set4}
\end{figure}

\section{Conclusion}\label{sec_conclusion}

In the previous section we have presented results of several numerical experiments, which allow us to compare the performance of four 
different methods for computing the scale function. 

Filon's method has the  advantage of being the most robust. We are guaranteed to have an approximation 
which is within $O(N^{-3})$ of the exact value, provided that we have chosen the cutoff $b$ to be large enough. 
Another advantage of this method is that it gives not just a single value of $W^{(q)}(x)$, but an array of $N$ values $W^{(q)}(x_i)$ (where
the $x_i$ are equally spaced)  at the computational cost of just $O(N \ln(N))$ operations. 
This means that by increasing $N$ we get a better accuracy and, at the same time, more values of $W^{(q)}(x)$, and the computational time would increase  
almost linearly in $N$. This is different from the other three methods, where the computational time needed to produce $N$ values of $W^{(q)}(x)$ is $O(NM)$. 

At the same time, there are several disadvantages relating to Filon's method. First of all, this algorithm requires quite some effort to  program. 
Secondly,  when the integrand decreases very slowly, it may be necessary to take  an extremely large value of the cutoff value  $b$. This would have to be compensated by an increase in the number of discretization points $N$ (see Table \ref{tab_theta2}). Finally, unlike the other three methods, which depend on a single parameter $M$, 
Filon's method has two important parameters, $b$ and $N$ (as well as $b_j$ and $n$ if one uses (\ref{subdividing_domain})), and it is not a priori 
clear how to choose these parameters. 
Usually, reasonable values of $b$ and $N$ can be found after some experimentation, 
which is of course time consuming.

The essential property of the other three methods, i.e the Gaver-Stehfest, Euler and Talbot algorithms, is that all of them require multi-precision arithmetic. 
 If the scale function is smooth and we need just 10-12 digits of precision, then we can use the quad (128 bits) precision, 
which is available in some programming languages, such as Fortran90. 
In all other cases one has to use specialized software which is capable of working with high-precision numbers. Whilst  very efficient multi-precision arithmetic libraries do exist, such as MPFUN90 by D.H. Bailey \cite{BaileyMPFUN90}, they are usually much slower than the native double/quad precision of any programming language.

An advantage of the Gaver-Stehfest/Euler/Talbot methods is their simplicity. 
These algorithms are very easy to program, and they all depend only on just a single integer parameter $M$. 

The performance of the the Gaver-Stehfest, Euler and Talbot algorithms depends a lot on the smoothness of the scale function $W^{(q)}(x)$. 
When $W^{(q)}(x)$ is not sufficiently smooth it might be better to use Filon's method.
 When the process has completely monotone jumps,  Talbot's algorithm is arguably preferable to the other methods as it produces excellent accuracy. Moreover,  it is almost twice as fast as  its nearest competitor, the Euler algorithm. When the process has a 
smooth, but not completely monotone, jump density the Euler algorithm is recommendable. One can still use the Talbot algorithm in this case, 
however, this should be done with a caution as it is not guaranteed to work (see the discussion in Section \ref{sec_Talbot} and the results of 
our second numerical experiment). 

Finally, the Gaver-Stehfest algorithm is somewhat special. It is the only algorithm which does not require the evaluation of  $\psi(z)$ at complex values of $z$. 
Our results show that this method is not the fastest and not the most precise, but it can still be useful in some situations. 
For example, 
this method might be the only available option when the Laplace exponent $\psi(z)$ is given by a  
 very complicated formula which is valid for real values of $z$ and, moreover, it is not clear how to analytically continue it into the complex half-plane $\re(z)>0$. 
The Gaver-Stehfest algorithm can also be helpful 
when one is using  multi-precision software which does not support operations on complex numbers.


\newpage

%



\endgroup
\end{document}